\documentclass[11pt,a4paper]{amsart}
\usepackage{a4wide}
\usepackage{amsfonts,amsthm,amsmath,amssymb}
\usepackage{amscd,color,caption,multirow}
\usepackage{psfrag,graphicx,subfigure,float,subfloat}
\usepackage{import,algorithmic,algorithm2e}

\newtheorem*{thmA}{Theorem A}
\newtheorem*{thmB}{Theorem B}

\newtheorem*{corC}{Corollary C}

\theoremstyle{plain}

\newtheorem{theorem}{Theorem}[section]
\newtheorem{lemma}[theorem]{Lemma}
\newtheorem{corollary}[theorem]{Corollary}
\newtheorem{prop}[theorem]{Proposition}

\newtheorem{remark}{Remark}

\renewcommand{\b}{\color{blue}}
\newcommand{\w}{\color{white}}
\newsavebox{\savepar}

\begin{document}

\title[On the  basins of attraction for the secant method ]{Topological properties of the immediate basins of attraction for the secant method}

\date{\today}

\author{Laura Gardini}
\address{Department of Economics, Society and Politics, University of Urbino, Italy}
\email{laura.gardini@uniurb.it} 
\author{Antonio  Garijo}
\address{Departament d'Enginyeria Inform\`atica i Matem\`atiques,
Universitat Rovira i Virgili, 43007 Tarragona, Catalonia.}
\email{antonio.garijo@urv.cat}
\author{Xavier Jarque}
\address{Departament de Matem\`atiques i Inform\`atica at Universitat de Barelona and Barcelona Graduate School of Mathematics, 08007 Barcelona, Catalonia.}
\email{xavier.jarque@ub.edu}

\thanks{This work has been partially supported by MINECO-AEI grants
MTM-2017-86795-C3-2-P and MTM-2017-86795-C3-3-P and the AGAUR grant 2017 SGR 1374}\

\begin{abstract}
We study the discrete dynamical system defined on a subset of  $R^2$  given by the iterates of the secant method applied to a real polynomial $p$. Each simple real root $\alpha$ of $p$ has associated its basin of attraction $\mathcal A(\alpha)$ formed by the set of points converging towards the fixed point $(\alpha,\alpha)$ of $S$.
We denote by $\mathcal A^*(\alpha)$ its immediate basin of attraction, that is, the connected component of $\mathcal A(\alpha)$ which contains $(\alpha,\alpha)$. We focus on some topological properties of $\mathcal A^*(\alpha)$, when $\alpha$ is an internal real root of $p$. More precisely, we show the existence of a 4-cycle in $\partial \mathcal A^*(\alpha)$ and we give conditions on $p$ to guarantee the simple connectivity of $\mathcal A^*(\alpha)$.

\vspace{0.5cm}


\vglue 0.2truecm


\end{abstract}

\maketitle

\section{Introduction and statement of the results}\label{sec:Introduction}

Dynamical systems is a powerful tool in order to have a deep understanding on the global behavior of the so called {\it root-finding} algorithms, that is, iterative methods capable to numerically determine the solutions of the equation $f(x)=0$. In most cases, it is well known the order of  convergence of those methods near the zeros of $f$, but it is in general unclear the behavior and effectiveness when initial conditions are chosen on the whole space; a natural question when we do not know a priori where the roots are or if there are many of them. 

The numerical exploration of the solutions of the equation $f(x)=0$ has been always  central problem in many areas of applied mathematics; from biology to engineering, since most mathematical models requires to have a thorough knowledge of the solutions of certain equations. Once we are certain that no algebraic manipulation of the equation will allow to explicitly find out the  solutions, one can try to built numerical methods which will approximate the solutions with arbitrary precision. Perhaps the most well known and universal method is the {\it  Newton method} inspired on the linearization of the equation $f(x)=0$. But also other methods has shown to be certainly efficient like the {\it secant method}, the main object of the paper.

Roughly speaking, all these iterative methods give efficient ways to find the solutions of $f(x)=0$, at least once you have a good approximation of them. However, there is a significant amount  of uncertainty  when the initial conditions are freely chosen, i.e.  when there is not a {\it natural} candidate for the solution or the number of solutions is high. It is  in this context where dynamical systems might play a central role. As an example we can refer to \cite{HowToNewton} where the authors first prove theoretical results on the global dynamics of the Newton method and then apply them to create efficient algorithms to find out {\it all} solutions, even in the case that the degree of $p$ is huge. 

This paper is a step forward in this direction for the secant method. Remarkably, this method presents some advantages to Newton's method but the natural phase space of its associated iterative system is not 1-dimensional anymore, but 2-dimensional. Therefore its study requires new techniques and ideas like the ones presented in this paper. See also \cite{BedFri,Tangent}.

Let $p$ be a real polynomial of degree $k$ given by 
\begin{equation} \label{eq:p}
p(x)=a_0 + a_1 x + \ldots + a_k x^k \quad  \hbox{  with  }\quad a_k \neq 0.
\end{equation}
We assume that $p$ has exactly $n \in \{3,\ldots k\}$  simple real roots denoted by $ \alpha_0 < \alpha_1 < \ldots < \alpha_{n-2} < \alpha_{n-1} $. The roots  $\alpha_0$ and $\alpha_{n-1}$ are called the {\it external } roots of $p$, in contrast the rest of the roots  $\alpha_j$ for $ 1 \leq j \leq n-2$ are called the {\it internal} roots of $p$. 

We consider the {\it secant method} applied  to the polynomial $p$ as a discrete dynamical system acting on the real plane,
\begin{equation} \label{eq:secant_real}
S:=S_p: \mathbb R^2 \mapsto \mathbb R^2 , \qquad
S: \left( 
\begin{array}{l}
x \\
y
\end{array}
\right) \mapsto  
 \left( 
\begin{array}{l}
y \\
y - p(y) \frac{y-x}{p(y)-p(x)}
\end{array}
\right),
\end{equation}
and the orbit of the seed $\left(x_0,y_0\right)\in \mathbb R^2$  is given by the iterates of the map; that is, the sequence $\{S^m\left(x_0,y_0\right)\}_{m \geq 0} $. 
We refer to \cite{Tangent} for a detailed discussion of the two-dimensional dynamical system induced by $S$ and also some consequences as a root finding algorithm. Here we will always consider $S:\mathbb R^2 \to \mathbb R^2$, but  there is a natural extension of this problem  by assuming $p$ as a  polynomial with complex coefficients and thus  $S:\mathbb C^2 \to \mathbb C^2$. See {\cite{BedFri} for a discussion on this context.   

Any simple root $\alpha$ of $p$ corresponds to an attracting fixed point $(\alpha,\alpha)$ of the secant map $S$. Thus, we can consider the {\it basin of attraction} of $(\alpha,\alpha)$, denoted by $\mathcal A(\alpha)$, consisting of all points tending towards this fixed point, 
\begin{equation}\label{eq:basin}
\mathcal A(\alpha) = \{ (x,y) \in \mathbb R^2 \, ; \, S^m(x,y) \to (\alpha,\alpha) \hbox{  as } m  \to \infty \}.
\end{equation}     
\noindent It is easy to see that when $\alpha$ is a simple root of $p$ then the point $(\alpha,\alpha)$ belongs to ${\rm Int}\left(\mathcal A(\alpha)\right)$. However this is not always the case when $\alpha$ is a multiple root of $p$ (see \cite{TangentMultiple}). It is remarkable that even in the case of $\alpha$ being a simple root of $p$ the local dynamics around the point $(\alpha,\alpha)$ does not follow the typical behavior of an attracting fixed point of a diffeomorphism ({\it \`a la Hartman-Grobman}) due to the presence of infinitely many points, in any neighbourhood of the fixed point, which under one iteration land on the fixed point. 

We also denote by $\mathcal A^*(\alpha)$ the {\it immediate basin of attraction} of $(\alpha,\alpha)$, i.e., the maximal connected component of $\mathcal A(\alpha)$ containing $(\alpha,\alpha)$. Moreover, if $\alpha$ is an external root of $p$ then its immediate basin of attraction is an unbounded set while if $\alpha$ is an internal root then $\mathcal A^*(\alpha)$ is bounded (See \cite{Tangent}). This  property shows the first topological difference between the immediate basin of attraction of an external and an internal simple root.

\begin{figure}[ht]
    \centering
    \subfigure[\scriptsize{ $T_3(x)=4x^3-3x$.}]{
     \includegraphics[width=0.4\textwidth]{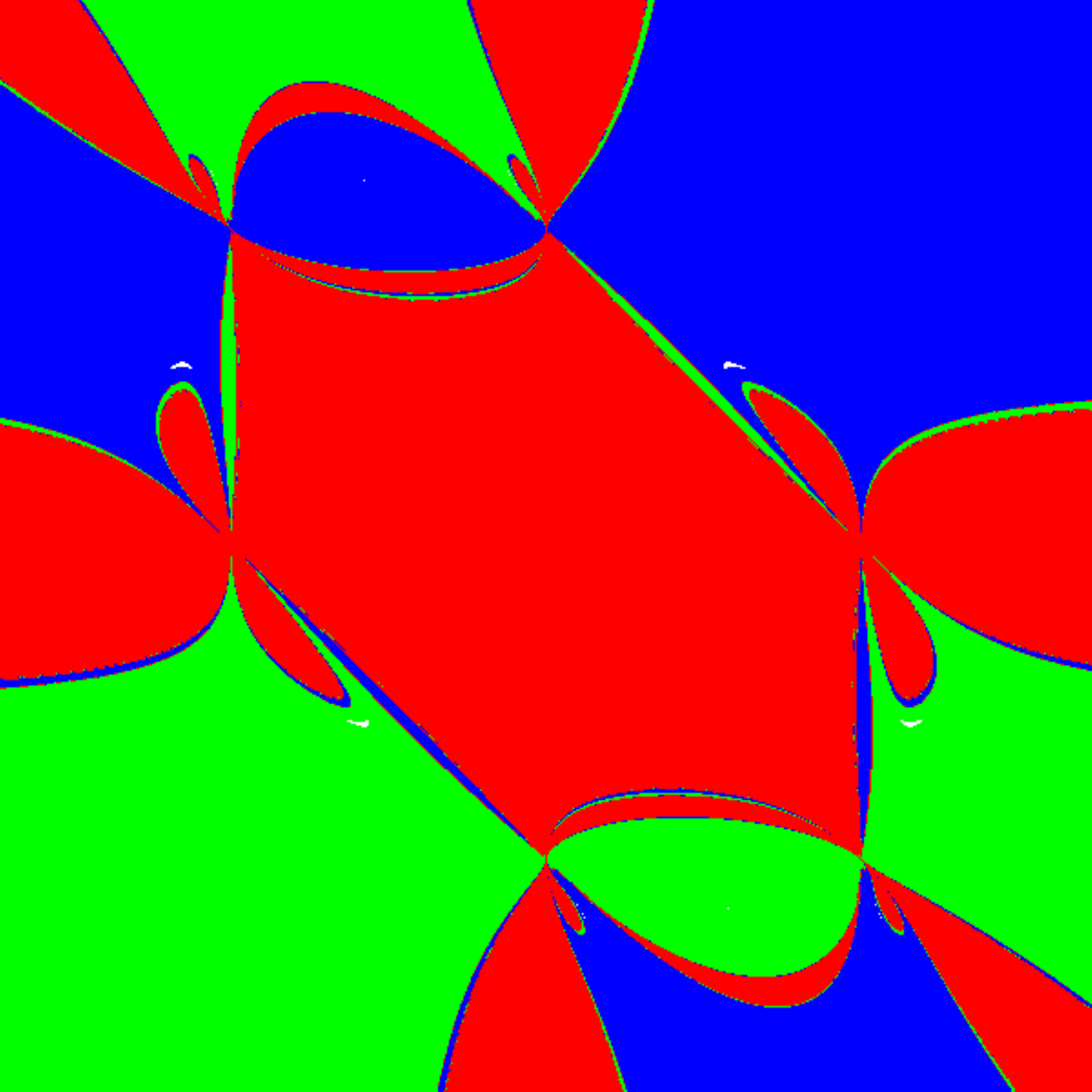}
      }
    \subfigure[\scriptsize{$T_4(x)=8x^4-8x^2+1$.}]{
     \includegraphics[width=0.4\textwidth]{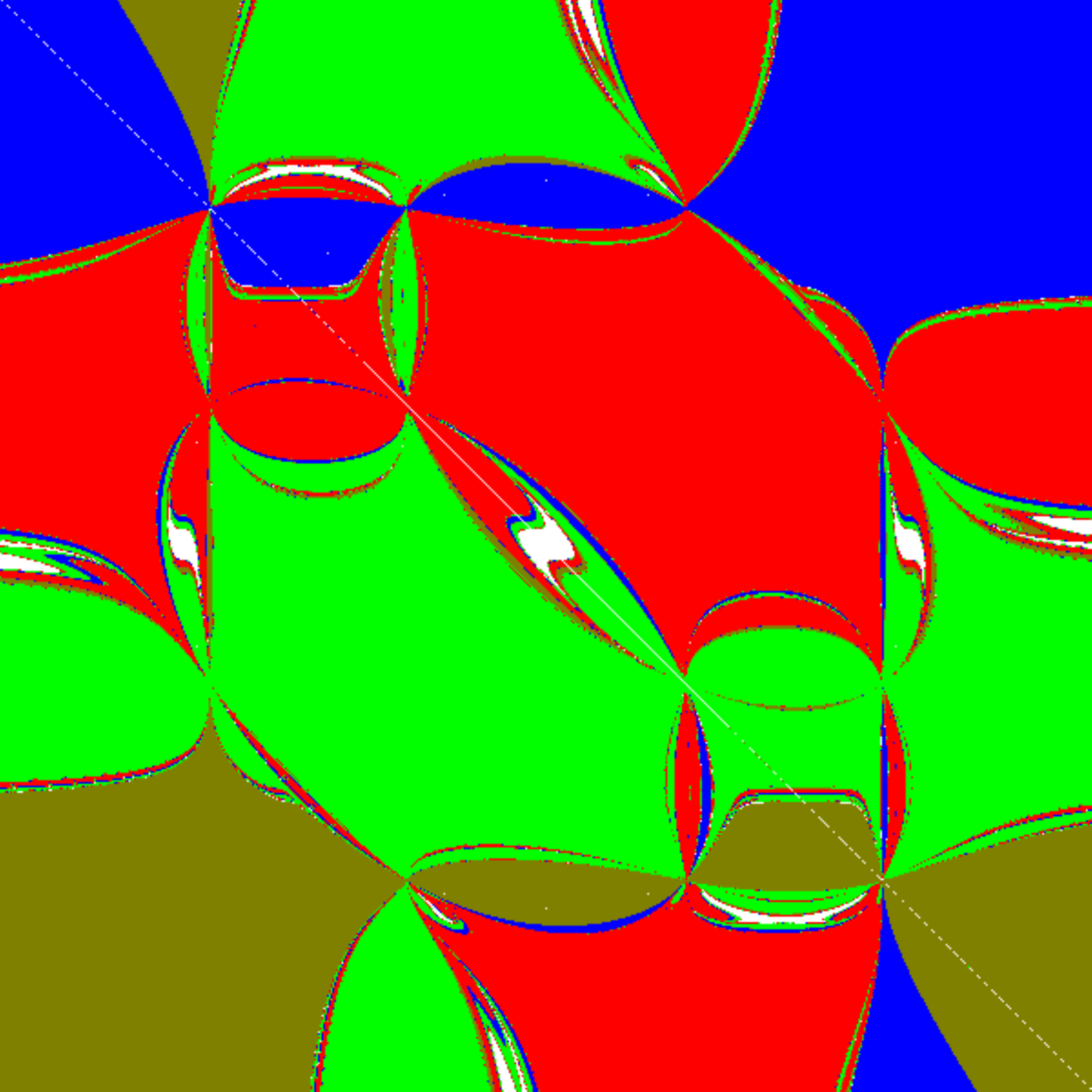}
      }   \\
       \subfigure[\scriptsize{$T_5(x)=16x^5-20x^3+5x$.}]{
     \includegraphics[width=0.4\textwidth]{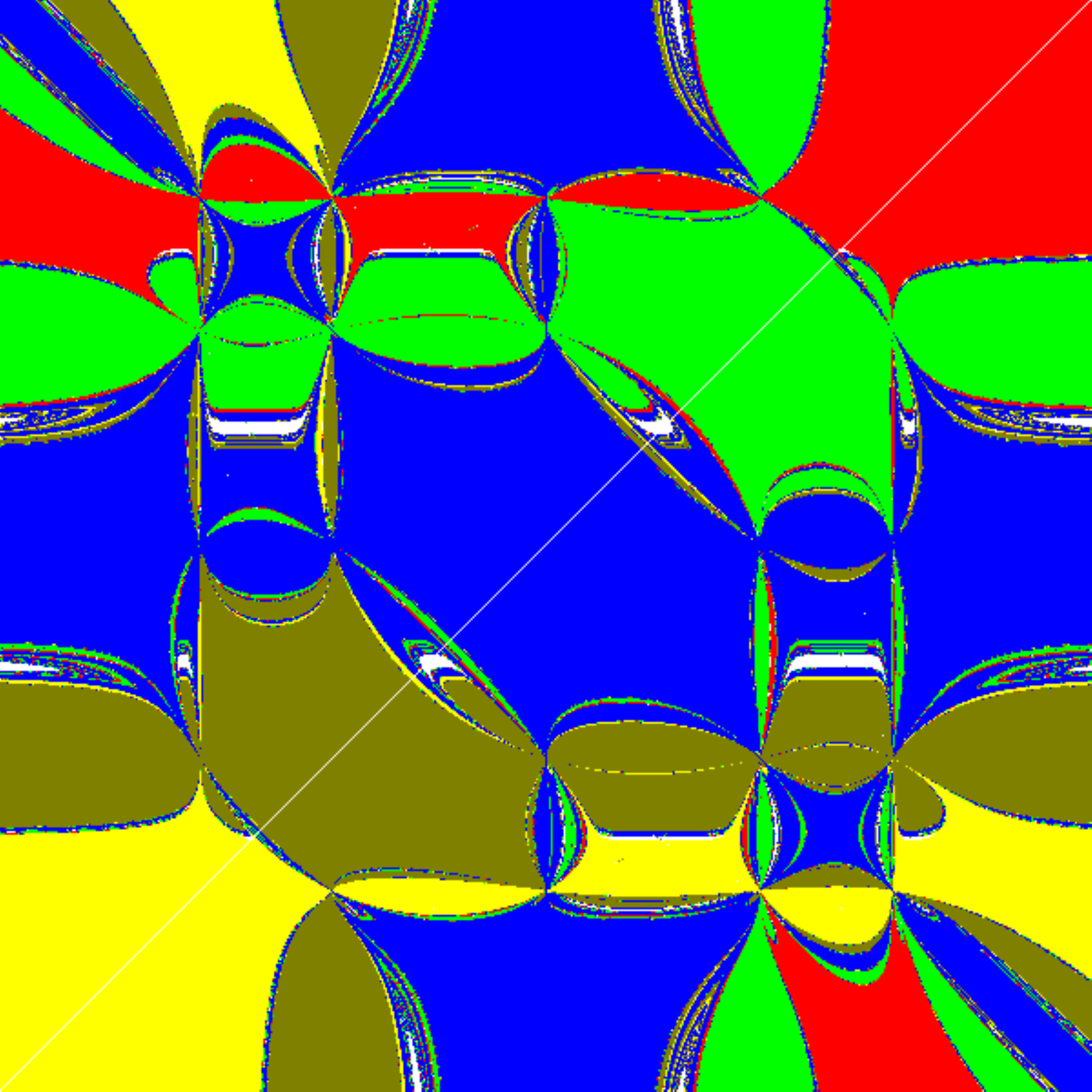}
      }      
    \subfigure[\scriptsize{$T_{11}(x)=1024x^{11}-2816x^9+2816x^7-1232x^5+220x^3-11x$.}]{
     \includegraphics[width=0.4\textwidth]{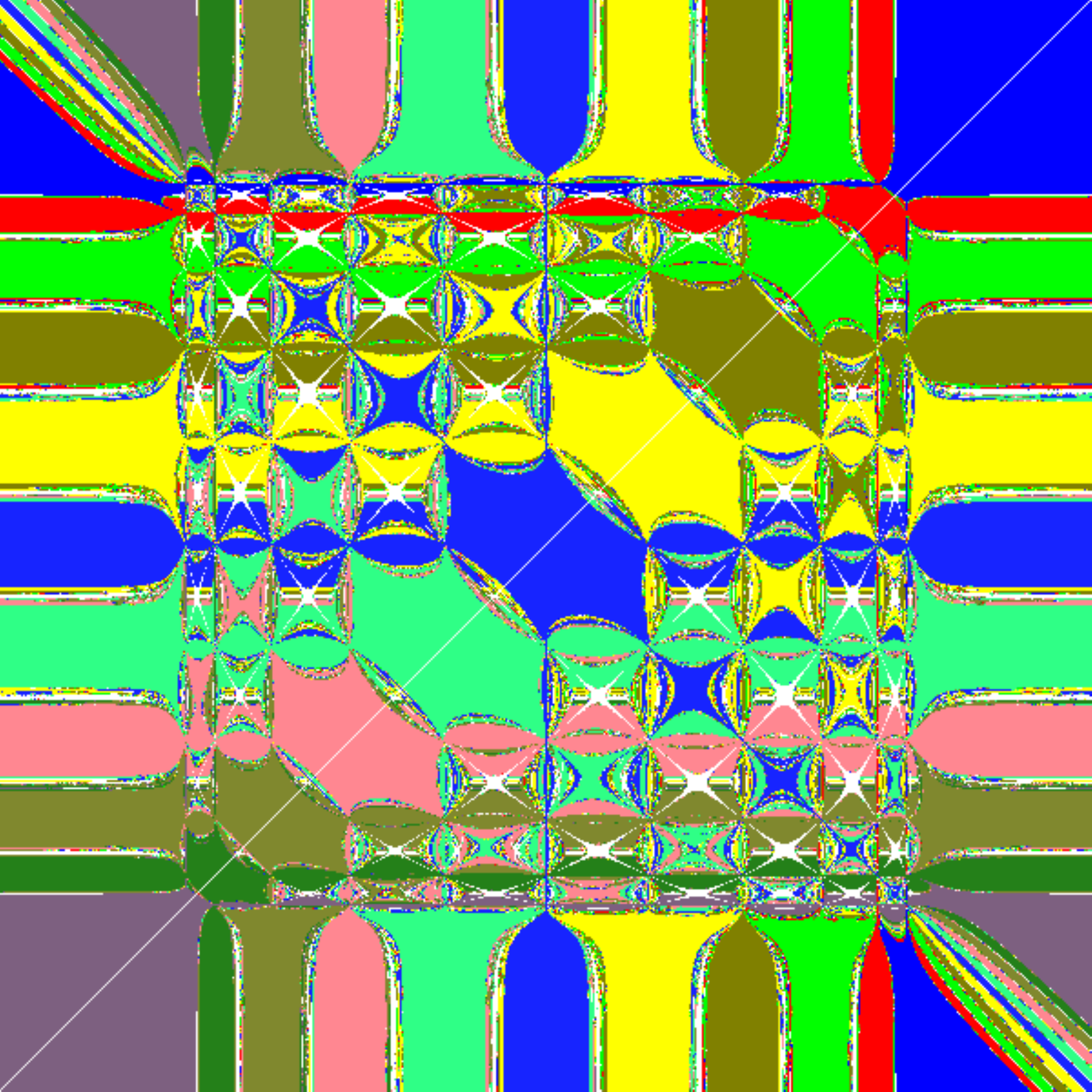}
}
      
 \caption{\small{  Phase plane of the secant map applied to the Chebyshev polynomials $T_k(x)$ for $k=3,4,5$ and 11. We show each basin of attraction with a different color. Range of the pictures  [-1.5,1.5]x[-1.5,1.5]. }}
    \label{fig:chebyshev}
    \end{figure}

Along the paper, as a toy model for numerical experiments, we take the family of  Chebychev polynomials $T_k(x)$ for $k \geq 0$. We recall that Chebyshev polynomials can be  defined by $T_0(x)=1$, $T_1(x)=x$ and recursively $T_{k+1}(x)=2xT_k(x)-T_{k-1}(x)$ for $k \geq 1$. Among other properties every polynomial  $T_k(x)$ has degree $k$ and exhibits $k$ simple real roots in the interval $(-1,1)$. Indeed  the roots of $T_k$ are located at points $x_j=\cos\left( \frac{\pi (j+1/2)}{n} \right)$, for $j = 0, \ldots, k-1$. 
In Figure \ref{fig:chebyshev} we show the phase plane of the Secant maps for  the polynomials $T_k$ for $k=3,4,5$ and $11$. The range of the picture is $[-1.5, 1.5]\times[-1.5,1.5]$ so the points $(\alpha,\alpha)$ are located at the diagonal of each picture. The topological structure of the immediate basin of attraction seems to  remain similar depending only on the character of the root (internal or external). In order to state the main results on this direction we first introduce some required notation. 


Let $T\subset \mathbb R^2$ be a bounded (infinite) graph formed by vertices and edges. We say that an edge of $T$ is a {\it lobe} if it connects a  vertex with itself. We  say that $T$ is a {\it smooth hexagon-like polygon with lobes} if it is formed by six vertices, six $C^1$-edges connecting those vertices and countably many $C^1$-lobes at some of the vertices.  See Figure \ref{fig:polygon}.


%

 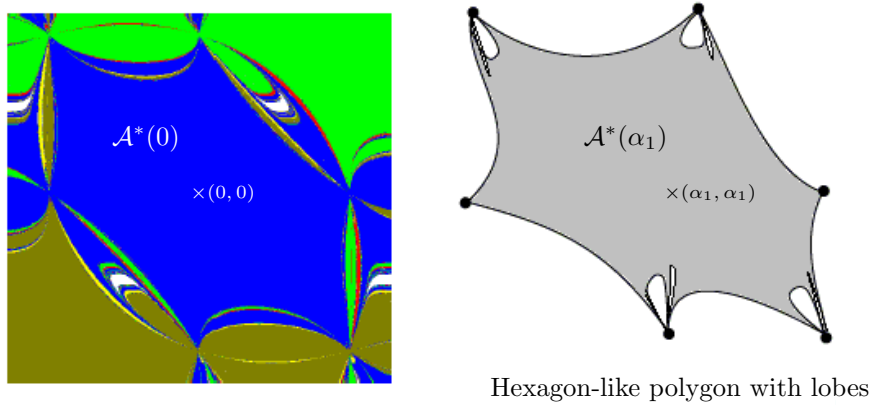
\begin{figure}[ht]
    \centering
     \includegraphics[width=0.75\textwidth]{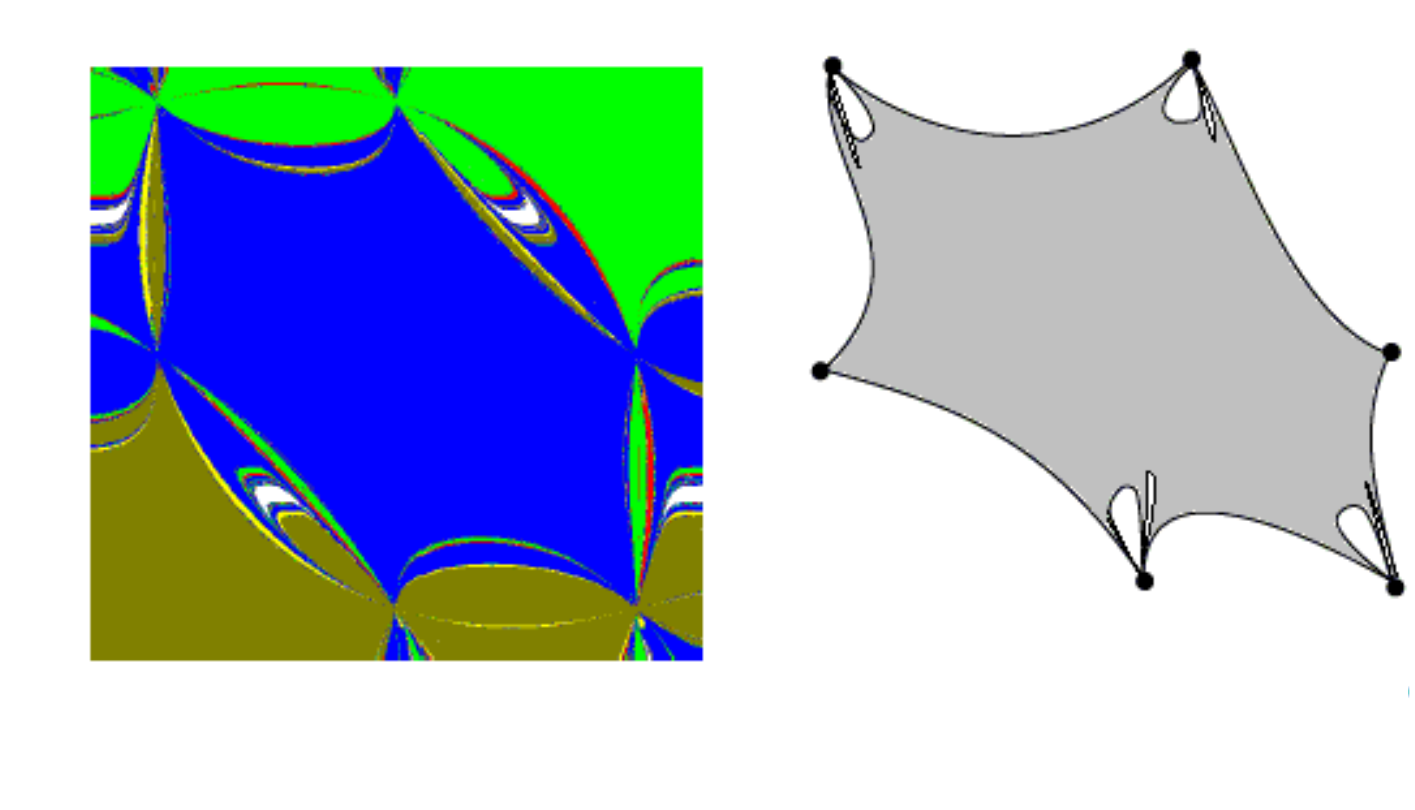}
    \put(-130,30) {\small Hexagon-like polygon with lobes} 
    \put(-245,105) {\tiny \w{ $ \times (0,0)$ }}
    \put(-275,125) {\small \w{ $ \mathcal A^*(0)$ }}
        \put(-65,105) {\tiny  $ \times (\alpha_1,\alpha_1)$ }
            \put(-95,125) {\small $ \mathcal A^*(\alpha_1)$ }
            \caption{\small{ On the left hand side we show the phase space of the secant map applied to the Chebyshev polynomial $T_5(x)=16x^5-20x^3 + 5x$ (see Figure
           \ref {fig:chebyshev}(c)), we show in blue  $\mathcal A^*(0)$.  Range of the phase plane  [-0.75,0.75]x[-0.75,0.75]. On the right hand side we sketch an hexagon-ike polygon with lobes which is  the topological model of the immediate basin of  attraction of an internal root $\alpha_1$. The six vertices of the hexagon are focal points and we only show two (of countable many) lobes attached to the focal points.}}
    \label{fig:polygon}
    \end{figure}

The goal of this paper is to describe the {\it topology} of  the immediate basin of attraction of an internal root  of $p$, when the roots are simple. We collect the main results on two statements. The first one is about the topology of the {\it external boundary} of $\partial \mathcal A^*(\alpha)$  and its dynamics.


\begin{thmA}
Let $\alpha_1$ be an internal root of $p$ and let $\alpha_0 < \alpha_1 < \alpha_2$ be simple consecutive roots of $p$. The following statements hold,  provided the external boundary is piecewise smooth.
\begin{enumerate}
\item[(a)] $\partial \mathcal A^*(\alpha_1)$ contains an hexagon-like polygon with lobes where the vertices are the focal points\footnote{Focal points and lobes
will be recalled in Section \ref{sec:planerational}.} $Q_{i,j} \ i\ne j \in \{0,1,2\}$.
\item[(b)] There exists a 4-cycle  in $\partial \mathcal A^*(\alpha_1)$.
\end{enumerate}
\end{thmA}

Secondly we investigate the connectedness of the immediate basin of attraction. Looking at the examples in Figure \ref{fig:chebyshev} the immediate basin of attraction of an internal root seems to be simply connected. However, it is easy to find examples where $\mathcal A^*(\alpha)$ is multiply connected. See Figure \ref{fig:multiply_connected}. In the next result we find sufficient conditions to assure that the immediate basin of attraction of an internal root is a simply connected set.

\begin{figure}[ht]
    \centering
    \subfigure[\scriptsize{$p_1(x)=16x^5-20x^3+x+0.8$.}]
    {\includegraphics[width=0.4\textwidth]{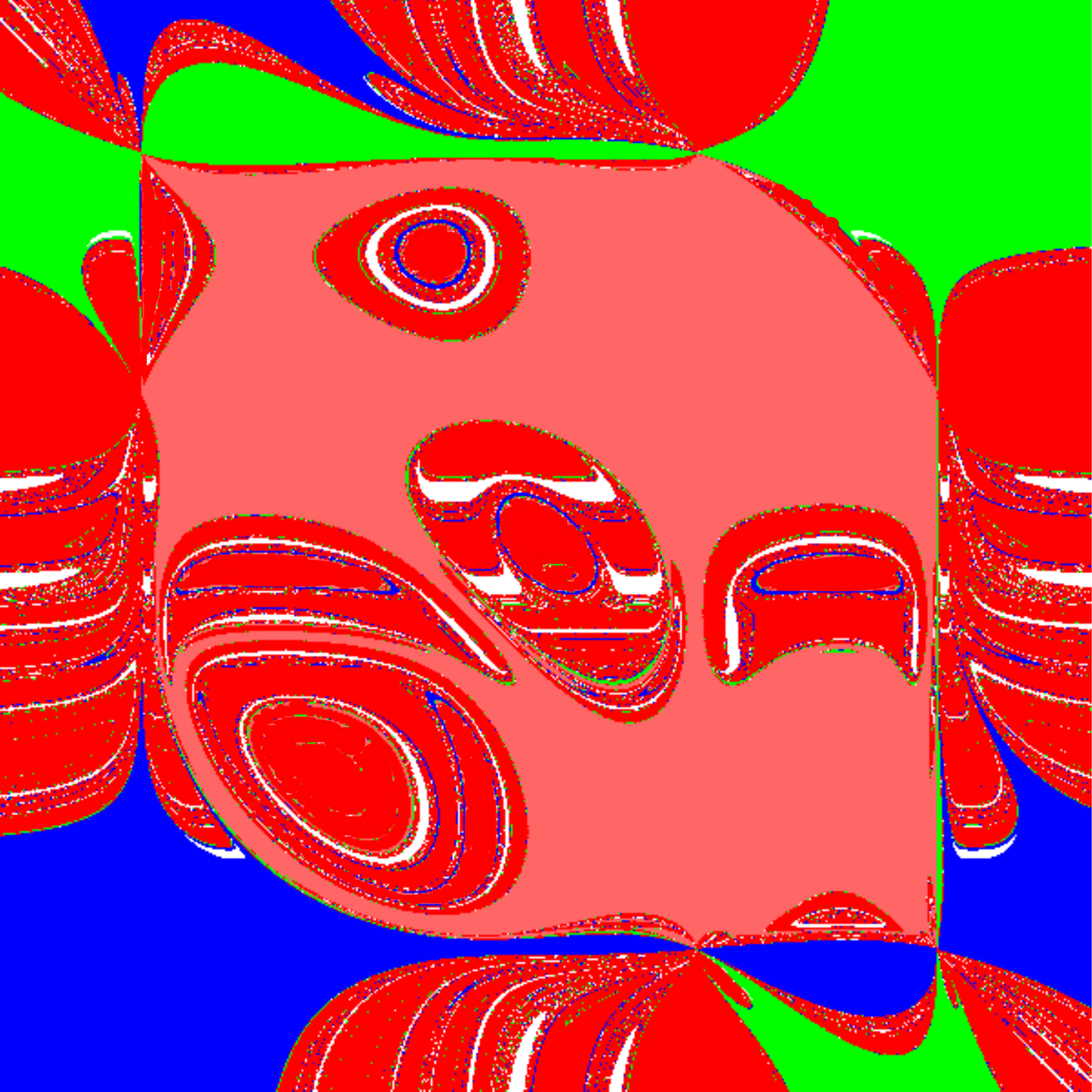}}
     \put(-60,115) {\scriptsize $\left(\alpha,\alpha\right)$} 
     \put(-69,115) {$\bullet$}
   \hglue 0.2truecm \subfigure[\scriptsize{$p_2(x)=\frac{x5}{5}-\frac{x^3}{3}-0.05x+0.15$.}]{\includegraphics[width=0.4\textwidth]{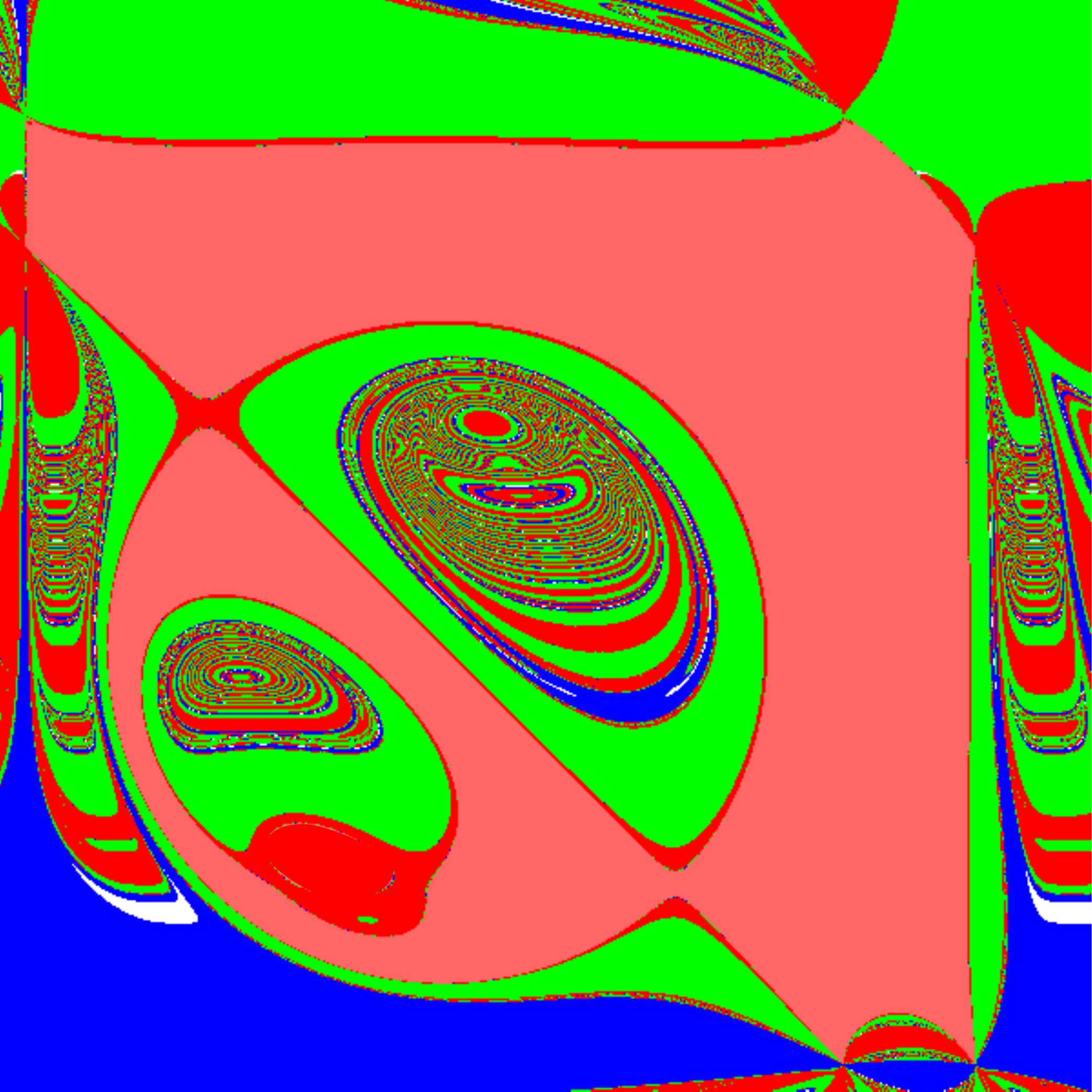}}
   \put(-68,137) {\scriptsize $\left(\alpha,\alpha\right)$} 
     \put(-43,135) {$\bullet$}
    \\
      \subfigure[\scriptsize{ Graph of the polynomilal $p_1$.}]{\includegraphics[width=0.35\textwidth]{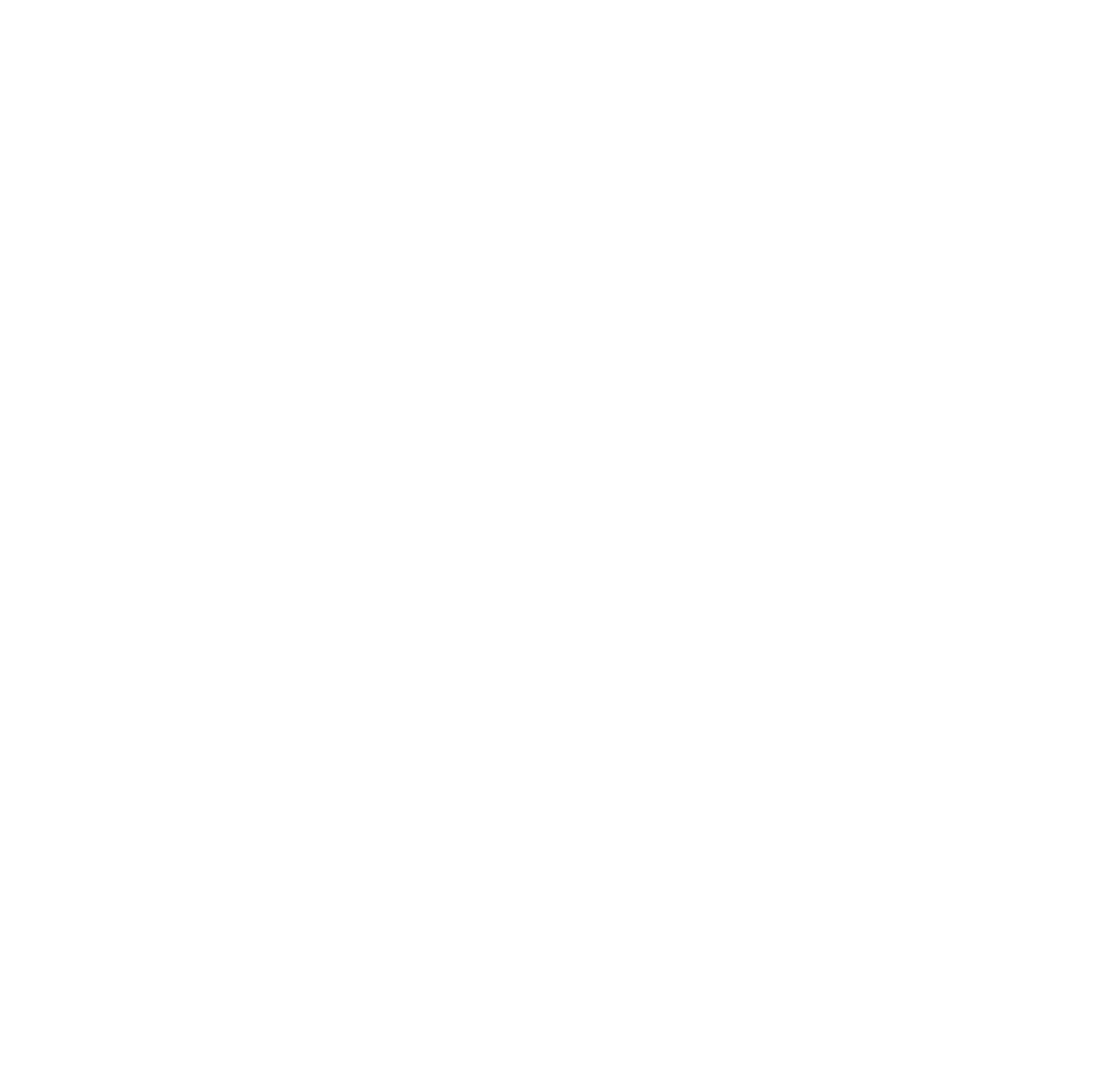}}
          \subfigure[\scriptsize{Graph of the polynomial $p_2$.}]{\includegraphics[width=0.5\textwidth]{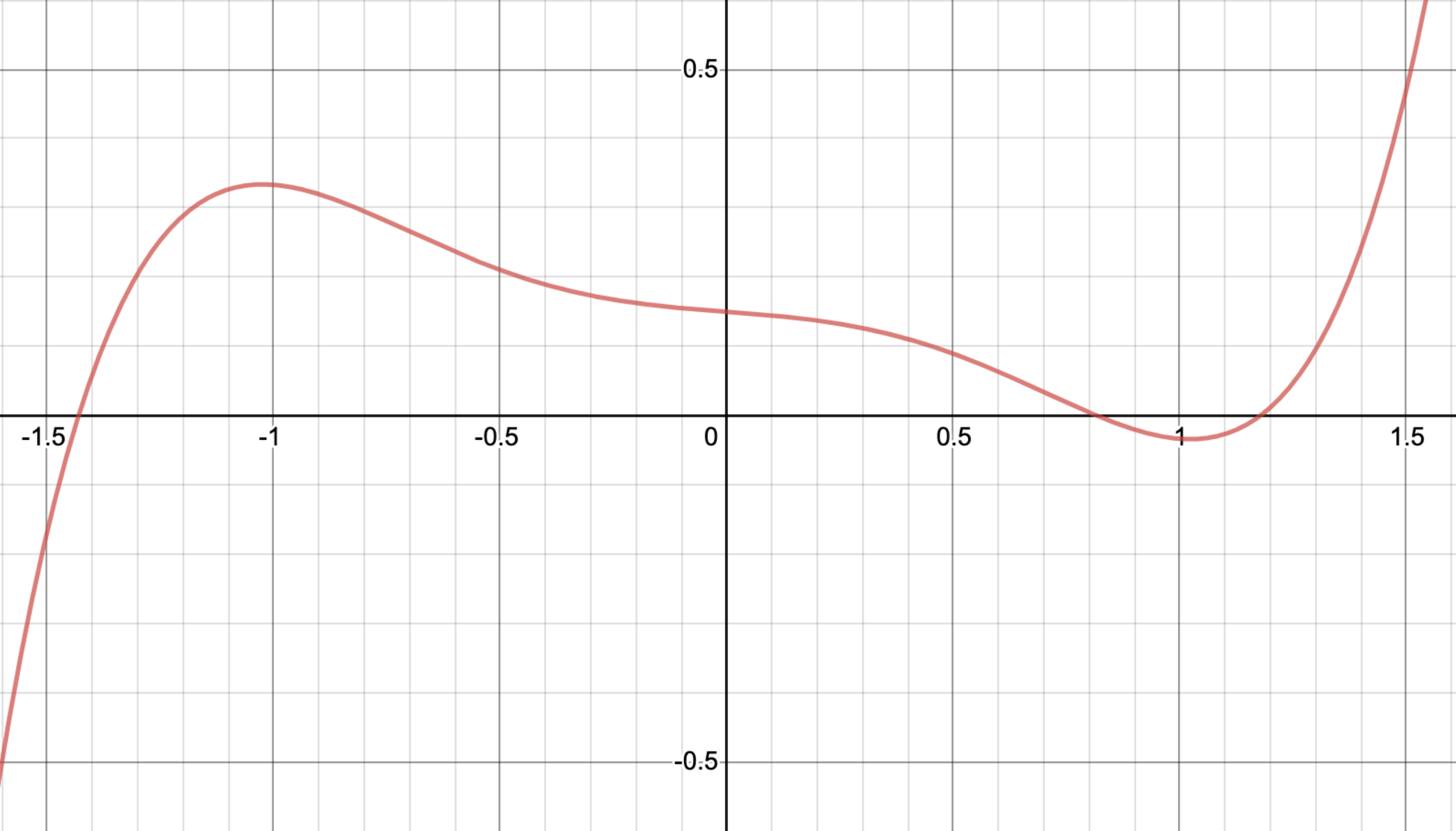}}
 \caption{\small{Phase plane of the secant map applied to two degree five polynomials (with only tree real roots) where the immediate basin of attraction of an internal root is multiply connected. In both cases, color blue and green refers to the attracting basins of the external roots, while red correspond to the attracting basin of the internal root $\alpha$. We use {\it pink} to emphasize the immediate basin of attraction of the internal root.  Range of the phase planes  [-1.5,1.5]x[-1.5,1.5]. We also show in (c) and (d) the graph of each polynomial.}}
    \label{fig:multiply_connected}
    \end{figure}

\begin{thmB}
Let $\alpha_1$ be an internal root of $p$ and let $\alpha_0 < \alpha_1 < \alpha_2$ be simple consecutive roots of $p$. Assume that $p$ has only one inflection point in the interval $(\alpha_0, \alpha_2)$, provided the external boundary is piecewise smooth. Then the immediate basin of attraction $\mathcal A^*(\alpha_1)$  is simply connected.
\end{thmB}

From Theorem A and Theorem B we can conclude the following corollary that applies to any real polynomial of degree $k$ with exactly $k$ simple real roots, as the family of Chebychev polynomials.

\begin{corC}\label{cor:simple_roots}
Let $p$ be a polynomial of degree $k$ with exactly $k$ simple real roots and one, and only one, inflection point between any three consecutive roots of $p$. Then for any internal root $\alpha$ of $p$ the immediate basin of attraction, $\mathcal A^*(\alpha)$, is a simply connected set and $\partial \mathcal A^*(\alpha)$ is an hexagon-like polygon with lobes where the vertices are focal points. Moreover, there exists a 4-cycle  in $\partial \mathcal A^*(\alpha)$. 
\end{corC}

%
The paper is organized as follows. In section \ref{sec:planerational} we introduce the terminology and tools on rational iteration on the plane. In section \ref{sec:fourcycle} we classify the cycles of minimal period 4 of the secant map. In section \ref{sec:interior}  and  \ref{section:TB} we prove Theorem A and Theorem B, respectively.

 \section{Plane rational iteration} \label{sec:planerational}

For the sake of completeness we briefly summarize the notions, tools and results from  \cite{PlaneDenominator1,PlaneDenominator2,PlaneDenominator3} which are needed here. Consider the plane rational map given by  
\begin{equation}
T: \left( 
\begin{array}{l}
x \\
y
\end{array}
\right) \mapsto  
 \left( 
\begin{array}{l}
F(x,y) \\
N(x,y)/D(x,y)
\end{array}
\right),
\label{eq:PlaneDenominator}
\end{equation}
where $F$, $N$ and $D$ are differentiable functions. Set  
\[
\delta_T = \{ (x,y) \in \mathbb R^2 \, | \, D(x,y)=0\} \quad {\rm and} \quad E_T =\mathbb R^2 \setminus \bigcup_{n \geq 0} T^{-n}(\delta_T).
\]
Easily $T=(T_1,T_2):E_T\to E_T$ defines a smooth dynamical system given by the iterates of $T$; that is $\{(x_m,y_m):=T^m\left(x_0,y_0\right)\}_{m \geq 0}$ with $(x_0,y_0)\in E_T$ (see \cite{Tangent} for details). Clearly $T$ {\it sends} points of $\delta_T$ to infinity unless $N$ also vanishes. At those points the definition of $T$ is uncertain in the sense that the value depends on the path we choose to approach the point.  As we will see  they play a crucial role on the local and global dynamics of $T$.

We say that a point  $Q \in \delta_T \subset \mathbb R^2$ is a  {\bf focal point (of $T$)} if $T_2(Q)$ takes the form 0/0 (i.e. $N(Q)=D(Q)=0$), and there exists a smooth simple arc $\gamma:=\gamma(t),\ t\in (-\varepsilon,\varepsilon)$, with $\gamma(0)=Q$, such that $\lim_{t \to 0} T_2(\gamma)$ exists and is finite. Moreover, the straight line given by $L_Q=\{(x,y)\in \mathbb R^2 \ | \ x=F(Q)\}$ is called the {\bf prefocal line (over $Q$)}. 

Let $\gamma$ passing through $Q$,  not tangent to $\delta_T$, with slope $m$ (that is $\gamma^\prime(0)=m$). Then $T\left(\gamma\right)$ will be a curve passing, at $t=0$, through some finite point $(F(Q),y(m)) \in L_Q$. If $Q$ is {\it simple} (that is, $N_x(Q)D_y(Q)-N_y(Q)D_x(Q)\neq0$)
then there is a  one-to-one correspondence between the slope $m$ and points in the prefocal line $L_Q=\{(x,y)\in \mathbb R^2 \ | \ x=F(Q) \}$. Precisely (Figure \ref{fig:focals_simple} illustrates the one-to-one correspondence), 
\begin{equation} \label{eq:one-to-one}
y(m) = \frac{ N_x(Q) + m N_y(Q) }{D_x(Q)+m D_y(Q)} \quad {\rm or} \quad  m(y)= \frac{D_x(Q)  y - N_x(Q)}{N_y(Q)-D_y(Q)y}.
\end{equation}
 
 \begin{figure}[ht]
    \centering
     \includegraphics[width=0.75\textwidth]{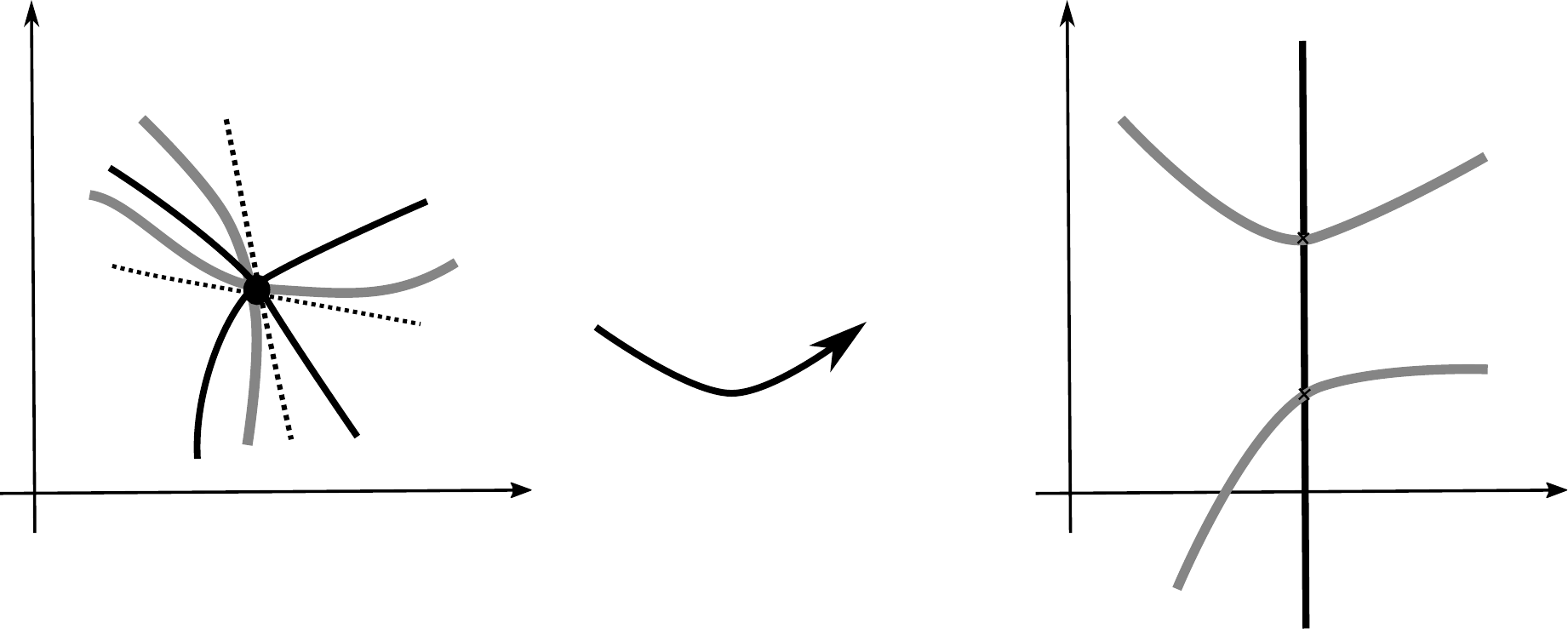}
    \put(-180,38) {\small $T$ } 
     \put(-305,112){\small $\gamma_1$}
     \put(-325,92){\small $\gamma_2$}
     \put(-255,58){\small $m_2$}
     \put(-275,35){\small $m_1$}
     \put(-278,80){\small $Q$}
     \put(-53,77){\small $y(m_1)$}
     \put(-53,42){\small $y(m_2)$}
     \put(-250,93){\small $ \delta_T\, [D(x,y)=0]$}
     \put(-253,37){\small $ N(x,y)=0$}
     \put(-15,100){\small $T(\gamma_1$)}
     \put(-12,52){\small $T(\gamma_2$)}
      \put(-50,3){\small $L_Q \, [  x = F(Q)] $}
           \caption{\small{Dynamics of $T$ near a simple focal point $Q$.}}
    \label{fig:focals_simple}
    \end{figure}
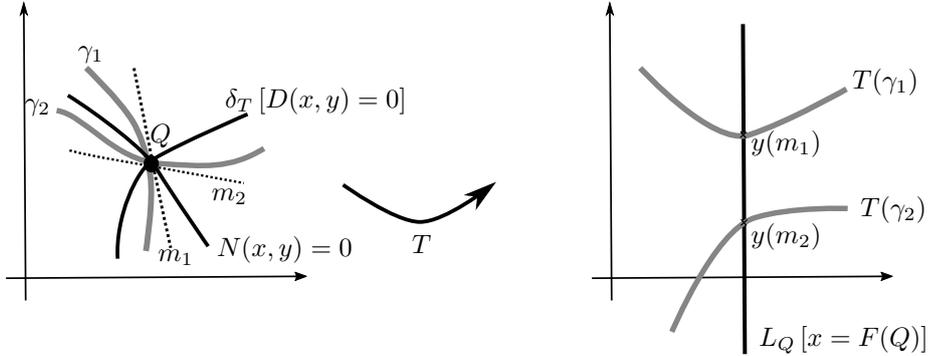

Among other dynamical aspects simple focal points are {\it responsible} of the presence of lobes and crescents  in the phase space of noninvertible maps, and in particular in the phase plane of the secant map (see Figure \ref{fig:chebyshev}).  This kind of phenomena occurs when a basin of attraction intersects the prefocal line. Again we refer to   \cite{PlaneDenominator1,PlaneDenominator2,PlaneDenominator3} for other details.
    
\begin{remark}
The name focal point used here to refer the points where the map $T$ is uncertain are also known as points of indeterminacy in complex and geometric analysis. 
\end{remark}

In Figure \ref{fig:focals_simple1} we sketch the mechanism for the creation of lobes in the phase plane of a noninvertible map with denominator.  If there exists an arc $\gamma$ crossing the prefocal line $L_Q$ in two different points $y(m_1)$ and $y(m_2)$ then a preimage of $T$ has a lobe issuing
from the focal point $Q$. If the map has two inverses and two focal points we can have two different lobes $T^{-1}_a(\gamma)$ and $T^{-1}_b(\gamma)$ issuing from $Q_a$ and $Q_b$. Also notice that if $\gamma$ is a lobe crossing the prefocal line $L_Q$ in one point $y(m)$ then an inverse $T^{-1}(\gamma)$ gives also a lobe from a focal point $Q$ but with two arcs having the same tangent $m$.

 \begin{figure}[ht]
    \centering
     \includegraphics[width=0.95\textwidth]{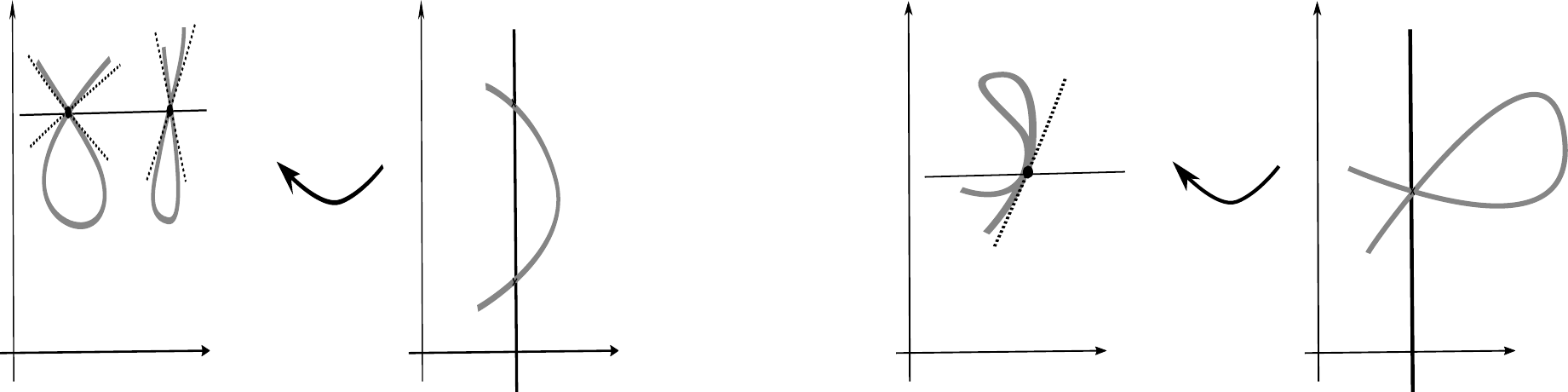}
    \put(-335,35) {\small $T^{-1}$ } 
          \put(-280,0){\small $L_Q $}
     \put(-280,78){\footnotesize $y(m_2)$}
     \put(-270,45){\small $\gamma$}
     \put(-280,25){\footnotesize $y(m_1)$}
     \put(-370,77){\footnotesize $Q_a$}
          \put(-415,76){\footnotesize $Q_b$}
               \put(-4,79){\small $\gamma$}
                   \put(-95,40) {\small $T^{-1}$ } 
               \put(-41,43){\footnotesize $y(m)$}
             \put(-143,50){\footnotesize $Q$}
               \put(-140,85){\footnotesize $m$}
      \put(-40,0){\small $L_Q $}

           \caption{\small{We sketch the mechanism responsible of the creation of lobes at focal points of a  map with denominator.}}
    \label{fig:focals_simple1}
    \end{figure}
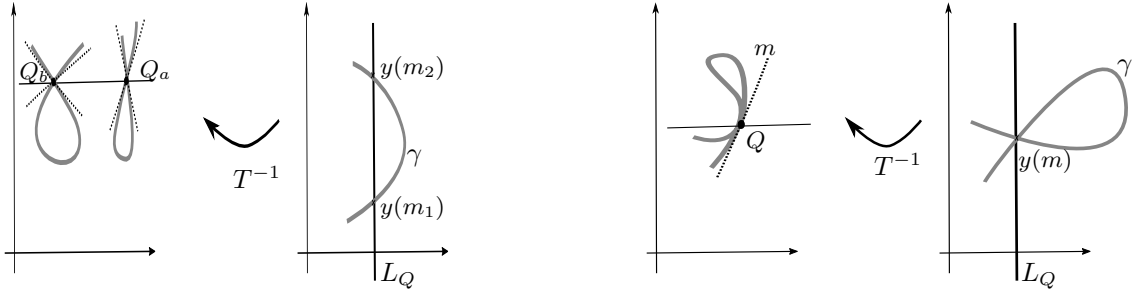


In \cite{Tangent} the authors used this approach to study the particular case of the secant map, that is when $T=S$, defined in \eqref{eq:secant_real}, under the assumption that all real roots of $p$ are simple. In particular it was shown that the equation   
\begin{equation}\label{eq:q}
p(x)-p(y)= q(x,y)(x-y),
\end{equation}
defines $q$ as a polynomial (that is, $x-y$ divides the polynomial $p(x)-p(y)$). Therefore, the secant map can also be written as
\begin{equation}
\label{eq:def_S_2}
S (x,y) = \left( y, \frac{y q(x,y)-p(y)}{q(x,y)}\right).
\end{equation}

Moreover, for the secant map, the set $\delta_T$ reduces to
 \begin{equation}\label{eq:delta_s}
 \delta_S = \{ (x,y) \in \mathbb R^2 \,; \,  x \neq y \, \hbox{and } p(x)=p(y)\} \,  \cup \, \{ (x,x) \in \mathbb R^2 \, ; \, p'(x)=0  \},
 \end{equation}
and focal points are given by $Q_{i,j}=(\alpha_i,\alpha_j)$ with $i \neq j$ running over all possible pairs of the roots of $p$. Easily, the prefocal line of $Q_{i,j}$ is the vertical line  $L_j= \{ (x,y) \in \mathbb R^2 \, ; \, x = \alpha_j\}$.  The one-to-one correspondence at the focal point $Q_{i,j}$ described in \eqref{eq:one-to-one} writes as 
\begin{equation}\label{eq:y(m)_Sp}
y(m)=\frac{\alpha_j p'(\alpha_i)-\alpha_i p'(\alpha_j)m}{p'(\alpha_i)-p'(\alpha_j)m} \quad {\rm or} \quad 
m(y)= \frac{p'\left(\alpha_i\right)\left(\alpha_j-y\right)}{p'\left(\alpha_j\right)\left(\alpha_i-y\right)}.
\end{equation}

%
%
%
%

\section{Periodic orbits of minimal period 4} \label{sec:fourcycle}

It can be proved  that the fixed points of the secant map applied to  the polynomial $p$  are given by the points $\left(\alpha,\alpha \right)$, where $\alpha$ is a root of $p$, and that they are all attracting. It is also known (see  \cite{BedFri,Tangent}) that the secant map has no periodic orbits of period two and three in the plane although  every critical point $c$ (i.e., $p'(c)=0$) has associated a periodic orbit of period three given by 
$$
(c,c)  \xrightarrow{S}  (c,\infty) \xrightarrow{S}  (\infty,c) \xrightarrow{S}  (c,c)$$
after properly extending $S$ to $\infty$.  Hence, it is natural to study the relevance of the four periodic orbits in the global dynamics. We already known that those periodic orbits might be attracting. See \cite{BedFri, Tangent} for precise statements. 

In this section we  study in detail the possible configurations of the period four orbits or 4-cycles, a key step to understand the boundary of the immediate basin of attraction of the fixed points of $S$. Assume that $S$ has a periodic orbit of (minimal) period 4 given by 
\begin{equation} \label{eq:4cycle}
(a,b) \xrightarrow{S} (b,c)  \xrightarrow{S} (c,d) \xrightarrow{S} (d,a) \xrightarrow{S} (a,b),
\end{equation}
where $a,b,c,d$ are real numbers. Under this notation we are describing the dynamics of the 4-cycle (as points in $\mathbb R^2$), but notice that we are not determining the relative position in $\mathbb R$ of the points $a,b,c$ and $d$ involved in the cycle. However, renaming points in the four cycle we can assume that $a$  is the value in the cycle with minimum value, that is, we can assume  without loss of generality  that  $a := \min\{a,b,c,d\}$ and the dynamics of the cycle is still given by \eqref{eq:4cycle}.

We recall that if $a,b,c,d$ are real numbers then the {\it cross ratio}, $\lambda_{(a,b\,;c,d)}$, is given by the expression
\begin{equation}\label{eq:crossratio}
\lambda:=\lambda_{(a,b\,;c,d)}=\frac{(c-a)(d-b)}{(c-b)(d-a)}.
\end{equation}
Easy computations show that 
\begin{equation} \label{eq:ratios}
\lambda_{(a,d\,;c,b)} = \frac{\lambda}{\lambda-1} \quad {\rm and} \quad \lambda_{(d,c\,;b,a)} = \lambda
\end{equation}

The next proposition classifies completely the possible types of 4-cycles (see Figure \ref{fig:4cycles_types}) depending on the relative position of the base points. 

\begin{prop}[{\bf Classification of  4-cycles}] \label{prop:configurations}
Assume that the secant map $S$ exhibits a 4-cycle as in \eqref{eq:4cycle}.
Then $\lambda = (-1+\sqrt{5})/2$ or $\lambda=  -(1+\sqrt{5})/2$.  The possible configurations (i.e, the relative position in $\mathbb R$ of the points $a,b,c,d$ involved in the cycle and  their images by $p$) are listed in Table \ref{table:types} and leads to four different types as described in Figure \ref{fig:4cycles_types}. Moreover, the four types of 4-cycles  are admissible.


\end{prop}

 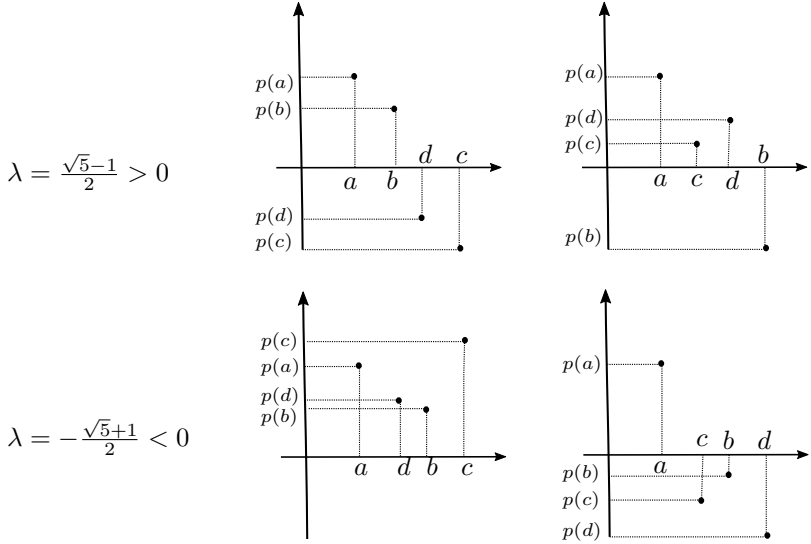
\begin{figure}[ht]
    \centering
     \includegraphics[width=0.45\textwidth]{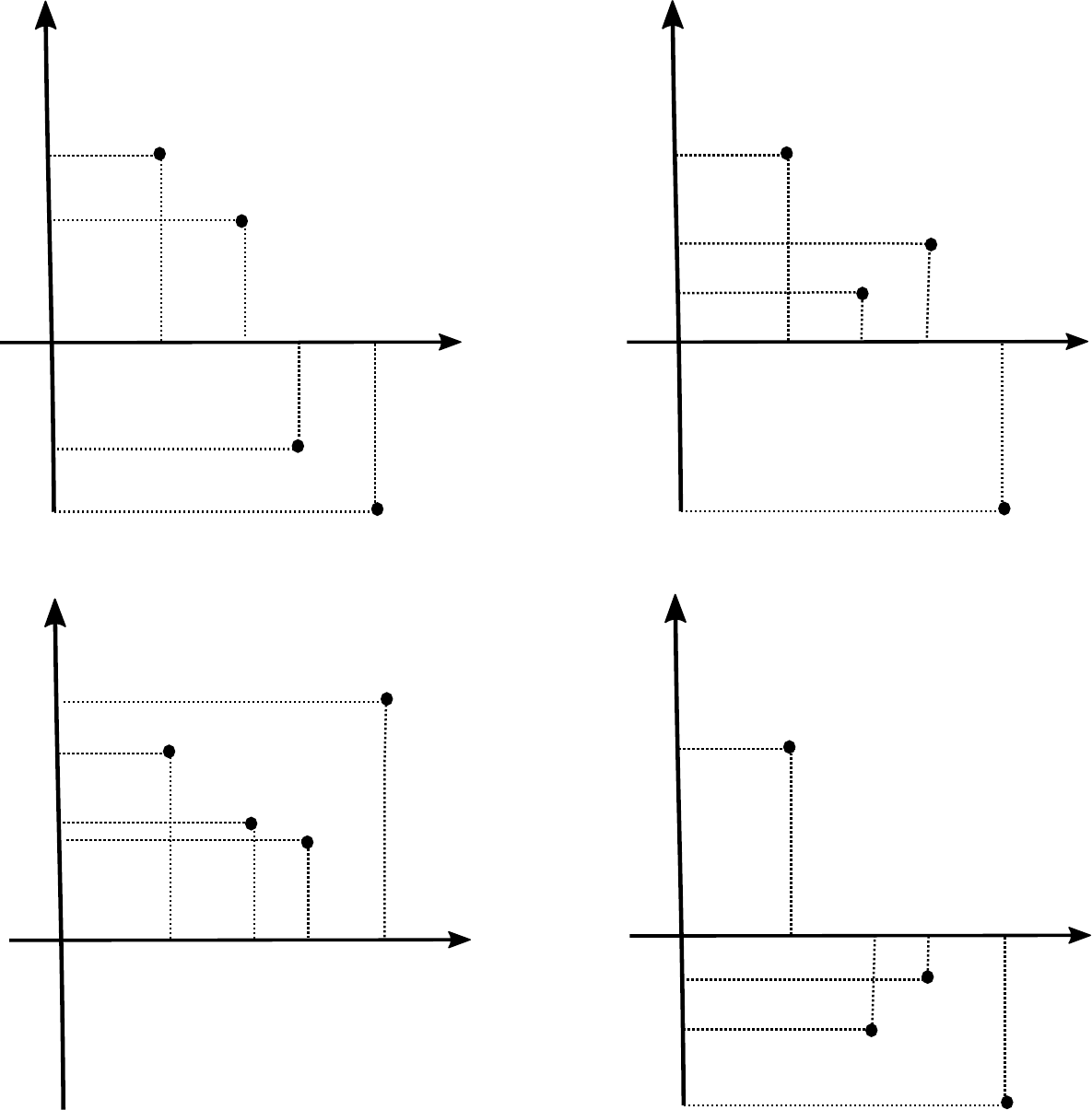}
      \put(-146,142) {\small $d$ } 
    \put(-132,142) {\small $c$ } 
    \put(-158,132) {\small $b$ } 
    \put(-174,132) {\small $a$ }
  \put(-207,170) {\tiny $p(a)$ }
    \put(-207,160) {\tiny $p(b)$ }
    \put(-207,120) {\tiny $p(d)$ }
    \put(-207,110) {\tiny $p(c)$ }
     \put(-58,25) {\small $a$ } 
      \put(-42,35) {\small $c$ } 
        \put(-32,34) {\small $b$ } 
        \put(-19,34) {\small $d$ } 
    \put(-92,64) {\tiny $p(a)$ }
    \put(-92,23) {\tiny $p(b)$ }
    \put(-92,13) {\tiny $p(c)$ }
    \put(-92,1) {\tiny $p(d)$ }  
          \put(-170,24) {\small $a$ } 
      \put(-154,24) {\small $d$ } 
      \put(-143,24) {\small $b$ } 
      \put(-130,24) {\small $c$ }     
        \put(-205,63) {\tiny $p(a)$ }
    \put(-205,52) {\tiny $p(d)$ }
    \put(-205,44) {\tiny $p(b)$ }
    \put(-205,73) {\tiny $p(c)$ }
          \put(-58,133) {\small $a$ } 
      \put(-44,133) {\small $c$ } 
        \put(-31,132) {\small $d$ } 
        \put(-19,142) {\small $b$ } 
        \put(-91,174) {\tiny $p(a)$ }
    \put(-91,157) {\tiny $p(d)$ }
    \put(-91,147) {\tiny $p(c)$ }
    \put(-91,112) {\tiny $p(b)$ }     
        \put(-300,135) {\small $\lambda = \frac{\sqrt{5}-1}{2} >0 $ }     
                \put(-300,35) {\small $\lambda=  -  \frac{\sqrt{5}+1}{2} <0$ }     
           \caption{\small{The four different types of 4-cycles for the secant map.  We show type I in top-left and type II in top-right,  corresponding to a cross ratio  $\lambda >0$. We show type III in bottom-left  and type IV in bottom-right, corresponding to cross ratio $\lambda <0$}}
    \label{fig:4cycles_types}
    \end{figure}

\begin{table}[h] 
\begin{tabular}{|c|l|l|l|l|}
\hline
\multirow{2}{*}{$\lambda>0$}                       & $a<b<c<d$ & $\star$  & $a<d<c<b$ & $\star$  \\ \cline{2-5} 
                                                   & $a<b<d<c$ & Type I   & $a<c<d<b$ & Type II  \\ \hline
\multicolumn{1}{|l|}{\multirow{1}{*}{$\lambda<0$}} & $a<d<b<c$ & Type III  & $a<c<b<d$ & Type IV   \\ \hline
\end{tabular}
\vglue 0.3cm
\caption{\small All possible configurations of a 4-cycle and their corresponding type. Here $\star$ means incompatible configuration with a 4-cycle. See Proposition \ref{prop:configurations}.} 
\label{table:types}
\end{table}

\begin{proof}
Using the definition of the secant map and the configuration given in \eqref{eq:4cycle} we easily have that 
{\small  \[
c =     b - p(b) \displaystyle \frac{b-a}{p(b)-p(a)} \quad  
d  =   c - p(c) \displaystyle \frac{c-b}{p(c)-p(b)}\quad 
a =  d - p(d) \displaystyle \frac{d-c}{p(d)-p(c)}\quad
b = a  - p(a) \displaystyle \frac{a-d}{p(a)-p(d)} 
\]
}

\noindent which is equivalent to  

\begin{equation}\label{eq:secant4}
\frac{p(a)}{p(b)} =    \frac{c-a}{c-b} ,  \quad  \frac{p(b)}{p(c)} =    \frac{d-b}{d-c} ,  \quad  \frac{p(c)}{p(d)} =    \frac{a-c}{a-d},   \quad \frac{p(d)}{p(a)} =    \frac{b-d}{b-a}.
\end{equation}
Multiplying both sides  of these four equations we obtain that 
\[
1 =  - \left[  \frac{(c-a)(b-d)}{(c-d)(b-a)}  \right]  \left[ \frac{(b-d)(a-c)}{(b-c)(a-d)} \right]  = -  \lambda_{(a,d;c,b)} \lambda_{(d,c;b,a)}=\frac{\lambda^2}{1-\lambda},
\]
and so $\lambda\in \{(-1+ \sqrt{5})/2,(-1- \sqrt{5})/2\}$.

Now we turn the attention to the classification of a 4-cycle of the secant map. Firstly, let us notice the following property of the secant map. Given two points $x_0<y_0$ the secant map is given by $S(x_0,y_0)=(y_0,z_0)$ where $\left(z_0,0\right)$ is the intersection between the line passing thorugh the points $(x_0,p(x_0))$ and $(y_0,p(y_0))$, and the horizontal line $y=0$. Thus, if $z_0\in \left(x_0,y_0\right)$  then $p\left(x_0\right)p\left(y_0\right)<0$ while if $z_0 \not\in \left(x_0,y_0\right)$ then $p\left(x_0\right)p\left(y_0\right)>0$. 

We need to consider 6 cases depending on the relative position of the points ${a,b,c,d}$ on the real line since we have assumed that $a < \min\{b,c,d\}$. It follows from the definition of the cross ratio $\lambda(a,b;c,d)$  \eqref{eq:crossratio}  that   $\lambda$ is positive  if and only if one and only one of $c$ and $d$ lays between $a$ and $b$. So, there are four cases where $\lambda>0$ and two cases where $\lambda<0$.
\vglue 0.2truecm
 
{\bf Case 1. $a<b<c<d \ (\lambda>0)$.}  We have $S(a,b)=(b,c)$ and $c\not\in (a,b)$. So $p(a)p(b)>0$. Since $S(b,c)=(c,d)$ and $b<c<d$ we get $p(c)p(b)>0$. Also, since $S(c,d)=(d,a)$ and $a \not\in(c,d)$, we get $p(a)p(b)p(c)p(d)>0$. Finally, since $S(d,a)=(a,b)$ and $b\in(a,d)$ we have $p(a)p(d)<0$, a contradiction. Thus there is no 4-periodic orbits with this configuration.
\vglue 0.1truecm

{\bf Case 2. $a<b<d<c \ (\lambda>0)$.}  We have 
$S(a,b)=(b,c)$ with $c>b$. So $p(a)p(b)>0$ (we assume $p(a)>p(b)>0$, the case $p(a)<p(b)<0$ follows similarly). Since $S(b,c)=(c,d)$ and $d\in(b,c)$ then $p(c)<0$ (we have assumed $p(b)>0$). Also we have $S(c,d)=(d,a)$ and since $a<d<c$ then $p(c)<p(d)<0$. Finally $S(d,a)=(a,b)$ which is compatible with the fact that $p(a)p(d)<0$. This 4-cycle corresponds to {\it type I}. See Figure  \ref{fig:4cycles_types} (first row  left). 
\vglue 0.1truecm

{\bf Case 3. $a<c<d<b \ (\lambda>0)$.}  We have $S(a,b)=(b,c)$ with $c\in (a,b)$. So $p(a)p(b)<0$ (moreover, assuming that $p(a)>0$, we have that $p(b)<0$; the case $p(a)<0$ follows similarly). Since  $S(b,c)=(c,d)$ and $d\in(c,b)$ we have $p(c)>0$. Since $S(c,d)=(d,a)$, $p(c)>0$ and $a<c<d$ we have $p(d)>p(c)>0$.  Finally, since $S(d,a)=(a,b)$ with $a<d<b$ we get  $p(a)>p(d)>p(c)>0$ and $p(b)<0$, a compatible configuration which corresponds to {\it type III}. See Figure  \ref{fig:4cycles_types} (first row right). 

\vglue 0.1truecm

{\bf Cases 4. $a<d<c<b \ (\lambda>0)$.}   This case leads to an incompatible configuration and we left the details to the reader.

\vglue 0.1truecm

{\bf Case 5. $a<d<b<c \ (\lambda<0)$.}  We have $S(a,b)=(b,c)$ with $a<b<c$. So $p(a)p(b)>0$ (moreover, assuming that $p(b)>0$, we have that $p(a)>p(b)$; the case $p(b)<0$ follows similarly). Since  $S(b,c)=(c,d)$ and  $d \notin (b,c)$ we conclude that $p(c)>p(b)>0$. Since $S(c,d)=(d,a)$ and $a<d<c$ we have $0<p(d)<p(c)$. Hence $p(a),p(b),p(c)$ and $p(d)$ are all positive. Finally, since $S(d,a)=(a,b)$ with $a<b<d$, we conclude that this configuration is possible and corresponds to {\it type II} (the case $p(b)<0$ is symmetric with $p(a),p(b),p(c)$ and $p(d)$ all negative). See Figure  \ref{fig:4cycles_types} (second row left). 

\vglue 0.1truecm

{\bf Case 6. $a<c<b<d \ (\lambda<0)$.}  We have $S(a,b)=(b,c)$ with $c\in (b,a)$. So $p(a)p(b)<0$ (moreover, assuming that $p(a)>0$, we have that $p(b)<0$; the case $p(a)<0$ follows similarly). Since  $S(b,c)=(c,d)$ and $d\notin(c,b)$ we have $p(c)<p(b)<0$. Since $S(c,d)=(d,a)$, $p(c)<0$ and $a<c<d$ we have $p(d)<p(c)<0$.  Finally, since $S(d,a)=(a,b)$ with $a<b<d$ we get  $p(d)<p(c)<p(b)<0$ and $p(a)>0$, a compatible configuration which corresponds to {\it type III}. See Figure  \ref{fig:4cycles_types} (second row right).

\vglue 0.1truecm

In Figure \ref{fig:4cycles_types1} we show  the relative position of a 4-cycle of the secant map 
$ (a,b) \to  (b,c)  \to (c,d) \to  (d,a) \to (a,b)$ in the phase plane.  According to the different cases (see Table \ref{table:types}) we observe that cycles of  types I and II  are arranged  making a clockwise turn  while  this is not the case in  types III and IV since they  flip four times around the line $y=x$. 

 \begin{figure}[ht]
    \centering
     \includegraphics[width=0.45\textwidth]{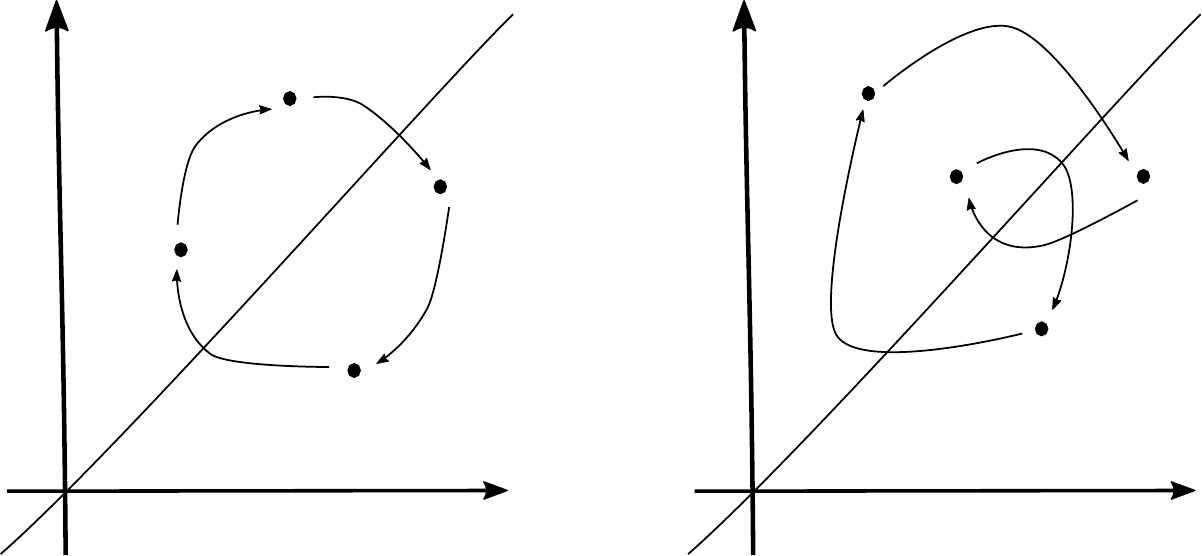}
  \caption{\small{Relative position of a four cycle with respect to the diagonal line $y=x$. On the left hand side types I and II and on the right hand side types III and IV. }}
    \label{fig:4cycles_types1}
    \end{figure}

We finally show that the four different types of  4-cycles are admissible. In fact we show how  to numerically built a concrete polynomial having a 4-cycle of Type I  and we leave the details of the other cases to the reader since the strategy is quite similar.

We choose  the  configuration: $a<b<d<c$ which corresponds to $\lambda >0$. We fix $a=1, b=2$ and $c=3$. Since we know that $\lambda=(\sqrt{5}-1)/2$ we get $d \approx 2.447213595$. Now we need to determine the value of $p(a),p(b),p(c)$ and $p(d)$ so that \eqref{eq:4cycle} is satisfied. From \eqref{eq:secant4} we can easily compute $p(a),p(b),p(c)$ and $p(d)$. Indeed it is an homogeneous linear system of equations with one degree of freedom. So fixing $p(d)=-1$ we obtain  
$p(a)\approx 2.23606798$, $p(b)\approx 1.118033989 $ and $p(c)\approx -1.381966011$. Finally we use Newton interpolation to get 
{\small
\begin{equation} \label{eq:pol}
p^{I}(x)=2.23606798  -1.11803390 (x-1) -0.6909830(x-1)(x-2)+3.27254249(x-1)(x-2)(x-3).
\end{equation}
}
According to the arguments above the secant map $S_{P^{I}}$ has a  4-cycle  of Type I (see Figure \ref{fig:type_I}). Similarly $S_{P^{II}},\, S_{P^{III}} $ and $S_{P^{IV}}$ have 4-cycle  of Type II, III  and IV, respectively, where 

\[
\begin{array} {rl} 
p^{II}(x) & =2.818 -5.236(x-1)+4.3316(x-2)(x-1)-16.106(x-2)(x-1)(x-3) \\
p^{III}(x) & =2.236 - 1.118(x-1)+1.809(x-2)(x-1)-0.4774(x-2)(x-1)(x-3) \, \\
p^{IV}(x) & =1.618 -2.118(x-1)+0.809(x-2)(x-1)-1.7135(x-2)(x-1)(x-3).
\end{array}
\]

\end{proof}

\begin{figure}[ht]
    \centering
     \includegraphics[width=0.45\textwidth]{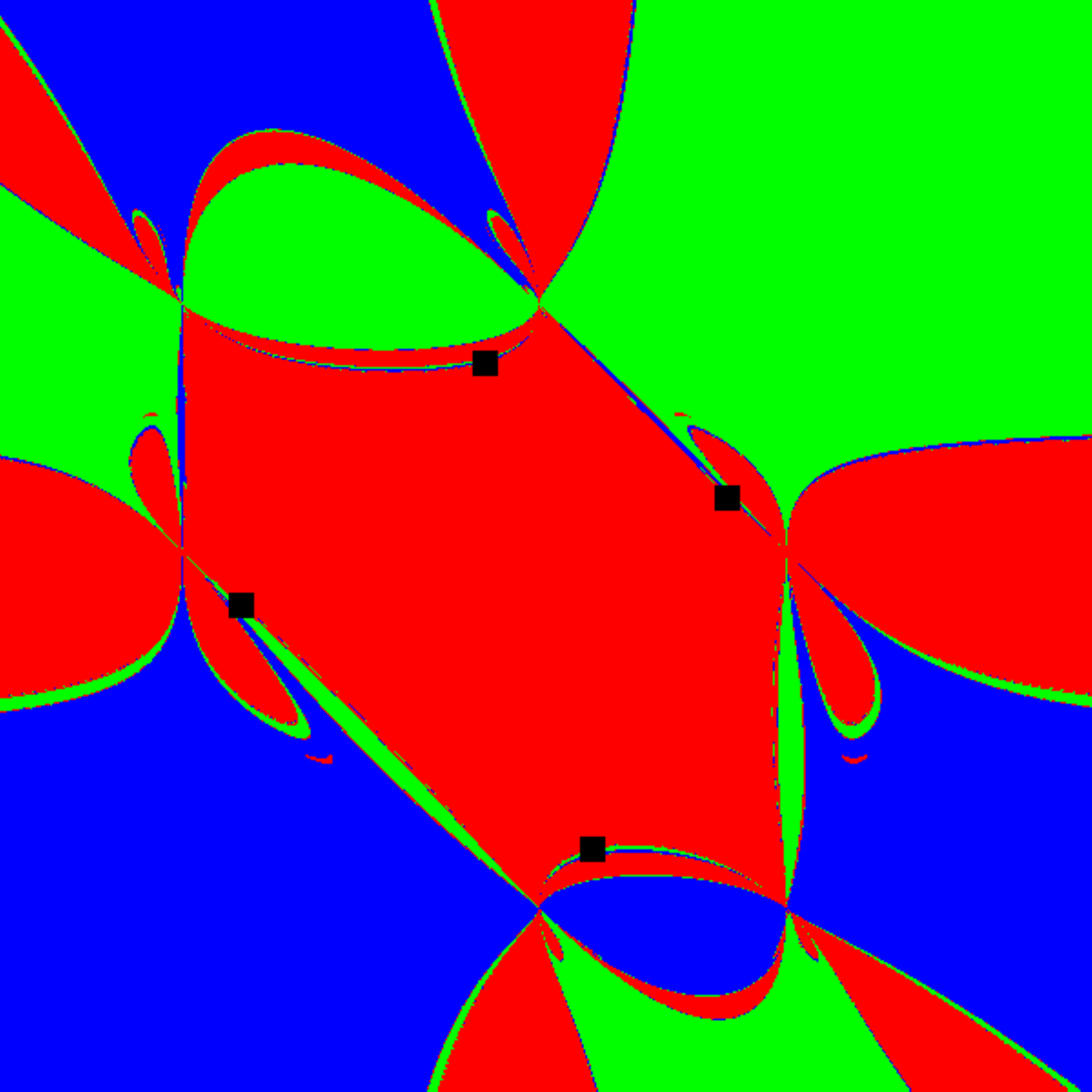}
   \put(-150,85) {\small $(1,2)$ } 
   \put(-143,105) {\Huge $ \nearrow$ } 
     \put(-120,120) {\small $(2,3)$ } 
        \put(-97,110) {\Huge $ \searrow$ } 
         \put(-95,95) {\small $(3,2.44)$ } 
                 \put(-90,70) {\Huge $ \swarrow$ } 
         \put(-99,49) {\small $(2.44,1)$ } 
                 \put(-130,60) {\Huge $ \nwarrow$ } 

           \caption{\small{Phase plane of the secant map applied to the polynomial $p^I$. We denote each point in the four cycle $(1,2) \xrightarrow{S}(2,3)  \xrightarrow{S} (3,2.447213595)  \xrightarrow{S} (2.447213595,1)  \xrightarrow{S} (1,2)$ with an small black square. }}
    \label{fig:type_I} 
    \end{figure}
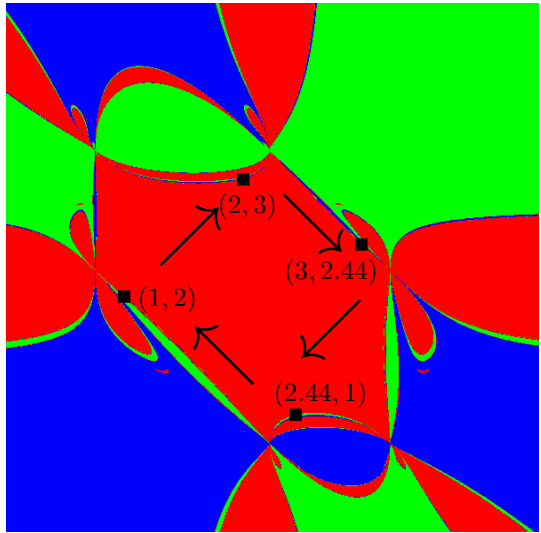


In Figure \ref{fig:type_I} we  show the phase plane of the secant map applied to the polynomial $p^I$. This polynomial exhibits three roots. We also show the four cycle 
\[
(1,2) \xrightarrow{S} (2,3)  \xrightarrow{S} (3,2.447213595) \xrightarrow{S} (2.447213595,1) \xrightarrow{S} (1,2).
\]
Every point  in the four cycle of Type I is shown in the picture with a small black square and we will see in the next sections the crucial role of this 4-cycle with the basin of attraction of the internal root of $p^I$.

\section{Proof of Theorem A}\label{sec:interior}

Firstly we prove the topological description of the boundary of the immediate basin of attraction of an internal root, that is Theorem A(a). At the end of the section we prove Theorem A(b).

Hereafter we fix the following notation.  We assume, without lost of generality,  that $\alpha_0 < \alpha_1 <\alpha_2$ are three consecutive real simple roots of $p$ and $p'(\alpha_0)>0$, $p'(\alpha_1)<0$ and $p'(\alpha_2)>0$. So  $p(x)>0$ for all  $x\in (\alpha_0,\alpha_1)$ and  $p(x)<0$ for all $x\in (\alpha_1,\alpha_2)$. Moreover, $p$ should have at least one critical point in each open interval $(\alpha_0,\alpha_1)$ and $(\alpha_1,\alpha_2)$. We denote by  $c_1$ the largest critical point of $p$ in $(\alpha_0,\alpha_1)$ and by $c_2$ the smallest critical point of $p$ in 
$(\alpha_1,\alpha_2)$ (equivalently the open interval $\left(c_1,c_2\right)$ is free of critical points). Of course $\alpha_1$ is the {\it target} internal root of Theorem A. See Figure \ref{fig:polynomial}.

Following the notation of Section \ref{sec:planerational} (see also \cite{Tangent}) one can show that the focal points of $S$ are given by $Q_{i,j}=(\alpha_i,\alpha_j), \ i \neq j \in \{0,1,2\}$, and that each $Q_{i,j}$ has the vertical line  $L_j=\{(x,y)\in \mathbb R^2 \ | \ x=\alpha_j\}$ as its prefocal line. Moreover, we also known that $\mathcal A^*(\alpha_1)$ is bounded. Next lemma makes this condition more precise.

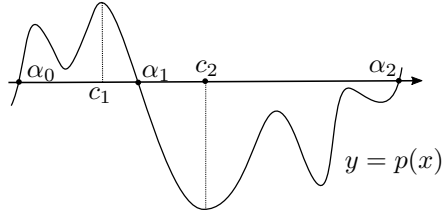
\begin{figure}[ht]
    \centering
     \includegraphics[width=0.35\textwidth]{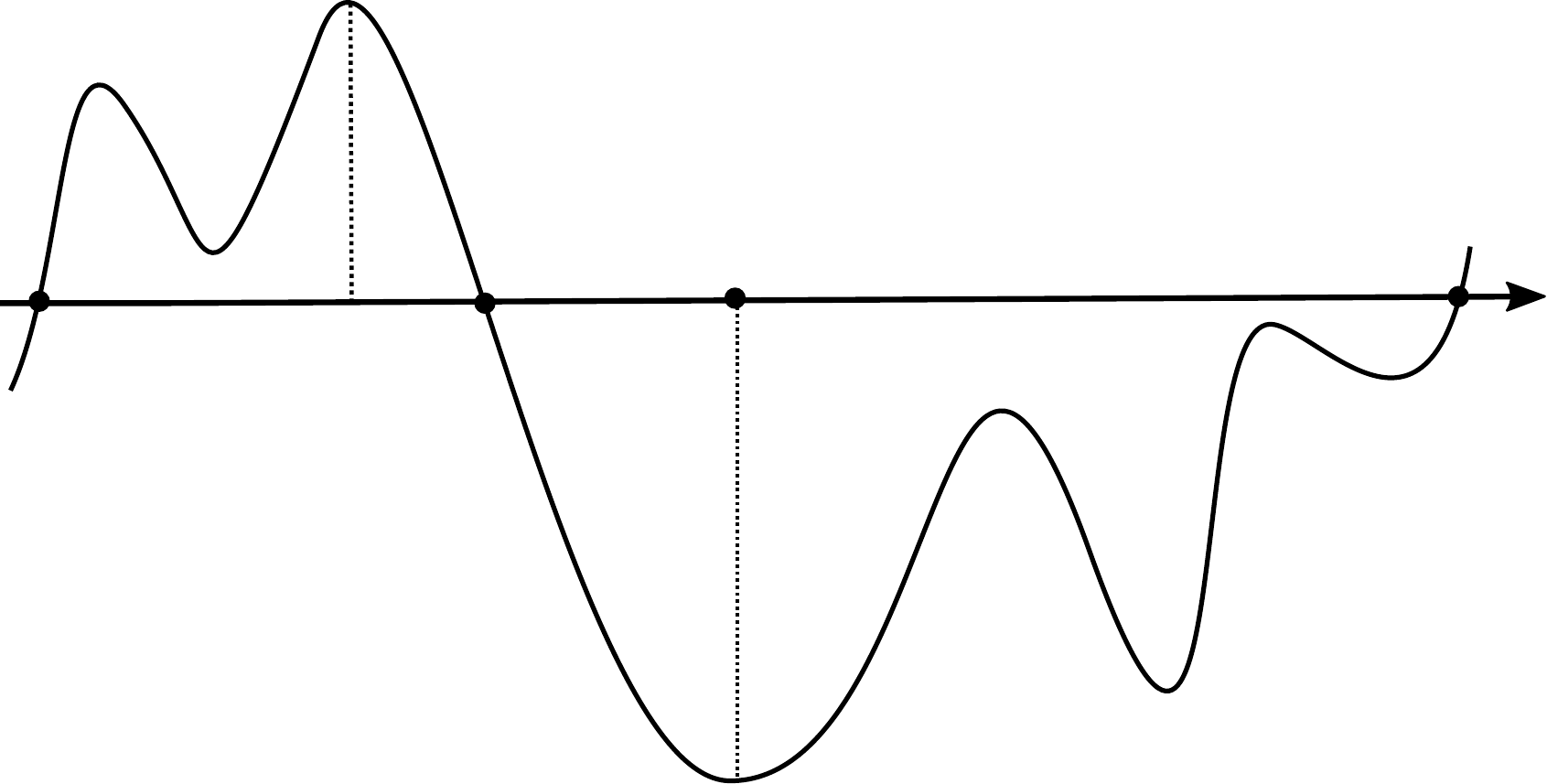}
    \put(-85,52) {\small $c_2$ } 
        \put(-105,50) {\small $\alpha_1$ } 
    \put(-148,50) {\small $\alpha_0$ } 
      \put(-20,52) {\small $\alpha_2$ }
          \put(-125,42) {\small $c_1$ }
          \put(-29,16) {\small $y=p(x)$}
             \caption{\small{Sketch of the polynomial $p$ with an internal root $\alpha_1$.}}
    \label{fig:polynomial}
    \end{figure}

\begin{lemma}\label{lem:bounded}
Let $\alpha_0<\alpha_1<\alpha_2$ be three real simple consecutive roots of $p$. Then $\mathcal A^*(\alpha_1) \subset R$ where $R:=\{(x,y)\in \mathbb R^2 \ | \ \alpha_0<x<\alpha_2, \, \alpha_0<y<\alpha_2\}$. 

\end{lemma}

\begin{proof}
From \eqref{eq:secant_real}  it is easy to see that given any root $\alpha\in \mathbb R$ of $p$ we have
$S(x,\alpha)=(\alpha,\alpha)$ and $S(\alpha,y)=(y,\alpha)$, as long as  $x$ and $y$ are not roots of $p$. This implies  that
$$
\left( \partial R \setminus
\bigcup_{i\ne j\in\{0,1,2\}} Q_{i,j} \right) \subset \left(A^{\star}\left(\alpha_0\right) \cup A^{\star}\left(\alpha_2\right)\right).
$$
Since the focal points $Q_{i,j}$ belong to $\delta_S$ where $S$ is not even defined the lemma follows. See Figure \ref{fig:period4}.
\end{proof}

We define the {\it external boundary} of $\mathcal  A^*(\alpha_1)$ as follows. Consider $U$ the open set $\mathbb C \setminus \overline{  \mathcal  A^*(\alpha_1)}$ and let $V$ be the unique unbounded connected component of $U$. Then  the external boundary of $\mathcal  A^*(\alpha_1)$ is $\partial V$. Notice that $V$ is unique since  $\mathcal  A^*(\alpha_1)$ is bounded (see Lemma \ref{lem:bounded}). We will assume that the external boundary of $\mathcal A^*(\alpha_1)$ is {\it piecewise smooth}; i.e.,  a union of smooth arcs (i.e., diffeomorphic to $(0,1)$) joining the focal points. 

\begin{prop}\label{prop:boundary_4cycle}
Let $p$ be a polynomial and let
$\alpha_0<\alpha_1<\alpha_2$ be three consecutive simple roots of $p$. Assume the external boundary of $\mathcal A^*(\alpha_1)$ is piecewise smooth. Then $\partial \mathcal A^*(\alpha_1)$ contains a smooth hexagon-like polygon with $C^1$-lobes where the vertices are the focal points $Q_{1,0},\,Q_{2,0},\,Q_{0,1},\,Q_{2,1},Q_{0,2}$ and $Q_{1,2}$, and lobes are issuing only from to $Q_{1,0},\,Q_{2,0},\,Q_{0,2}$ and $Q_{1,2}$.  
\end{prop}

\begin{proof}
We will assume, without lost of generality, that $p'\left(\alpha_0\right)>0$ (and so $p'\left(\alpha_1\right)<0$  and $p'\left(\alpha_2\right)>0$). 

Focal points do not belong to  $\mathcal A^*(\alpha_1)$, while  from Lemma \ref{lem:bounded} it follows  that the segments $S_v:=\{(\alpha_1,y) \, ; \, \alpha_0 < y < \alpha_2\}$ and $S_h:=\{(x,\alpha_1) \, ; \, \alpha_0 < x < \alpha_2\}$ do. In particular, we have that $\{Q_{0,1},Q_{2,1},Q_{1,0},Q_{1,2}\} \in \partial \mathcal A^*(\alpha_1)$. Since $\mathcal A^*(\alpha_1) \subset R$ (see Lemma \ref{lem:bounded}) and the external boundary of $\mathcal A^*(\alpha_1)$ is piecewise smooth 
there should be an arc $I\subset \partial \mathcal A^*(\alpha_1)$ joining $Q_{1,0}$ and $Q_{0,1}$ and an arc $K\subset \partial \mathcal A^*(\alpha_1)$  joining $Q_{1,2}$ and $Q_{2,1}$ belonging to the external  boundary.

\begin{figure}[ht]
    \centering
     \includegraphics[width=0.75\textwidth]{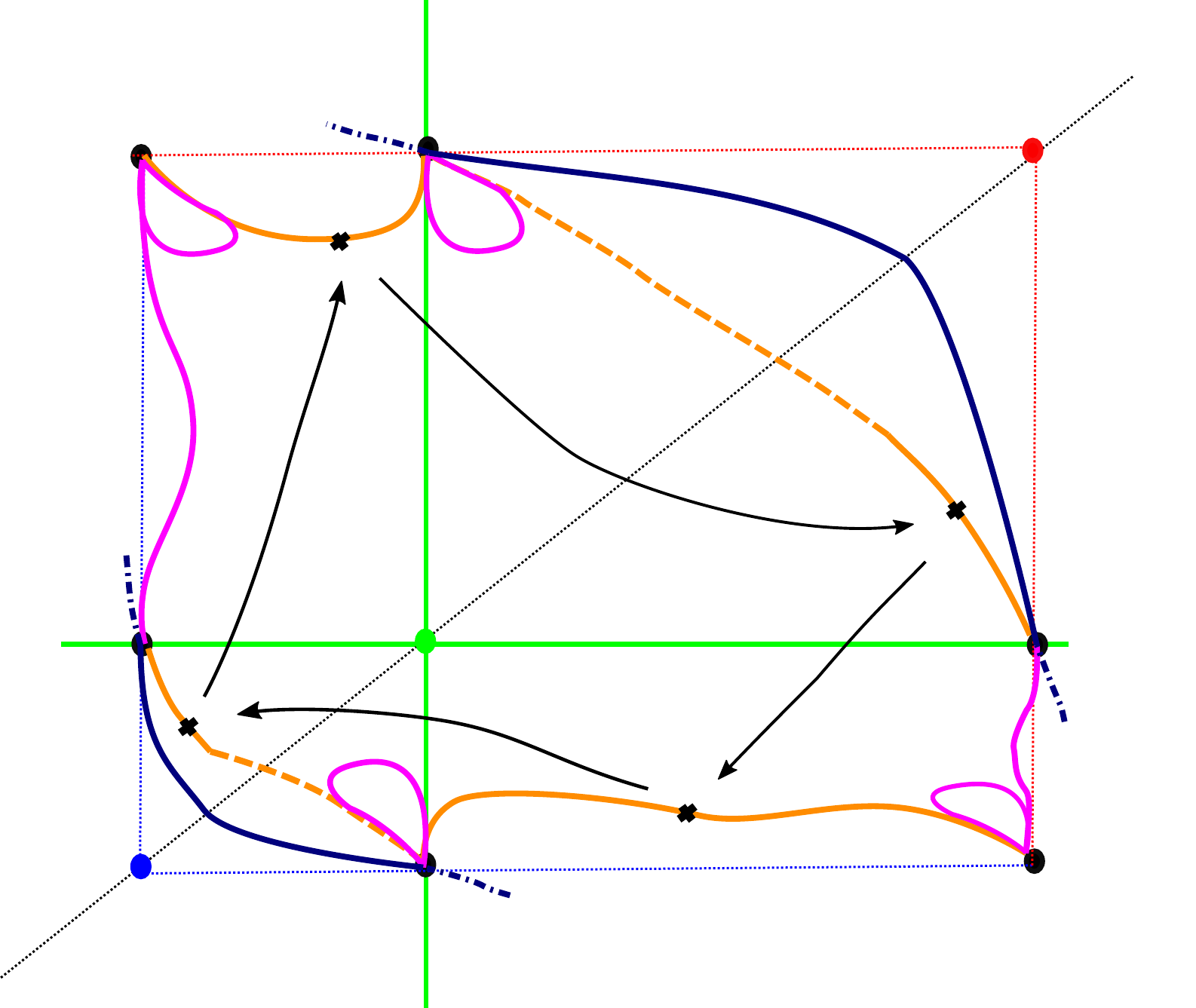}
         \put(-300,25) {\small $(\alpha_0,\alpha_0)$ } 
    \put(-279,51) {\small $+ (c_1,c_1)$ } 
    \put(-320,107) {\small $Q_{0,1}$ } 
    \put(-277,147) {\small $M$} 
    \put(-286,212) {\small $\ell_{0,2}$ } 
    \put(-204,218) {\small $\ell_{1,2}$ } 
     \put(-230,57) {\small $\ell_{1,0}$ } 
    \put(-64,53) {\small $\ell_{2,0}$ } 
    \put(-315,244) {\small $Q_{0,2}$ }
     \put(-210,244) {\small $Q_{1,2}$ }
     \put(-210,25) {\small $Q_{1,0}$ }
     \put(-40,35) {\small $Q_{2,0}$ }
          \put(-65,75) {\small $N$ }
      \put(-40,107) {\small $Q_{2,1}$ }  
      \put(-40,235) {\small $(\alpha_2,\alpha_2)$ }
       \put(-85,205) {\small $+(c_2,c_2)$ }
       \put(-105,220) {\small $\delta_S$ }
      \put(-239,107) {\small $(\alpha_1,\alpha_1)$ }
      \put(-280,76) {\small \bf{ $\zeta$} }
      \put(-295,56) {\small \bf{ $\delta_S$} }
      \put(-267,70) {\small \bf{ $I$} }
       \put(-250,220) {\small \bf{ $S(\zeta)$} }
              \put(-260,203) {\small \bf{ $J$} }
        \put(-82,155) {\small \bf{ $K$} }
                \put(-97,153) {\small \bf{ $r$} }
                \put(-65,140) {\small \bf{ $S^2(\zeta)$} }
         \put(-162,40) {\small \bf{ $S^3(\zeta)$} }
                  \put(-112,58) {\small \bf{ $\Lambda$} }
         \put(-182,180) {\small \bf{ $\mathcal A^*(\alpha_1)$} }
        \put(-162,105) {\small \bf{ $\mathcal S_h$} }
                \put(-215,155) {\small \bf{ $\mathcal S_v$} }
                \put(-213,265) {\small $L_1$ }                       
                \put(-305,165) {\small $L_0$ }
                \put(-43,175) {\small $L_2$ }
           \caption{\small{Sketch of external boundary of the immediate basin of attraction of an internal root  $\alpha_1$. In the picture we can see the six focal points $Q_{i,j}$, where $i,j$ are two different numbers in $\{ 0, 1, 2\}$ and the 4-cycle $\zeta \to S(\zeta) \to S^2(\zeta) \to S^3(\zeta) \to \zeta$.  }}
    \label{fig:period4}
    \end{figure}
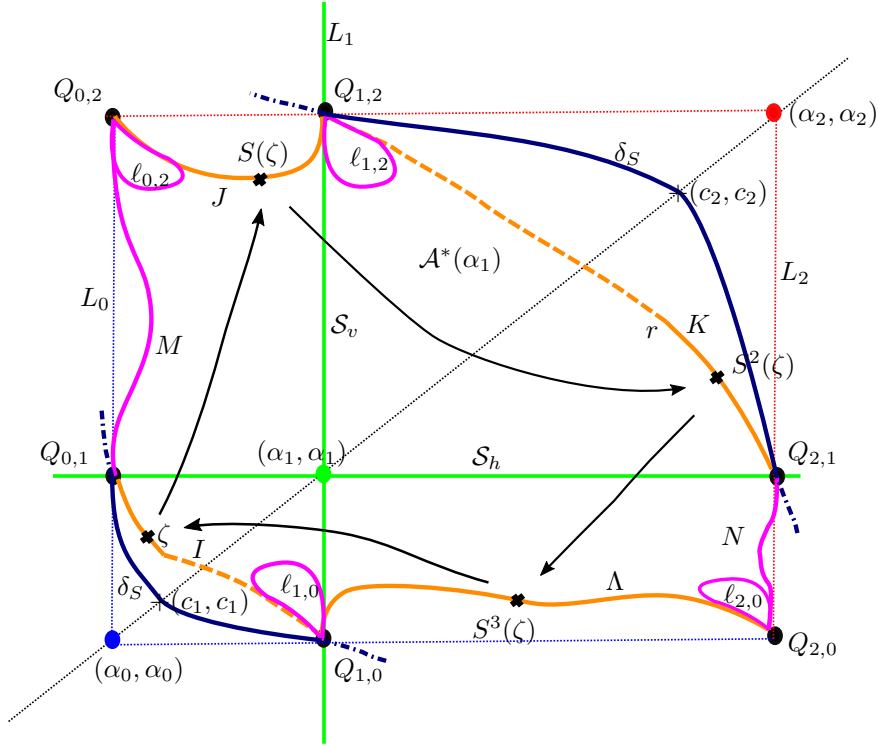

We claim that  $S\left(I\right)$ is an arc $J \subset \partial \mathcal A^*(\alpha_1)$ connecting the focal points $Q_{0,2}$ and $Q_{1,2}$. To see the claim we notice that when $I$ approaches $Q_{0,1}$ (with negative slope by construction; see Figure \ref{fig:period4}) its image should be an arc landing at $L_1 \cap \partial \mathcal A^*(\alpha_1)$. Since $ \mathcal A^*(\alpha_1) \subset R$ and $L_1 \cap R \subset \mathcal A^*(\alpha_1)$ we conclude that the landing point should be either $Q_{1,2}$ or $Q_{1,0}$. Using the one-to-one correspondence defined in 
\eqref{eq:y(m)_Sp} it is clear that the landing point cannot be $Q_{1,0}$ because this corresponds to $m=\infty$. Similarly we can show that  when $I$ approaches $Q_{1,0}$ (again with negative slope by construction) its image should be an arc landing at $Q_{0,2}$. Moreover $J \subset R$ since, by Lemma \ref{lem:bounded}, we have that $\partial R \cap \mathcal A^*(\alpha_1) =\emptyset$. 

Arguing similarly on $K$ instead of $I$ we see that $\Lambda:=S\left(K\right)$ is a smooth arc joining $Q_{1,0}$ and $Q_{2,0}$ entirely contained in $R$  as it is illustrated in Figure \ref{fig:period4}.

Finally since $\mathcal A^*(\alpha_1) \subset R$ and the assumption on the smoothness of  the external boundary of $\mathcal A^*(\alpha_1)$ there should be two arcs, one denoted by $N$ joining $Q_{2,0}$ and $Q_{2,1}$ and another denoted by $M$ joining $Q_{0,1}$ and $Q_{0,2}$, with 
$\{N,M\} \subset \partial \mathcal A^*(\alpha_1) $ belonging to $R$. 

Since $N$ is an arc issuing from two focal points $Q_{2,0}
$ and $Q_{2,1}$, its image $S(N)$ must be an arc
issuing from the prefocal line of the two focal points, which are $L_{0}%
$ and $L_{1}$. Moreover, since the two focal points 
$Q_{2,0}$ and $Q_{2,1}$ belong to the line $L_{2}$
which is mapped into $y=\alpha_{2}$ we have that necessarily the arc
$S(N)$ connects the focal points $Q_{0,2}$ and $Q_{1,2}
$ so that it must be $ J=S(N)=S(I).$ Reasoning in a similar
way we can state that the image of the arc $M$, must be 
$\Lambda=S(M)=S(K)$.

Up to this point we have constructed an hexagon-like polygon without lobes formed by six smooth arcs $I, J, K,\Lambda,M$ and $N$ with vertices at the focal points $Q_{1,0},Q_{0,1},Q_{0,2},Q_{1,2},Q_{2,1}$ and $Q_{2,0}$ contained in  $\partial \mathcal A^{\ast}(\alpha_{1})$. Of course the hexagon (without the vertices) is forward invariant and $I\to J$, $N\to J$, $K\to \Lambda$ and $M\to \Lambda$. Moreover, observe that each curve approaching $Q_{0,1}$ inside the internal sector defined by the arcs $I$ and $M$ (for instance $S_h$) will be sent to a curve through a point in $S_v$ so contained in $A^{\star}\left(\alpha_1\right)$. Hence no curve in this sector might be in $\partial A^{\star}\left(\alpha_1\right)$. Similarly for the focal point $Q_{2,1}$ in the internal sector defined by $K$ and $N$.

%
%
The arc $J$\ is issuing from $Q_{0,2}$ and $Q_{1,2}$ and its image
must be also on the boundary, issuing from points of the prefocal $L_{2}$.
However, its image cannot be the arc $N$ since this would lead a two-cyclic
set implying the existence of a 2-cycle which is impossible. Thus, the arc
issuing from $Q_{0,2}$ and the arc issuing from $Q_{1,2}$ are both mapped into
an arc issuing from $Q_{2,1},$ which means that the image of $J$ is folded on
a portion of the arc $K$, and a folding point $r_{K}$ must exist on $K.$ Similarly for the other arc $\Lambda$, its image is folded on an arc of $I$ issuing from $Q_{0,1}$.

Finally, taking preimages of the arcs $N$ and $M$ we obtain countable many  lobes attached at the four focal points $Q_{1,0}, Q_{2,0},Q_{0,2}$ and $Q_{1,2}$. See Figure  \ref{fig:period4}. We  briefly show the inductive construction of this sequence of  lobes. The arc $M$ connects two points in the prefocal
line $L_{0}$, hence the preimage of $M$ should be given by a
lobe issuing from a related focal point (see the qualitative picture in
Fig.\ref{fig:focals_simple1}(a)). In our case we have two focal points both having the prefocal line
$L_{0},$ which are $Q_{1,0}$ and $Q_{2,0},$ thus we
have two preimages of $M$ giving two lobes issuing from these two
focal points. We denote by $\ell_{1,0}$ and $\ell_{2,0}$ the lobes attached to the focal points $Q_{1,0}$ and $Q_{2,0}$, respectively.  Similarly we can construct the other two lobes (as preimages of $N$) $\ell_{0,2}$ and $\ell_{1,2}$ attached to $Q_{0,2}$ and $Q_{1,2}$.


Now we can take the preimages of the lobe $\ell_{0,2}
,$ since it is issuing from the prefocal line $L_{0}$ its
preimage should be given by a lobe issuing from a related focal point (see the
qualitative picture in Fig.\ref{fig:focals_simple1}(b)). In our case we have two focal points both
having the prefocal line $L_{0},$ which are $Q_{1,0}$ and
$Q_{2,0},$ thus we have two preimages of the lobe $\ell_{0,2}$
giving two lobes issuing from these two focal points, say $\ell_{1,0}^2$ and  $\ell_{2,0}^2$. 

In the same way we can prove the existence of two lobes $\ell_{0,2}^{2}$ and  $\ell_{1,2}^{2}%
$ issuing from the focal points $Q_{0,2}$ and  $Q_{1,2}%
$ as preimages of the lobe $\ell_{2,0}.$ Inductively, each lobe
$\ell_{2,0}^{n}$ issuing from the focal point  $Q_{2,0}$ has
preimages in two lobes $\ell_{0,2}^{n+1}$ and  $\ell_{1,2}^{n+1}$
issuing from the focal points $Q_{0,2}$ and  $Q_{1,2},$ and
$\ell_{0,2}^{n}$ issuing from the focal point $Q_{0,2}$ has
preimages in two lobes $\ell_{1,0}^{n+1}$ and  $\ell_{2,0}^{n+1}$
issuing from the focal points $Q_{1,0}$ and $Q_{2,0}.$
Notice that the lobes issuing from $Q_{1,0}$ and $Q_{1,2}$
have not preimages internal to the immediate basin, because such preimages are
issuing from the focal points $Q_{0,1}$ and $Q_{2,1}$ and we
have shown that lobes cannot exist inside the external boundary detected
above, so that the related preimages must be outside the external boundary. 
\end{proof}


 In Figure \ref{fig:lobes} we show the phase plane of the secant map applied to the Chebychev polynomial $T_3$ near the focal point $Q_{2,0}$. In this picture we can see the lobe $\ell_{20}$ which is a preimage of $M$ (Figure \ref{fig:lobes} left) and the lobe $\ell^2_{2,0}$ attached to the focal point $Q_{2,0}$ with slope equal to $\infty$ (Figure \ref{fig:lobes} right).


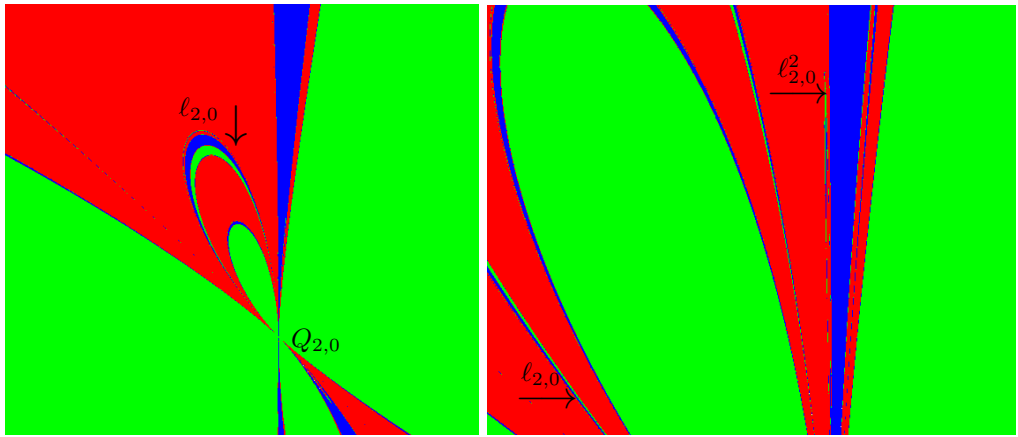
\begin{figure}[ht]
    \centering
    \subfigure[\scriptsize{ The lobe $\ell_{02}$ attached to the focal point $Q_{2,0}$.}]{
     \includegraphics[width=0.4\textwidth]{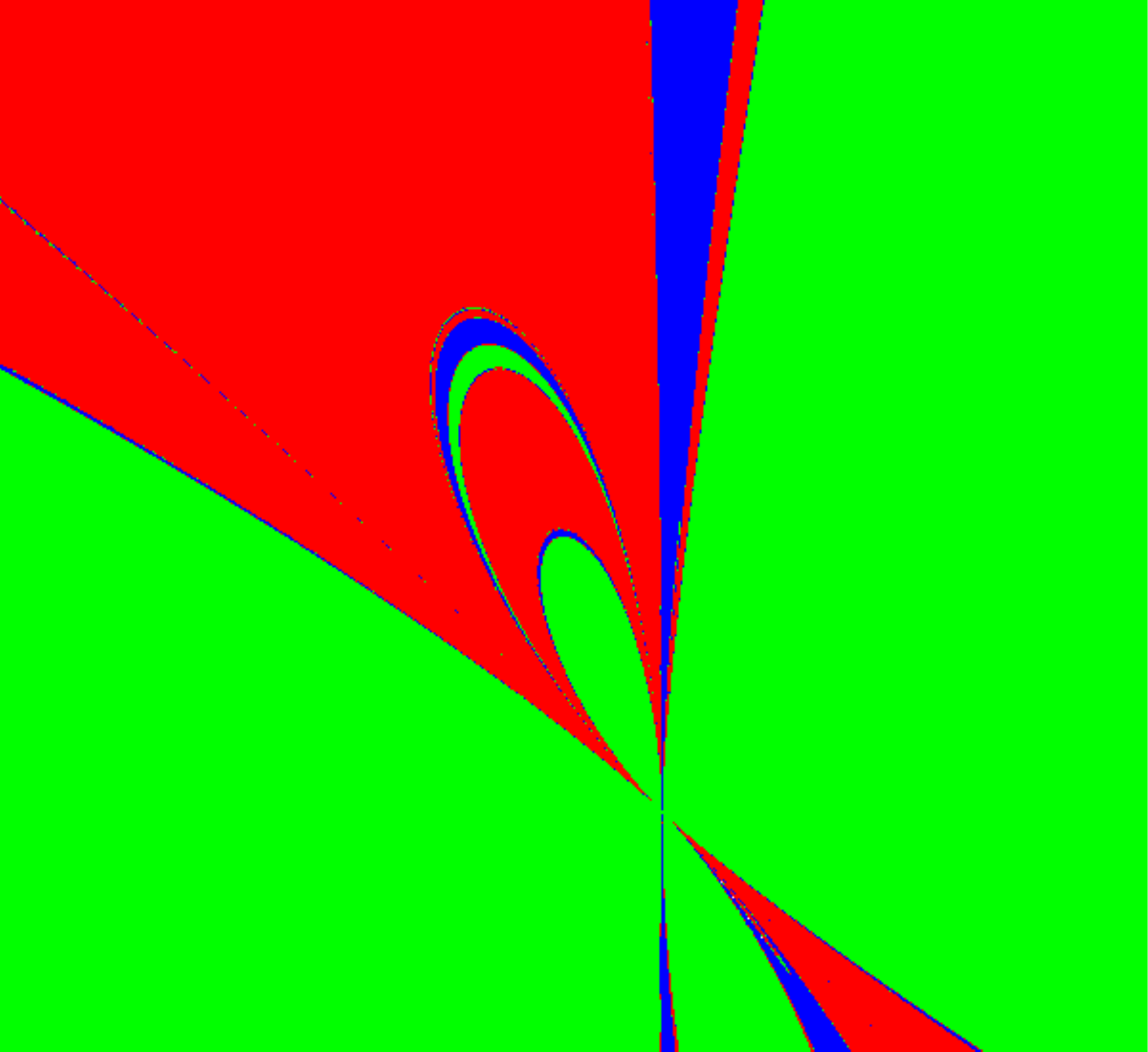}}
      \put(-70,33) {\small $Q_{2,0}$ } 
      \put(-112,120) {\small $\ell_{2,0} $ } 
            \put(-95,112) {\LARGE $\downarrow$ } 
    \subfigure[\scriptsize{The lobe $\ell^2_{2,0}$ attached to the focal point $Q_{2,0}$.}]{
     \includegraphics[width=0.446\textwidth]{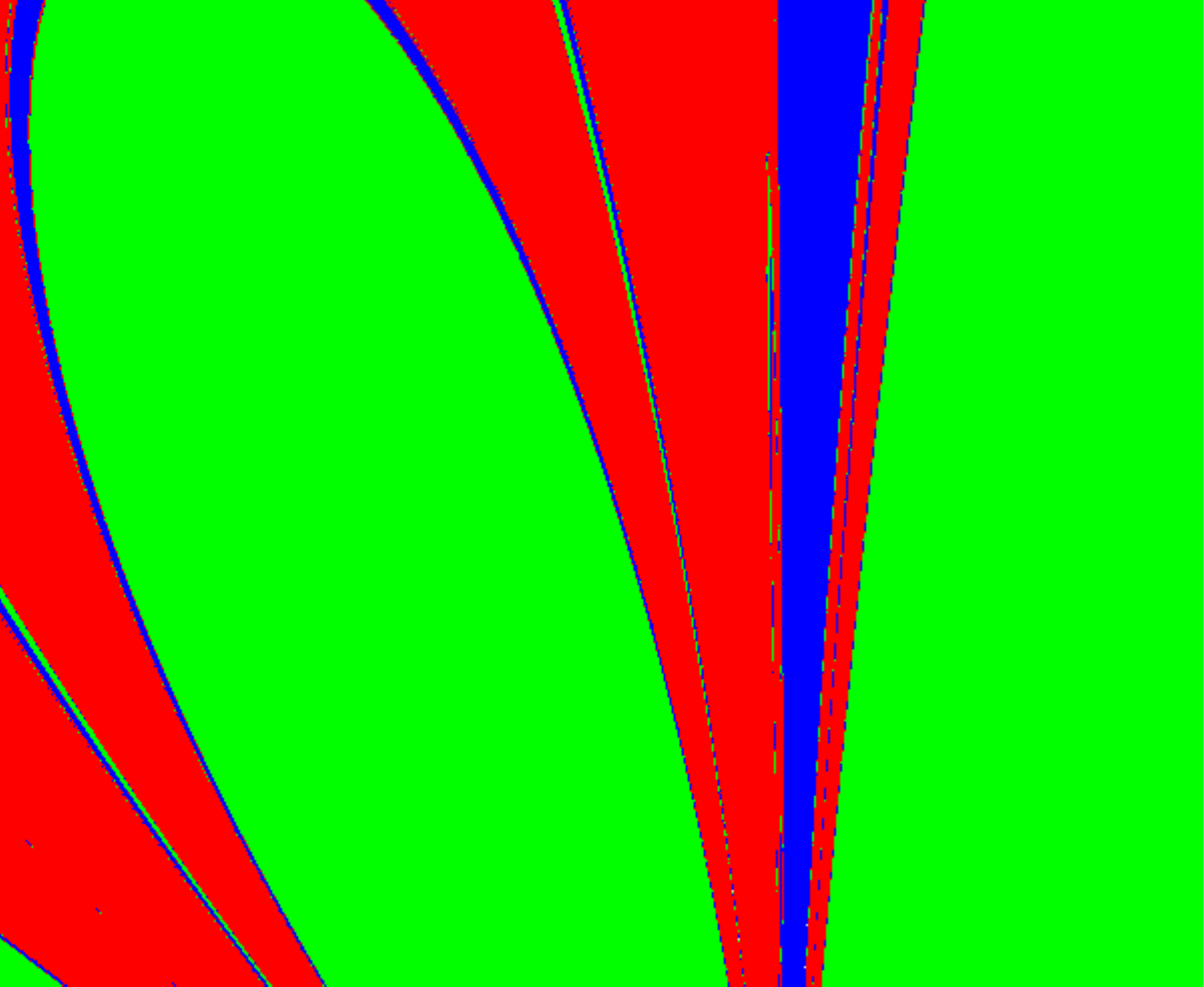}} 
           \put(-185,20) {\small $\ell_{2,0}$ } 
           \put(-187,10) {\Large $\longrightarrow$ }
                      \put(-89,135) {\small $\ell^2_{2,0}$ }
                        \put(-93,125) {\Large $\longrightarrow$} 
 \caption{\small{Enlargement of a portion of the phase plane of the secant map applied to  $T_3(x)= 4x^3-3x$ near the focal point $Q_{2,0}$. }}
    \label{fig:lobes}
    \end{figure}



\begin{corollary}\label{coro:4cycle}
Let $p$ be a polynomial and let
$\alpha_0<\alpha_1<\alpha_2$ be three consecutive real simple roots of $p$. Then there exists a 4-cycle $\mathcal C\in \partial \mathcal A^*(\alpha_1)$ of type I.
\end{corollary}

\begin{proof}
According to the arguments used in the proof of Proposition \ref{prop:boundary_4cycle} we know that for the {\it arc-edge} $I$ of the hexagon-like polygon we have $S^4:I \to I_1\subset I$, where $I_1$ is an arc issuing from the focal point $Q_{0,1}$. Hence there should be a fixed point $\zeta \in I_1$. Of course 
$\mathcal C=\{\zeta,S\left(\zeta\right),S^2\left(\zeta\right),S^3\left(\zeta\right)\}$ is a four cycle of $S$ since each point belongs to a different edge of the hexagon-like border and, from Proposition \ref{prop:configurations}, it is of type $I$. Moreover, we know that on the transverse direction to $I$ the point $\zeta$ should be a repeller (for $S^4$) since the points near $\zeta$ outside $I$ move away from $I$, in particular the ones converging to $(\alpha_1, \alpha_1)$. Hence $\zeta$ is a transversely  repelling point for $S^4$.   
\end{proof}

\begin{remark}
We conjecture that the hypothesis on the smoothness of the external boundary of  $\partial \mathcal A^*(\alpha_1)$ is not needed.
\end{remark}

\begin{remark}
Corollary \ref{coro:4cycle} does not claim that the period 4-cycle is a saddle point of $S^4$. However, we conjecture it is so with one side of its unstable 1-dimensional manifold entering on $\mathcal A^*(\alpha_1)$ and the stable manifold lying on $\partial \mathcal A^*(\alpha_1)$. As an example for this we consider the polynomial $p^{I}$ given in \eqref{eq:pol} and its 4-cycle $\zeta=(1,2) \mapsto (2,3) \mapsto (3,2.44) \mapsto (2.44,1)$. Some computations show that
$$
DS^4(\zeta)\approx\left(\begin{array}{cc}
207.26 & 236.15 \\
242.42 & 276.37
\end{array}\right) 
$$
with eigenvalues $\lambda_1\approx  483.55$ and $\lambda_2\approx 0.05$. So clearly $\zeta$ is a saddle point. Moreover, the corresponding eigenvectors $v_1\approx (-0.65,-0.76)$ and $v_2\approx (-0.75,0.66)$ 
show that the unstable and  stable manifolds (locally) coincide with the mentioned directions. See Figure \ref{fig:type_I}. 
\end{remark}

\begin{figure}[ht]
    \centering
    \subfigure[\scriptsize{ $T_3(x)=4x^3-3x$.}]{
     \includegraphics[width=0.4\textwidth]{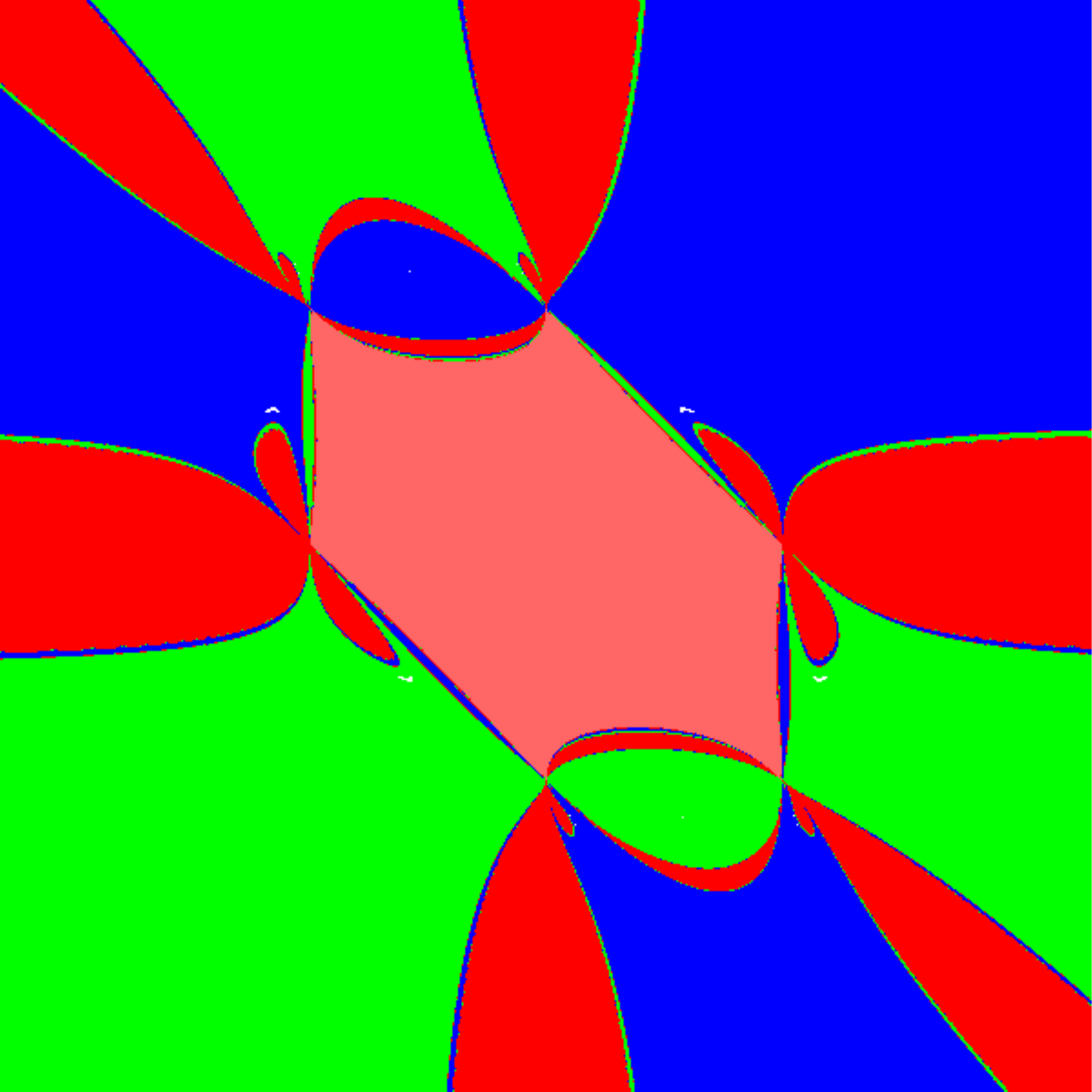}
      \put(-119,77) {\small $\bullet$ }
      \put(-115,97) {\LARGE$\nearrow$ }
            \put(-100,116) {\small $\bullet$ }
                  \put(-90,100) {\LARGE$\searrow$ }
                \put(-86,55) {\small $\bullet$ }
                                  \put(-79,70) {\LARGE$\swarrow$ }
                     \put(-63,93) {\small $\bullet$}
                                                       \put(-107,70) {\LARGE$\nwarrow$ }
      }
    \subfigure[\scriptsize{$p(x)=\frac{x5}{5}-\frac{x^3}{3}-0.05x+0.15$.}]{
     \includegraphics[width=0.4\textwidth]{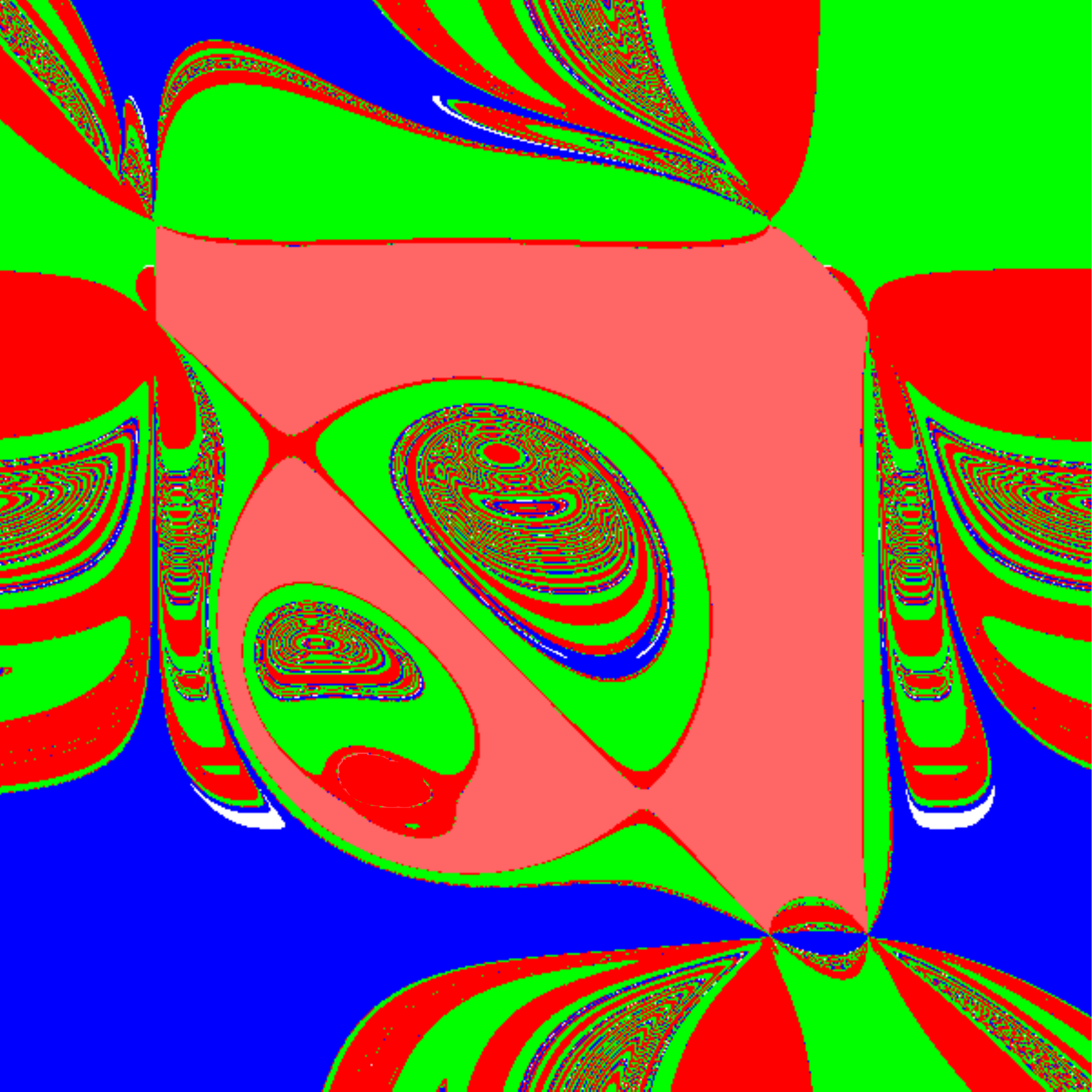}
           \put(-75,125) { $\nearrow$ }
      \put(-59,135) {\small $\bullet$ }
            \put(-55,127) {$\searrow$ }
                        \put(-45,75) {\LARGE $\downarrow$ }
            \put(-42,126) {\small $\bullet$ }
                                    \put(-68,38) {\Large $\nwarrow$ }
                \put(-51,28) {\small $\bullet$ }
                                                    \put(-138,109) {$\nwarrow$ }
                     \put(-144,114) {\small $\bullet$}
                                          \put(-135,115) {\Large$\rightarrow$}
                     }
 \caption{\small{ Phase plane of the secant map applied to the Chebyshev polynomial $T_3$ (left) and to the polynomial $p$ (right).  We show the immediate basin of attraction of the internal root in pink. We also mark the 4-cycle contained in the boundary of the immediate basin proved in Theorem A. Range of the pictures  [-2,2]x[-2,2]. }}
    \label{fig:holes_4cycle}
    \end{figure}
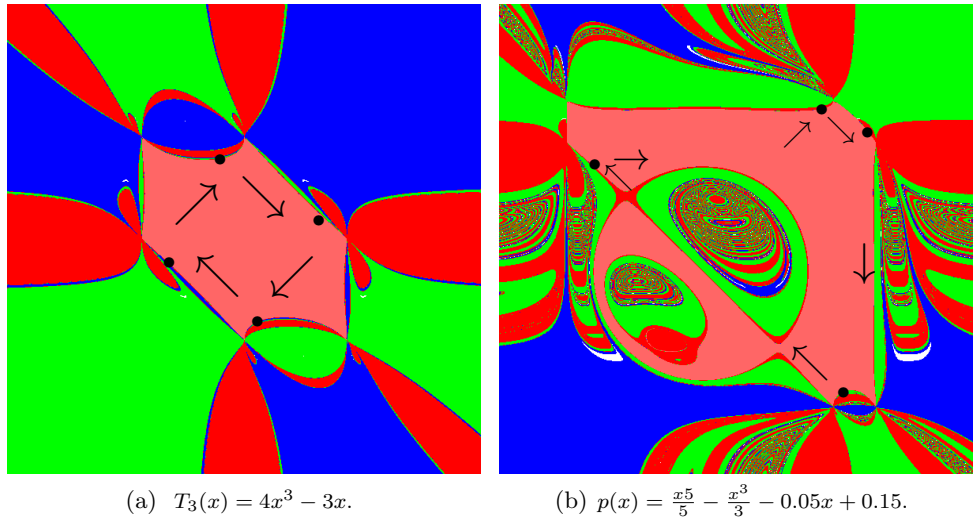

\begin{proof}[Proof of Theorem A] Statement (a) follows from Proposition \ref{prop:boundary_4cycle} while statement (b) follows from Corollary \ref{coro:4cycle}.
\end{proof}

 We notice that Theorem A applies independently of the connectedness of the immediate basin of attraction of the internal root. We present two examples to focus on this fact. We consider the phase space of the secant map applied to the Chebyshev polynomial $T_3(x)= 4x^3-3x$, see Figure \ref{fig:holes_4cycle} (left), in contrast with the phase space of the secant map applied to the polynomial $p(x)=\frac{x5}{5}-\frac{x^3}{3}-0.05x+0.15$, see Figure \ref{fig:holes_4cycle} (b). In both cases the two polynomials exhibit three simple root, and thus in both cases there exist a unique internal root. In pink we show the immediate basin of attraction of the internal root. In the case of $T_3$ the immediate basin of attraction is simply connected while in the case of $p$ the immediate basin is multiply connected.  Moreover, we numerically compute the 4-cycle contained in the boundary of the immediate basin of attraction as Theorem A states. Every point in the 4-cycle is depicted in the phase plane with a small black circle. Finally, we mention that Theorem A only deals with the external boundary of the immediate basin of attraction of the internal root.  In the next section we precisely focus on sufficient conditions which ensure that the immediate basin of attraction of an internal root is simply connected.
 
\section{Proof of Theorem B}\label{section:TB}

As in the previous section we assume, without lost of generality,  that $\alpha_0 < \alpha_1 <\alpha_2$ are three consecutive real simple roots of $p$ and $p'(\alpha_0)>0$, $p'(\alpha_1)<0$ and $p'(\alpha_2)>0$.  We denote by  $R, H_{y_0}$ and $V_{y_0}$ the  following open sets 
\[
\begin{array}{ll}
R & =\{(x,y)\in \mathbb R^2 \ ; \ \alpha_0<x<\alpha_2, \, \alpha_0<y<\alpha_2\}, \\
H_{y_0}  & = \{ (x,y_0) \in \mathbb R^2 \, ; \, \alpha_0 < x <  \alpha_2\},   \\
V_{y_0} & =\{(y_0,y)\in \mathbb R^2 \, ; \, \alpha_0 < y <  \alpha_2\}.
\end{array}
\]
Moreover, we introduce the auxiliary map 
\begin{equation}\label{eq:varphi}
\varphi_y(x)= y - \frac{p(y)}{q(x,y)}
\end{equation}
which coincides with the second component of the secant map; i.e.,  $S(x,y)=(y,\varphi_y(x))$ where remember that the polynomial $q$ was defined in \eqref{eq:q}.

We now investigate the connectedness of the basin of attraction of an internal root $\alpha_1$.  In the next lemma we count the number of inverses of the secant map for a given point $(x,y)\in R$. In particular this lemma will apply to points in $\mathcal A^{\star}\left(\alpha_1\right)$ (see Lemma \ref{lem:bounded}). 

\begin{lemma}\label{lem:inverse}
Let $p$ be a polynomial and let  $\alpha_0 < \alpha_1 < \alpha_2$ be three consecutive simple real roots of $p$. Assume further that $p$ has only one inflection point in the interval $(\alpha_0, \alpha_2)$.  Then for  any point  $(x,y)\neq  (\alpha_1,\alpha_1)$ in $R$  we have  that $\#\{S^{-1}(x,y)\}\leq 2$ where $S^{-1}$ means preimages of $(x,y)$ in $R$. 
\end{lemma}

\begin{proof}

We reason by contradiction. We assume that there exists $(x_1,y_1)\in R$ with three different preimages in $R$, say $(w_0,x_1), \, (w_1,x_1)$ and $(w_2,x_1)$ so that  $S(w_i,x_1) = (x_1,y_1)$ with $i=0,1,2$. Renaming these points if necessary we can assume that $w_0 < w_1 < w_2$. Let $r$ be the line passing  through $(x_1,p(x_1))$ and 
$(y_1,0)$.  By construction the points  $(w_i,p(w_i)),\ i=0,1,2$  belong to $r$.  Thus, the line $r$ contains the points $(x_1,p(x_1))$ and $(w_i,p(w_i)), \ i=0,1,2$ and this implies the existence of at least two inflection points of $p$ in the interval defined by $\beta_0:=\min\{x_1,w_0\}$ and $\beta_2:=\max\{x_1,w_2\}$ with $[\beta_0,\beta_2]\subset (\alpha_0,\alpha_2)$, a contradiction with the assumptions.  See Figure  \ref{fig:inverse}.


\end{proof}

\begin{figure}[ht]
    \centering
     \includegraphics[width=0.40\textwidth]{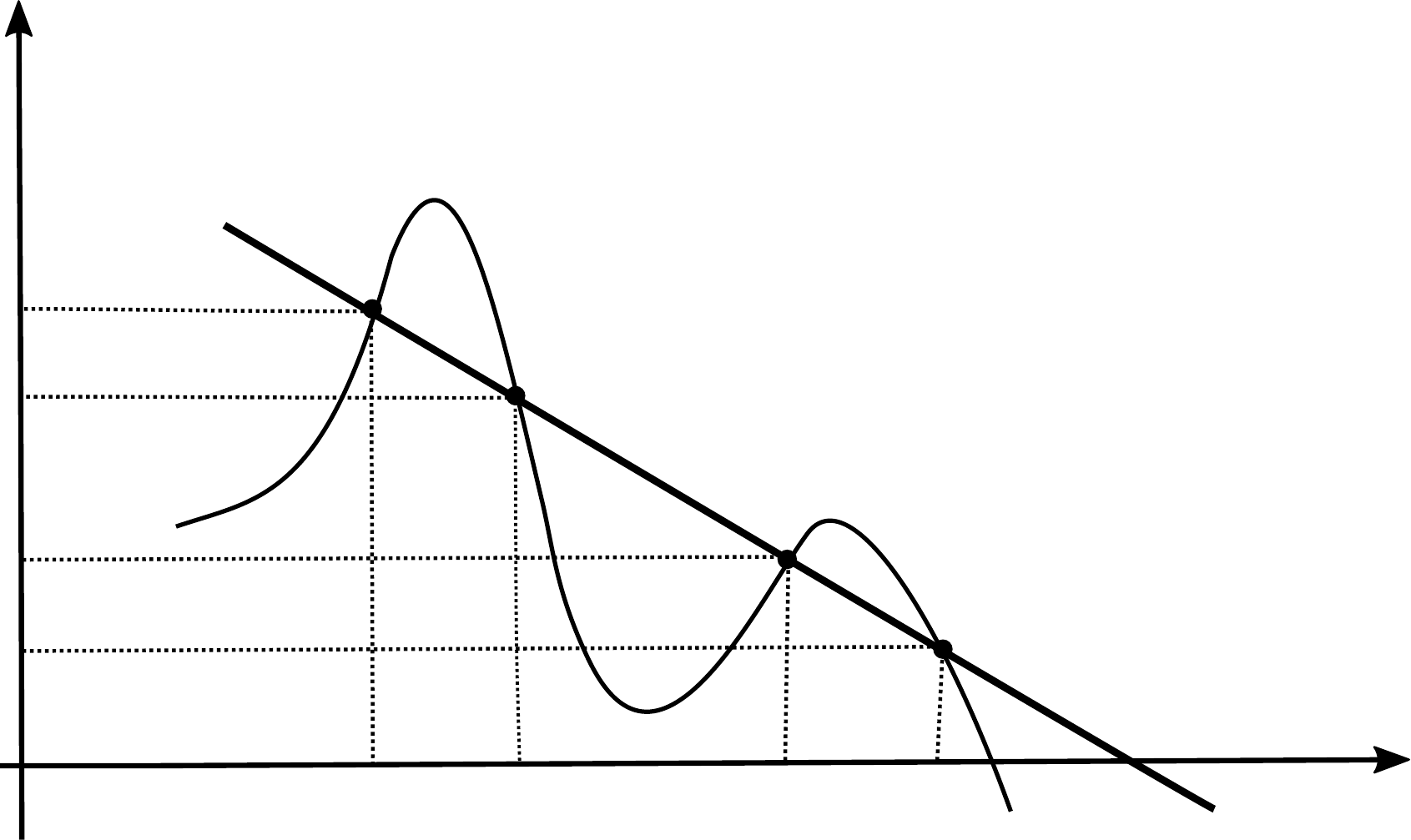}
    \put(-84,3) {\small $w_1$ } 
    \put(-65,3) {\small $w_2$ } 
    \put(-136,3) {\small $w_0$ } 
    \put(-40,3) {\small $y_1$ }
    \put(-116,3) {\small $x_1$ }
    \put(-148,78) {\small $r$ }
     \put(-198,20) {\footnotesize $p(w_2)$}
      \put(-198,33) {\footnotesize $p(w_1)$}
      \put(-198,55) {\footnotesize $p(x_1)$}
       \put(-198,65) {\footnotesize $p(w_0)$}
             \caption{\small{Sketch of the inverses of the point $(x_1,y_1)$.}}
    \label{fig:inverse}
    \end{figure}
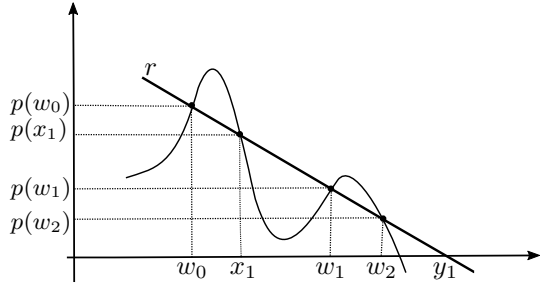

\begin{lemma}\label{lem:vert}
Let $p$ be a polynomial and let  $\alpha_0 < \alpha_1 < \alpha_2$ be three consecutive simple real roots of $p$. Assume further that  $p$ has only one inflection point in $(\alpha_0, \alpha_2)$ and let $y_0 \in  (\alpha_0, \alpha_2)$. Then the set $J_{y_0}:= S(H_{y_0}) \cap \overline{R}$ is a closed vertical segment belonging to 
$\overline{V_{y_0}}$ and, if $y_0\ne \alpha_1$, any point of $J_{y_0}$ has two preimages in $H_{y_0}$, counting multiplicity. Moreover, 
\begin{itemize}
\item[(a)] if $y_0 < \alpha_1$ then $J_{y_0}= [\varphi(x_0^*(y_0)), \alpha_2]$, 
\item[(b)] if $y_0 > \alpha_1$ then $J_{y_0}= [\alpha_0,\varphi(x_0^*(y_0))]$,
\item[(c)] if $y_0 = \alpha_1$ then $J_{y_0}= [ \alpha_1,\alpha_1]$ (degenerate closed interval),
\end{itemize}
\noindent where  $x_0^*(y_0)$ is the unique point in $H_{y_0}$ such that $\frac{\partial q}{\partial x }(x_0^*(y_0),y_0)=0$.
\end{lemma}

\begin{proof} 
First remember that $\varphi_y(x)$ is defined in \eqref{eq:varphi} as the second component of the secant map. Therefore we already know that if $y_0=\alpha_1$ then $\varphi_{y_0}(x)\equiv \alpha_1$ and (c) follows. In what follows we fix a concrete value of $y_0 \in (\alpha_0,\alpha_2)$ with $y_0\ne \alpha_1$. On the one hand from the expression of the secant map we have that  $S(H_{y_0}) \cap \overline{R}$ is a closed vertical segment $J_{y_0}:=\left[a,b\right] \subset \overline{V_{y_0}}$.  On the other hand it is a direct computation to see that  
\begin{equation*}
\varphi'(x)=\frac{ p(y_0)}{q^2(x,y_0)} \frac{\partial q}{\partial x }(x,y_0) \quad {\rm and} \quad \frac{\partial q}{\partial x} (x,y_0)= \frac{p'(x)-q(x,y_0)}{x-y_0}.
\end{equation*}
Observe from   \eqref{eq:q} that $q(x,y)$ is a polynomial and simple computations show that  when $x=y_0$ the second formula becomes $\partial q / \partial x (y_0,y_0)=p''(y_0)/2$. Hence $\varphi^{\prime}(x)$  vanishes if and only if   $p^{\prime}(x)-q(x,y_0)=0$ for $x \neq y_0$, or $p''(y_0)=0$ if $x=y_0$. 

As already said, the focal points $Q_{i,j}\in \delta_S$, the
set of non definition of map $S$ (now we focus on $i\ne j\in\{0,1,2\}$). Moreover, it is easy to argue that the points $(c_{k},c_{k}),\ k=1,2$ also belong to $\delta_S\cap R$. It follows that an arc of $\delta_{S}$ must exist in $R$ connecting the points $Q_{0,2}$, $(c_{1},c_{1})$ and
$Q_{1,0}$ (as qualitatively shown in Figure \ref{fig:period4}), so that for $\alpha
_{0}<y_{0}<\alpha_{1}$ the graph of $\varphi(x)$ has a
vertical asymptote for $x\in(\alpha_{0},\alpha_{1}).$ Similarly, an
arc of $\delta_{S}$ must exist in $R$ connecting the points
$Q_{1,2}$, $(c_{2},c_{2})$ and $Q_{2,1}$ (as
qualitatively shown in Figure \ref{fig:period4}), so that for $\alpha_{1}<y_{0}<\alpha_{2}
$ the graph of $\varphi(x)$ has a vertical asymptote for
$x\in(\alpha_{1},\alpha_{2}).$

We claim that there exists a unique  point $x_0^\star:=x_0^\star(y_0)$  in $(\alpha_0,\alpha_2)$ verifying  $\partial q  / \partial x (x_0^\star,y_0)=0$ (i.e., verifying that $\varphi^{\prime}(x_0^\star)=0$). See Figure  \ref{fig:varphi}. To see the claim we consider first $y_0\in(\alpha_0,\alpha_1)$. From the above paragraph we know there exists $\tilde{y}_0$ such that $\varphi |_{\left(\tilde{y}_0,\alpha_2\right)}$ satisfies
\begin{equation*}
\displaystyle \lim_{x \to \tilde{y}_0^+}  \varphi(x) = + \infty , \quad \varphi(\alpha_1) = \alpha_1 \quad {\rm and} \quad \varphi(\alpha_2)=\alpha_2.
\end{equation*}
Hence, since $\varphi |_{(\tilde{y}_0,\alpha_1)}$ is smooth, we might conclude using Bolzano that there exist $x_0^\star \in (\alpha_1,\alpha_2)$ such that $\varphi^{\prime}\left(x_0^\star\right)=0$ (in fact it is a local minimum of $\varphi$). Uniqueness follows from the fact that there is no further change of convexity for the polynomial $p$.  It remains to check the case $y_0\in(\alpha_1,\alpha_2)$. In this case the previous paragraph indicates the existence of $\tilde{y}_0$ such that $\varphi |_{\left(\alpha_0,\tilde{y}_0\right)}$ satisfies
\begin{equation*}
\displaystyle \lim_{x \to \tilde{y}_0^-}  \varphi(x) = - \infty , \quad \varphi(\alpha_0) = \alpha_0 \quad {\rm and} \quad \varphi(\alpha_1)=\alpha_1.
\end{equation*}
Arguing similarly we find that there exist $x_0^\star \in (\alpha_0,\alpha_1)$ such that $\varphi^{\prime}\left(x_0^\star\right)=0$ (in fact it is a local maximum of $\varphi$), and uniqueness is due to the non existence of further convexity changes.

 \begin{figure}[ht]
    \centering
    \subfigure[\scriptsize{Phase space of the secant map of the Chebyshev polynomial $T_3(x)=4x^3-3x$. Range of the picture $[-1,1]\times[-1,1]$.}]{
\includegraphics[width=0.45\textwidth]{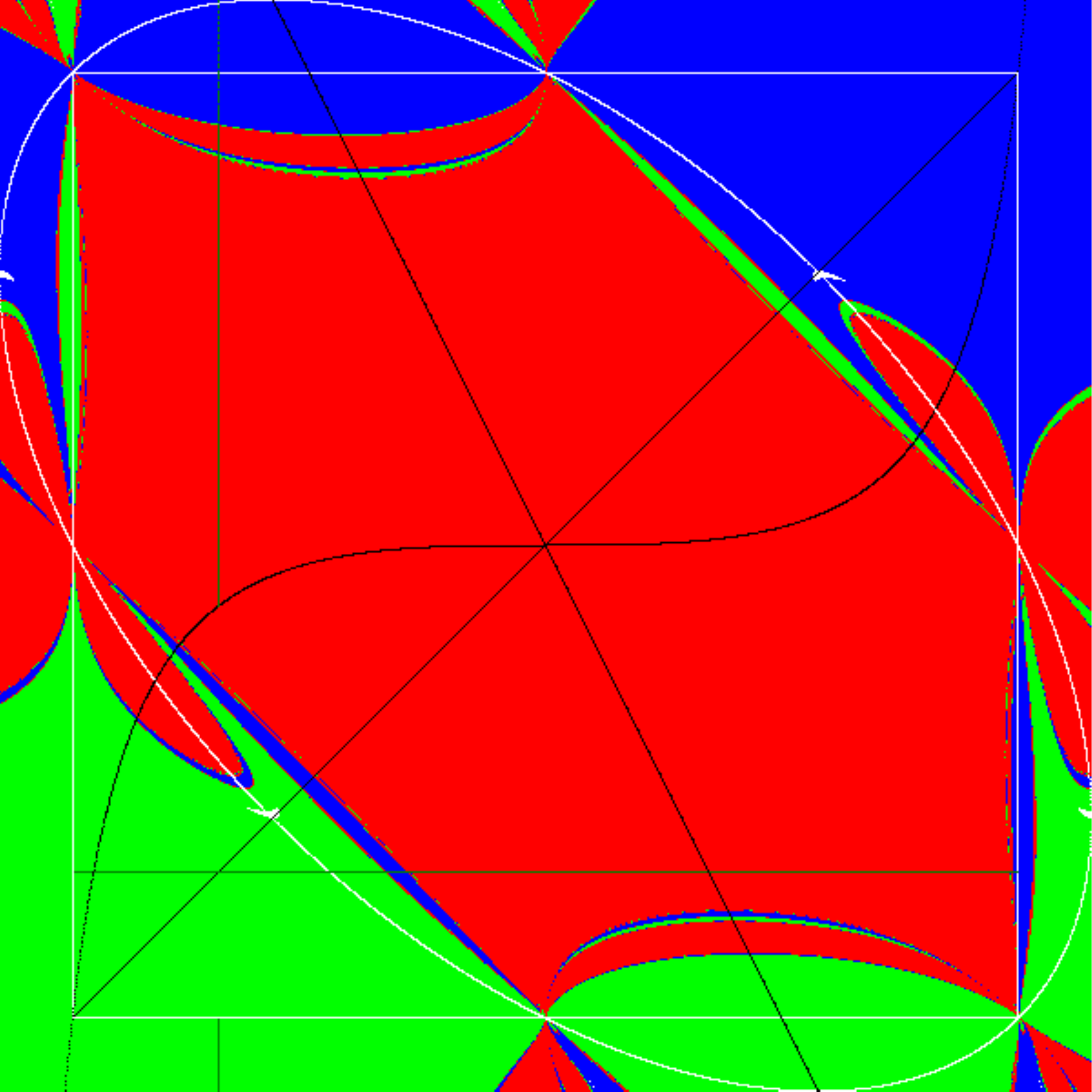}}
          \put(-199,40){\scriptsize $H_{y_0}$}
          \put(-155,135){\small \w{$J_{y_0}$}}
          \put(-73,38){\small \w{$ \bullet \, x^*$}}
           \put(-165,88){\footnotesize \w{ $ \bullet \, \varphi(x^*)$}}
          \put( -68,167){\small \w{$\delta_S $}}
          \put( -122,151){\small \w{$q_x=0 $}} 
                   \put( -145,25){\small \w{$\delta_S $}}
           \put( -45,175){\small \w{$R$}}
           \put( -45,95){\small \w{$\Gamma$}}
        \subfigure[\scriptsize{Graph of $\varphi(x) = y_0 - T_3(y_0)/ q(x,y_0)$.}]{
\includegraphics[width=0.59\textwidth]{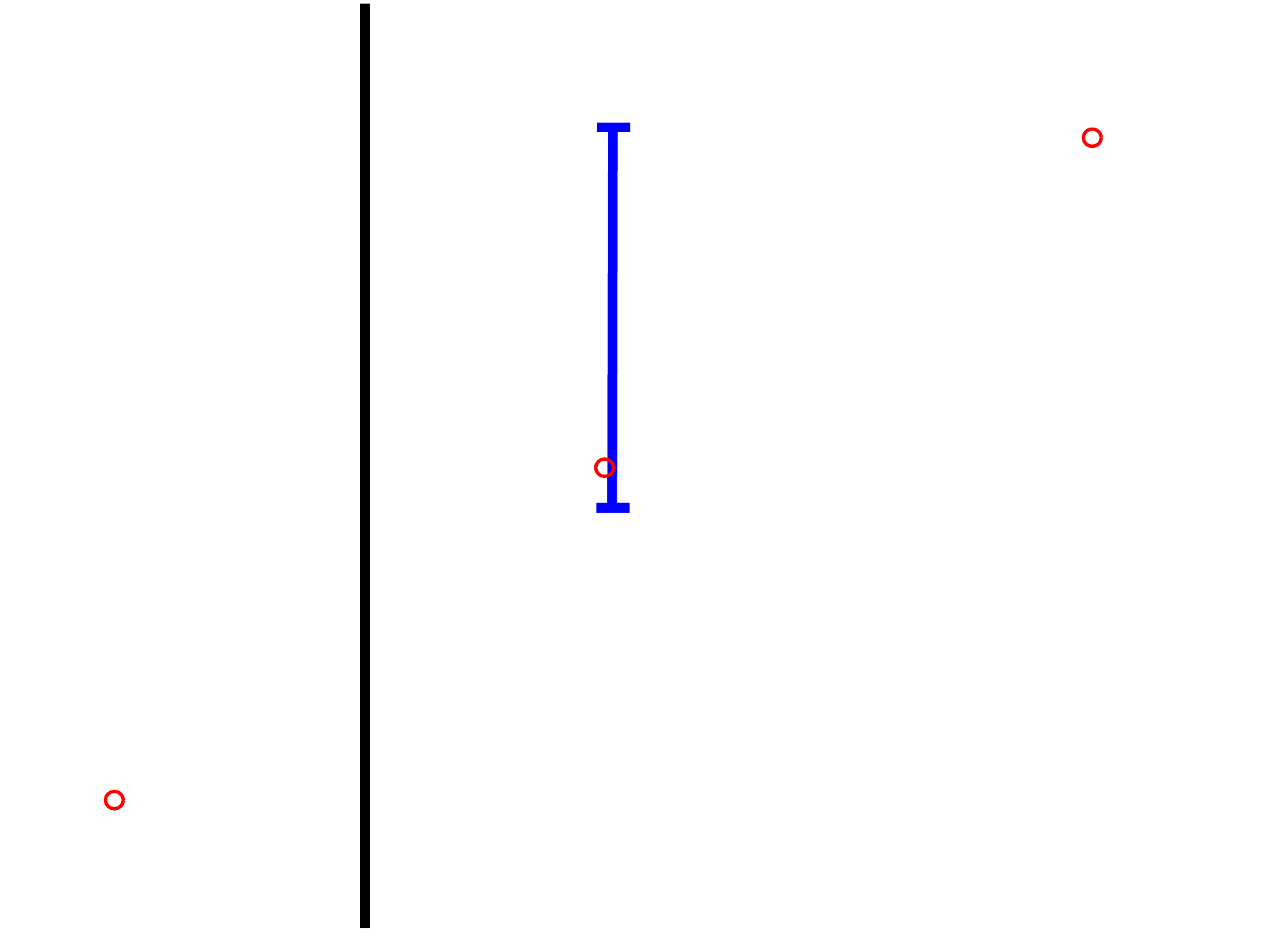}} 
           \put( -100,88){\small $\bullet$} 
                      \put( -95,80){\small $(x^*, \varphi(x^*))$} 
                      \put( -35,150){\small $y= \varphi(x)$} 
                       \put( -130,135){\small {\b  $J_{y_0}$}} 
 \caption{\small{The secant map applied to $T_3$. In (a) it is shown the line $H_{y_0}$ with $y=y_0$ in the interval $\left(\alpha_0,\alpha_1\right)$. In (b) we show the graph of the related function $\varphi_{y_0}(x)$ when $y_0\in \left(\alpha_0,\alpha_1\right)$.}}
    \label{fig:varphi}
    \end{figure}
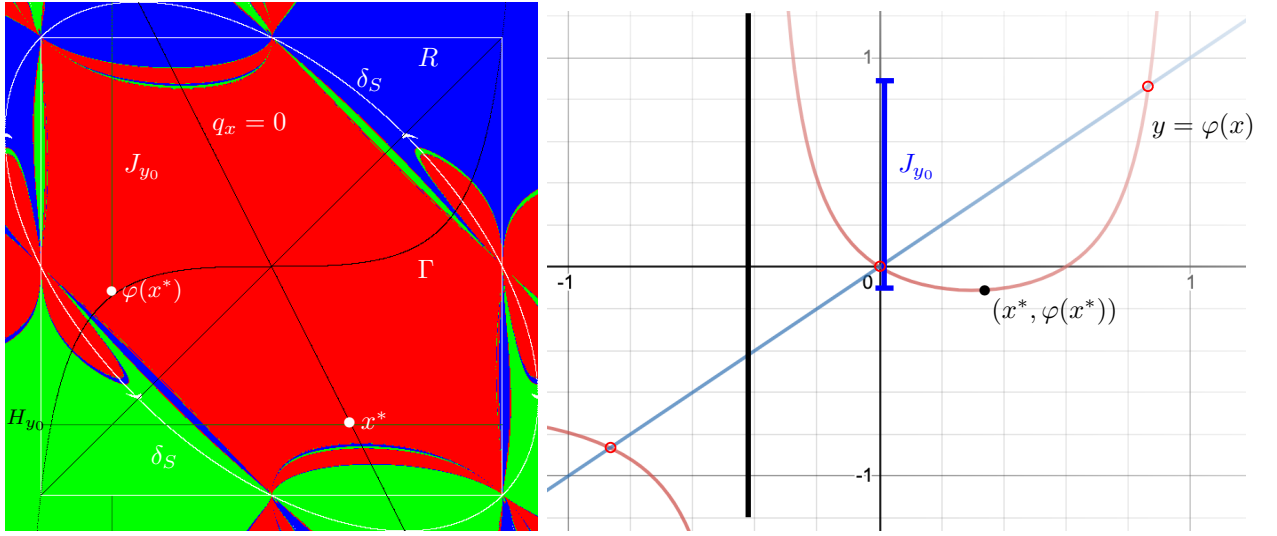

From the description performed so far,  we clearly know how acts $S$ inside $R$, concretely fixing a value of $y_0 \in (\alpha_0,\alpha_2)$ there exist a unique point  $x_0^*:= x_0^*(y_0) \in (\alpha_0,\alpha_2)$ solution of $\partial q / \partial x=0$  and  the map $S(\cdot,y_0)$ is monotone for $x\in (\alpha_0, x_0^*)$ and monotone for $ x \in (x_0^*, \alpha_2) $ with a turning point at $x_0^*$ and a vertical asymptote. So the lemma follows.  
\end{proof}

For the shake of clarity we exemplify Lemma \ref{lem:vert} with the study of one particular case. In Figure \ref{fig:varphi} (left) we show the phase space of the secant map applied to the  Chebyshev polynomial of degree $T_3(x)=4x^3-3x$. The three roots of $T_3$ are given by 
$\alpha_0= -\sqrt{3}/2$, $\alpha_1=0$ and $\alpha_2=\sqrt{3}/2$.  It is easy to compute explicitly all the dynamical objets appearing in the above proof since the degree of the polynomial is three. Thus for example we can compute  the set of no definition of the secant map $\delta_S$  \eqref{eq:delta_s}  given by  
$$q(x,y)=  \frac{T_3(x) -T_3(y)}{x-y} = 4x^2+4xy+4y^2-3=0, $$
\noindent which is an ellipse, and the set of points where $q_x:= \partial q / \partial x =0$ is given by the line $y =  -2x$.  Moreover, given  $ y_0  \in ( \alpha_0, \alpha_1)$ we can  evaluate explicitly $x_0^*(y_0)= - y_0/2$  and we can compute the image of $H_{y_0}$ under the map $\varphi: (\alpha_0,\alpha_2) \mapsto \mathbb R$

$$
\varphi(x)= y_0 - \frac{p(y_0)}{q(x,y_0)}= y_0 - \frac{4y_0^3-3y_0}{ 4 x^2 + 4 y_0 x + 4y_0^2 -3}.
$$
The map $\varphi$, see  Figure \ref{fig:varphi} (right), exhibits in the interval $(\alpha_0,\alpha_2)$,  a vertical asymptote at the point $\tilde{y}_0\in\left(\alpha_0,\alpha_1\right)$ in which the line $y=y_0$ intersects $\delta_S$ given by
$$
\tilde{y}_0:= -\frac{ y_0 +  \sqrt{3 (1-y_0^2)} }{2},
$$ 
and has a local minimum at the point $(x_0^*(y_0),\varphi(x_0^*(y_0)):=(-y_0/2, \varphi(-y_0/2))$. Obtaining thus that $J_{y_0}=  S(H_{y_0}) \cap \overline{R} =  [\varphi(-y_0/2),\alpha_2].$ 

Next technical lemma is the last result we need to proof Theorem B. Its content gives further information about the sets 
\begin{equation}
\begin{split}
&{\rm LC}_{-1}:=\{(x,y)\in R \ |\ x \neq y, \,  DS(x,y)=0 \}, \\
& {\rm LC}:=\{S(x,y) \ |\ (x,y)\in {\rm LC}_{-1} \}.
\end{split}
\end{equation}
In particular ${\rm LC}$ is the set of points where we cross from regions where points have either zero or two preimages in $R$. From the definition  of the secant map it is easy to see that $S(x,x)=(x,N_p(x))$, where $N_p(x):=x-p(x)/p^{\prime}(x)$ is the {\it Newton's map} associated to $p$.

\begin{lemma}\label{lemma:theta_gamma}
Let $p$ be a polynomial and let  $\alpha_0 < \alpha_1 < \alpha_2$ be three consecutive simple real roots of $p$. Assume  that  $p$ has only one inflection point, denoted by $\gamma_0$, in the interval $(\alpha_0, \alpha_2)$.  Then the following statements hold
\begin{itemize}
\item[(a)] The set ${\rm LC}_{-1} \cup \{(\gamma_0,\gamma_0)\}$ is given by 
$$
\Theta=\{(x,y)\in R \ | \   y\in \left(\alpha_0,\alpha_2\right),\ x\ne y, \  \ x=x^{\star}(y) \} \cup \{ (\gamma_0,\gamma_0)\},
$$
where $x=x^{\star}(y)$ is the unique point such that $q_x(x^{\star}(y),y)=0$ with $x\ne y$ unless $y=\gamma_0$ for which $x^\star(\gamma_0)=\gamma_0$. So it can be written as the graph of a function $y \mapsto \Theta(y),\ y\in \left(\alpha_0,\alpha_2\right)$ . Moreover $\Theta(y)$ is strictly decreasing.
\item[(b)] 

The set $\Gamma: =  {\rm LC} = S(\Theta)$ is given by the graph of $N_p$ evaluated at the point $x^{\star}(y)$. Equivalently,
\begin{equation} \label{eq:Gamma}
\Gamma=\{ \left(y ,N_p(x^{\star}(y) \right)  \,  |  \, (x^*(y),y)\in \Theta  \}.
\end{equation}
Let $\xi $ be such that  $\Theta \cap \{x=\alpha_1\} = (\alpha_1,\xi)$, then
\begin{itemize}
\item If $\gamma_0 <\alpha_1$ then $\Gamma$ has a local minimum at $(\gamma_0,N_p(\gamma_0))$ and a local maximum at $(\xi,\alpha_1)$.
\item If $\gamma_0 >\alpha_1$ then $\Gamma$ has a local maximum at $(\gamma_0,N_p(\gamma_0))$ and a local minimum at $(\xi,\alpha_1)$. 
\item If $\gamma_0=\alpha_1$ then $\Gamma$ is strictly increasing and has an inflection point at $(\alpha_1,\alpha_1)$.
\end{itemize}

\item[(c)] The points $\left(\gamma_0,N(\gamma_0)\right)$ and $(\xi,\alpha_1)$ belong to $\mathcal A^\star(\alpha_1)$. 
\end{itemize}
\end{lemma}

\begin{proof}
Without lost of generality we assume $p^{\prime}\left(\alpha_0\right)>0$ and so $p|_{(\alpha_0,\alpha_1)}>0$ and $p|_{(\alpha_1,\alpha_2)}<0$. 

The first part of statement (a) follows from Lemma \ref{lem:vert}. So, only remains to prove that $\Theta(y)$ is strictly decreasing. This fact is easy by drawing qualitatively the graph of $p$ in the interval $(\alpha_0,\alpha_2)$ under, of course, the assumption of a unique inflection point. To be more precise observe that on the one hand if $y\ne \gamma_0$ then  $q_x(x^{\star}(y),y)=0$ if and only if
\begin{equation}\label{eq:x_star}
p^{\prime}\left(x^{\star}(y)\right)=\frac{p(y)-p\left(x^{\star}(y)\right)}{y-x^{\star}(y)},
\end{equation}
and on the other hand if $y=\gamma_0$ then $q_x\left(\gamma_0,\gamma_0\right)=p''(\gamma_0)/2=0$, since $\gamma_0$ is the unique inflection point in $(\alpha_0,\alpha_2)$. So $x^{\star}(\gamma_0)=\gamma_0$.
In other words for all $y\ne \gamma_0$, the point $x^{\star}(y)$ is the unique point such that $x^{\star}(y)\ne y$ and its tangent line (to the graph of $p$) coincides with the secant line between the points $(y,p(y))$ and $(x^{\star}(y),p(x^{\star}(y)))$ (see Figure \ref{fig:Theta}). Hence if we take a point $y\in (\alpha_0,\gamma_0$), since $p$ is concave in this interval and convex in $(\gamma_0,\alpha_2)$, it immediately follows that $x^{\star}(y) \in (\gamma_0,\tilde{c}_2)$ where $\tilde{c_2}\in(\alpha_1,c_2)$ corresponds to the solution of  \eqref{eq:x_star} for $y=\alpha_0$. The tangent line at $(x^*(y),p(x^*(y)))$ is below the graph of $p$ since $p$ is convex in the interval $(\gamma_0,\alpha_2)$.
Moreover, if $y$ increases towards  $\gamma_0$ then $x^{\star}(y)$ decreases towards $\gamma_0$.  See Figure \ref{fig:Theta}. In the case that  $y\in (\gamma_0,\alpha_2)$  the polynomial $p$ is convex in this interval with $x^{\star}(y) \in (\tilde{c}_1, \gamma_0)$ where $\tilde{c_1}\in(c_1,\alpha_1)$. Moreover, when $y$ decreases from $\alpha_2$ towards $\gamma_0$ then $x^*(y)$ increases from $\tilde{c}_1$ towards $\gamma_0$. Summarizing the closure of the curve  $\Theta(y)$ is an analytic curve joining the points $(\alpha_0,\tilde{c}_2)$ and $(\alpha_2,\tilde{c}_1)$ and being decreasing on $y$.

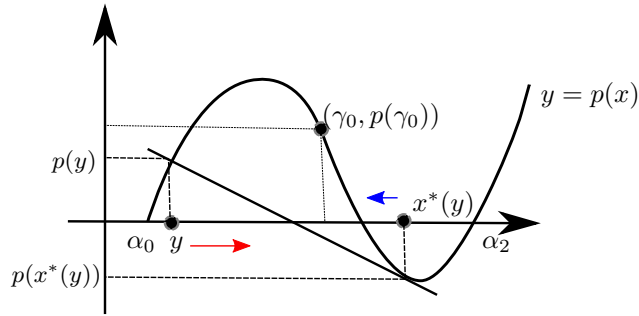
\begin{figure}[ht]
    \centering
     \includegraphics[width=0.40\textwidth]{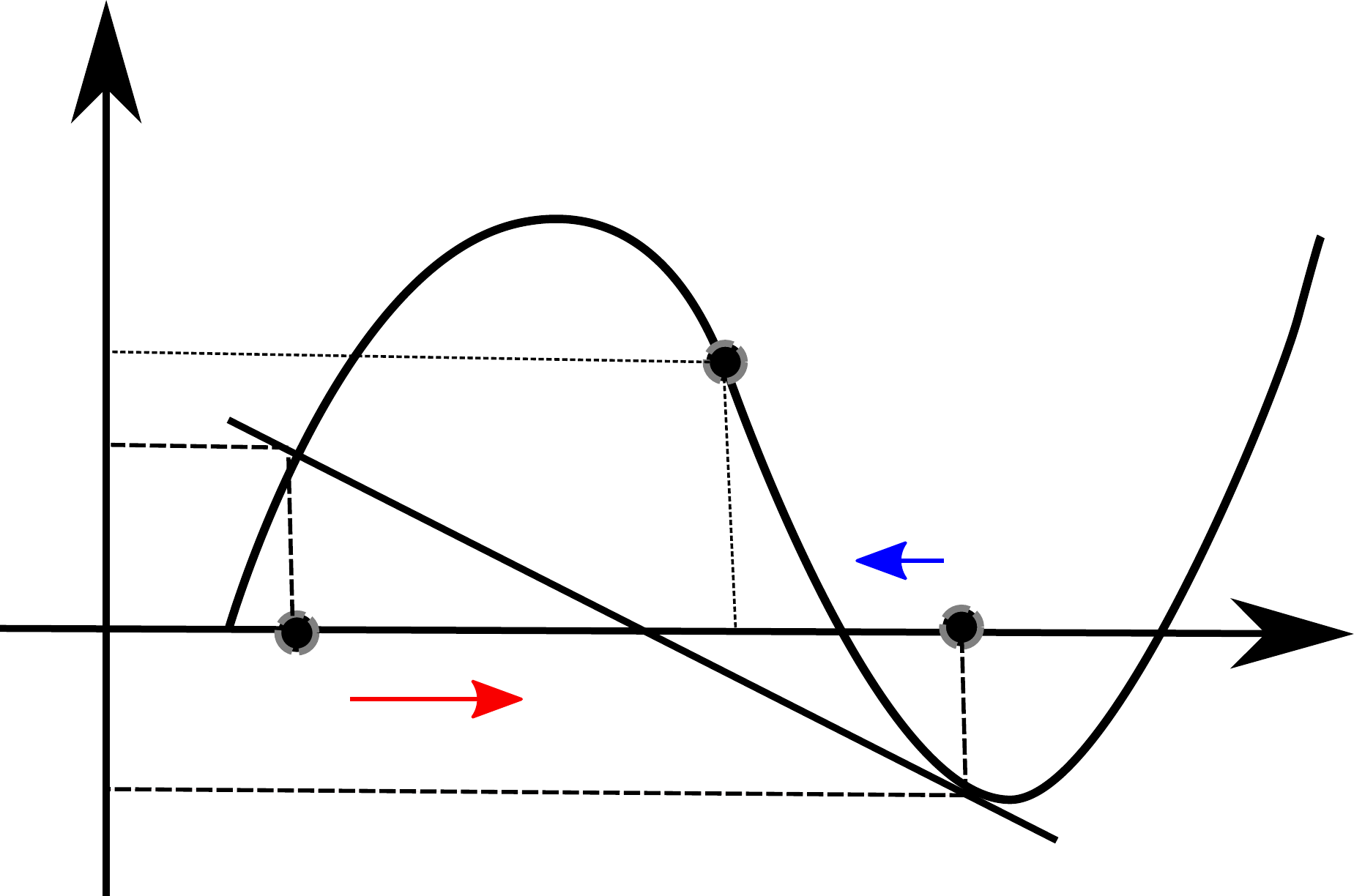}
    \put(-82,72) {\small $(\gamma_0,p(\gamma_0))$ } 
        \put(-155,25) {\small $\alpha_0$ } 
        \put(-22,25) {\small $\alpha_2$ } 
    \put(-139,25) {\small $y$ } 
    \put(-48,39) {\small $x^*(y)$ }
    \put(0,80) {\small $y=p(x)$ }
     \put(-198,12) {\footnotesize $p(x^*(y))$}
      \put(-183,55) {\footnotesize $p(y)$}
             \caption{\small{Sketch of the dependence of $x^*(y)$  with respect to  $y$, when $y \in (\alpha_0,\gamma_0)$.}}
    \label{fig:Theta}
    \end{figure}

We turn the attention to statement (b), that is the study of $\Gamma:=S(\Theta)$. Take a point $\left(x^{\star}(y),y\right)$. Its image is given by 
$$
S\left(x^{\star}(y),y\right)=\left(y,y-\frac{p(y)}{q\left(x^\star(y),y\right)}\right)=\left(y,x^\star(y)-\frac{p(x^\star(y))}{q\left(x^\star(y),y\right)}\right)=\left(y,N_p\left(x^{\star}(y)\right)\right),
$$
where the second equality follows from the general fact that the secant line passing through the points $(z,w)$ and $(w,z)$ is the same and the polynomial $q$ is symmetric. But now we can take advantage of the fact that $q\left(x^\star(y),y\right)=p^{\prime}\left(x^\star(y)\right)$ since $q(x,y)$ is the slope of the secant line through $x$ and $y$ and $x^\star(y)$ is precisely the point where this slope coincide with the derivative (as slope) of $p$ at $x^\star(y)$. Thus we conclude \eqref{eq:Gamma}.

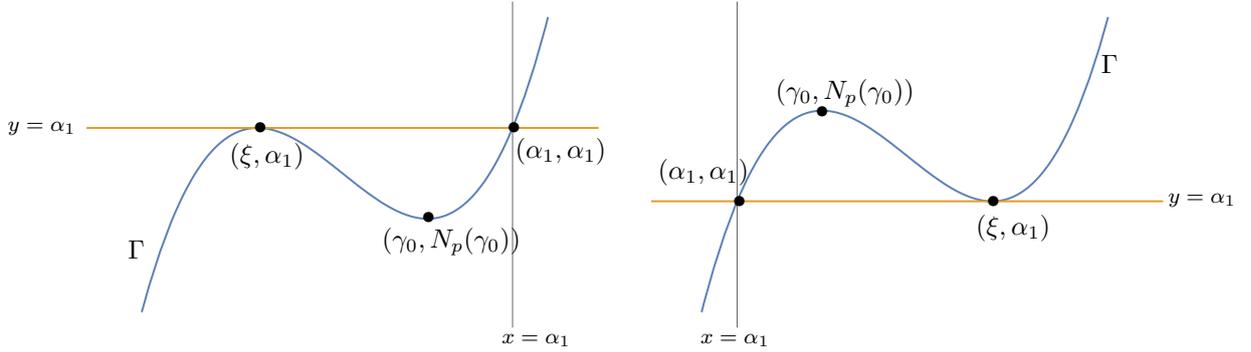
\begin{figure}[ht]
    \centering
    \subfigure{
\includegraphics[width=0.45\textwidth]{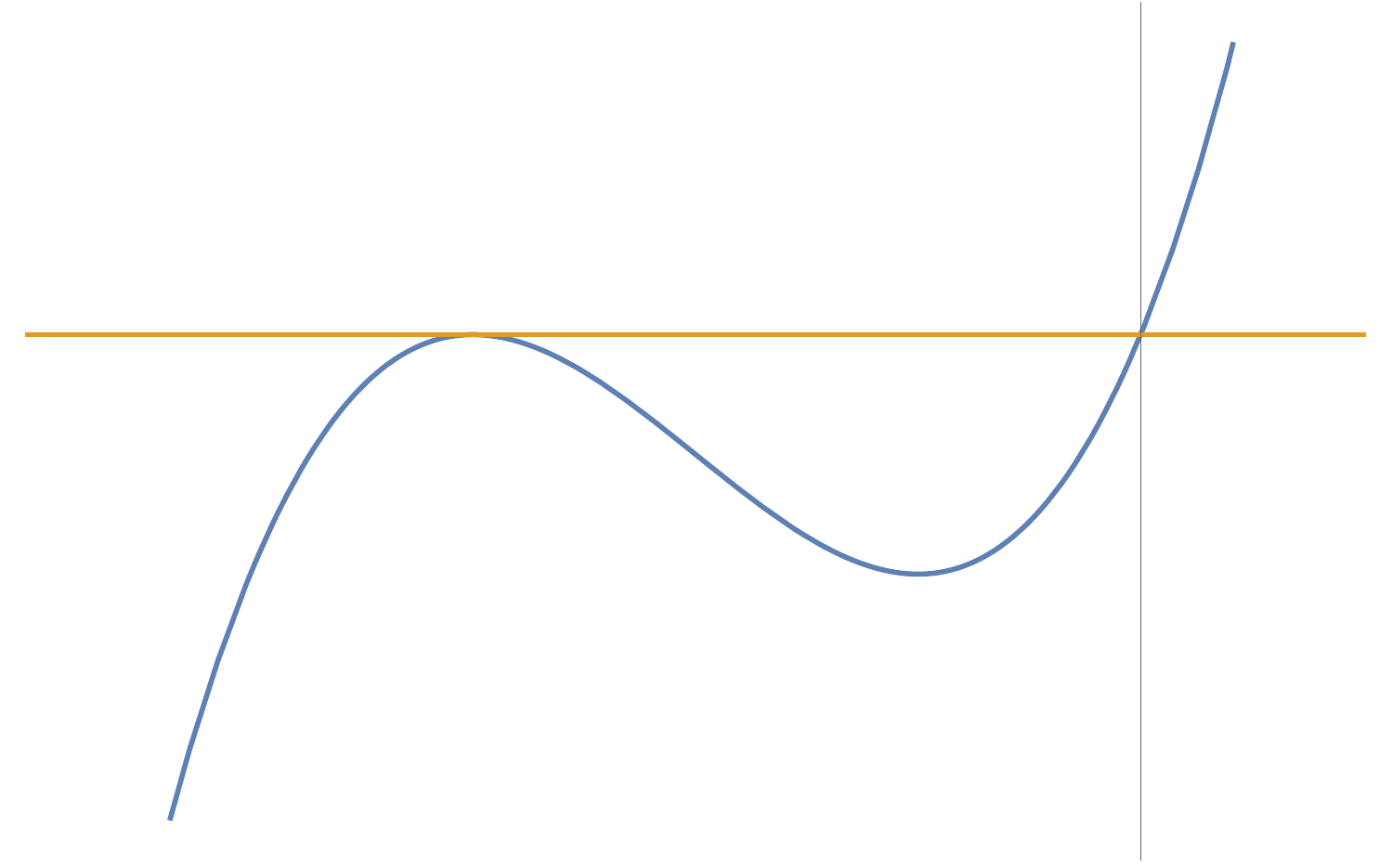}}
    \put(-225,75){\scriptsize $y=\alpha_1$}
       \put(-40,-5){\scriptsize $x=\alpha_1$}
     \put( -70,40){\small $\bullet$}
      \put( -38,74){\small $\bullet$}
      \put( -133,74){\small $\bullet$}
       \put( -142,63){\small $(\xi,\alpha_1)$}
      \put( -35,65){\small $(\alpha_1,\alpha_1)$}
      \put( -85,31){\small $(\gamma_0,N_p(\gamma_0))$} 
            \put( -180,27){\small $\Gamma$}
         \hglue 0.3truecm
        \subfigure{
\includegraphics[width=0.45\textwidth]{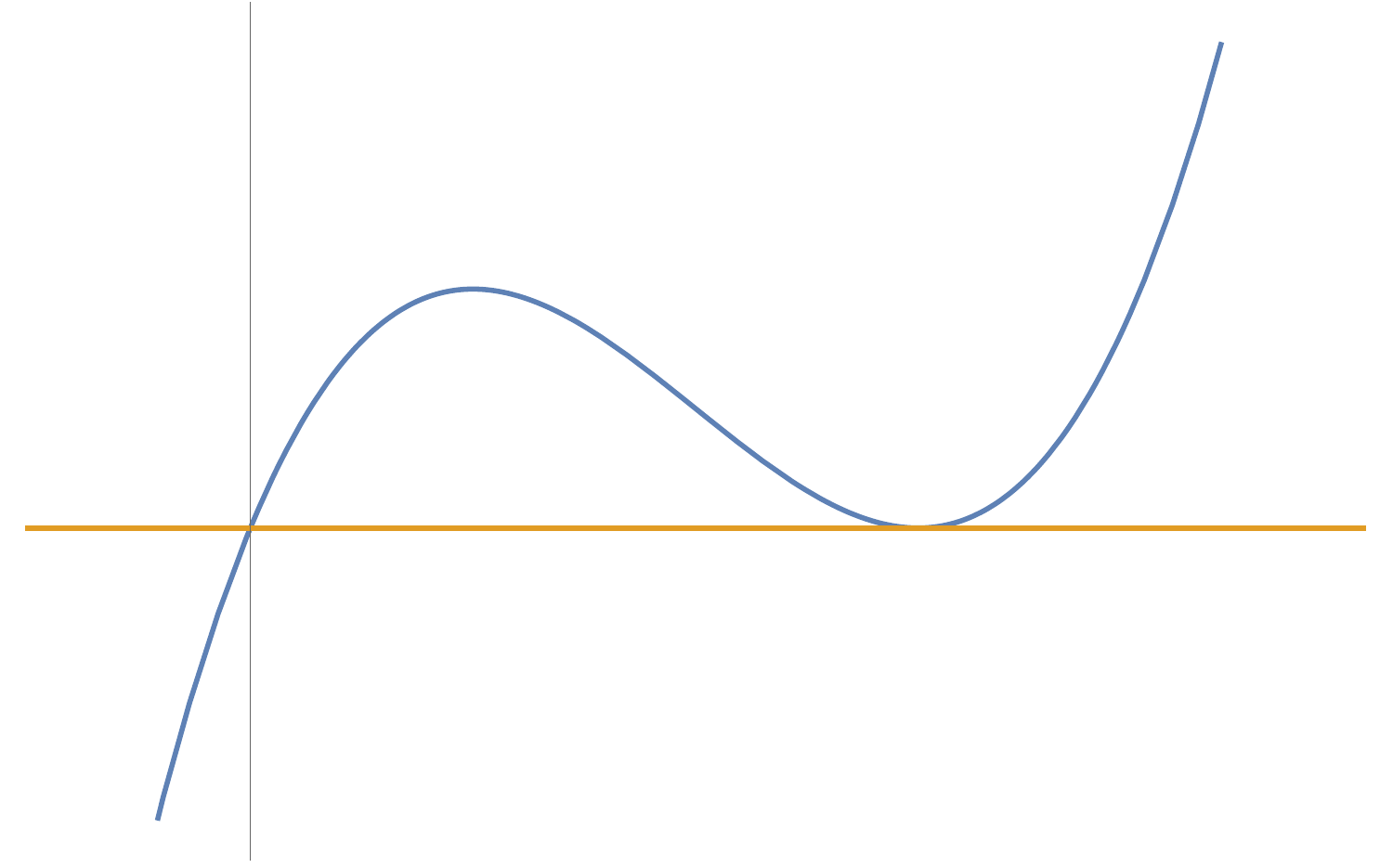}} 
    \put(-2,48){\scriptsize $y=\alpha_1$}
       \put(-177,-5){\scriptsize $x=\alpha_1$}
     \put( -165,46){\small $\bullet$}
      \put( -134,80){\small $\bullet$}
      \put( -70,46){\small $\bullet$}
           \put( -74,36){\small $(\xi,\alpha_1)$}
      \put( -149,87){\small $(\gamma_0,N_p(\gamma_0))$}
      \put( -193,57){\small $(\alpha_1,\alpha_1)$}  
      \put( -27,97){\small $\Gamma$}
 \caption{\small{Qualitative draw of $\Gamma$ depending on the relative position of $\gamma_0$ and $\alpha_1$. On the left hand side
  we  when $\gamma_0 < \alpha_1$ and on the right hand side when $\gamma_0>\alpha_1$. The case when $\gamma_0=\alpha_1$ both local extrema collide forming an 
  inflection point. See Figure \ref{fig:varphi} (a).}}
    \label{fig:Gamma}
    \end{figure}
    
The second part of statement (b) follows from the computation
$$
N_p'(x)=\frac{p(x)p^{\prime\prime}(x)}{(p^{\prime}(x))^2},$$ 
thus $ x \mapsto N_p(x)$ exhibits two local extrema at $x=\alpha_1$ and $x=\gamma_0$  in  the open interval $(\alpha_0,\alpha_2)$ since $N'_p$ changes it sign at $\alpha_1$ and at $\gamma_0$. Simple computations show that if $\gamma_0 < \alpha_1$ then $N_p$ has a local minimum at $\gamma_0$ and a local maximum at $\alpha_1$, and if $ \gamma_0>\alpha_1$ then the local minimum occurs at $\alpha_1$ and the local maximum at $\gamma_0$. Now using the  one-to-one relationship between $y$ and $x^\star(y)$, we obtain that $\Gamma$ has two local extrema  at $y=\gamma_0$ and $y=\xi$ since  $x^\star(\gamma_0)=\gamma_0$ and $x^\star(\xi)=\alpha_1$. If $\gamma_0<\alpha_1$ then $\Gamma$ exhibits a local maximum at $\xi$ and a local minimum at $\gamma_0$,  if $\gamma_0>\alpha_1$ then the local maximum appears at $\gamma_0$ and the local minimum at $\xi$. Summarizing the closure of $\Gamma$ is an analytic curve joining the points $(\alpha_0,\alpha_0)$ and $(\alpha_2,\alpha_2)$ inside $R$.

It remains to show (c). We easily have  that $S(\xi,\alpha_1)=(\alpha_1,\alpha_1)$ and  we observe that $S(\gamma_0,\gamma_0)=(\gamma_0,N_p(\gamma_0))$. Assume $\gamma_0<\alpha_1$ (the case $\gamma_0>\alpha_1$ follows similarly). Consider the interval 
$[\gamma_0,\alpha_1]$. Since $p$ is concave in this interval we can  deduce that no mater the pair of initial conditions on this interval, the secant map produces a point which  is much closer to $\alpha_1$ than whose predecessors and always being a point in $[\gamma_0,\alpha_1]$. Hence the whole square $[\gamma_0,\alpha_1]\times[\gamma_0,\alpha_1]$ belongs to the immediate basin.  
\end{proof}

%


With all these in our hands we can start the proof of Theorem B, which is somehow a direct consequence of the previous results. 

\begin{proof}[Proof of Theorem B] We reason by contradiction. Let us assume, under
the assumptions, the existence of an immediate basin $\mathcal A^{\ast}(\alpha_{1})$
not simply connected. This means that internal
to the region bounded by the external frontier that we have described in
Proposition 4.2 there exists at least one internal region $U$ whose points are mapped
outside the immediate basin. Let us consider one point $y_{0} \in (\alpha_0,\alpha_2)$ such that the  horizontal line $H_{y_{0}}$ intersects $U$.  Hereafter can assume that $y_0 \neq \alpha_1$ since, except at focal points, the rest of the  points  in  $y = \alpha_1$ are in $\mathcal A(\alpha_1)$.  There must be three open segments  in $H_{y_0}$, as shown in Figure \ref{fig:U}, where $B \cap  \mathcal A^{\star}\left(\alpha_1\right)=\emptyset$. 


We first show that for the considered $y_0\in (\alpha_0,\alpha_2)$ the point $x^\star:=x^\star\left(y_0\right) \in B$. Let us assume $y_0\in (\alpha_0,\alpha_1)$. We know that, as the value $x$ increases on $H_{y_0}$ over the segment $\overline{ABC}$ the shape of $\varphi_{y_0}(x)$ should be first decreasing and then increasing with the change of monotonicity occurring at the unique critical point $x^{\star}:=x^{\star}\left(y_0\right)$. Clearly $x^\star$ cannot be in the segment $A$ because as $x$ increases the image of $\varphi_{y_0}(x)$ must necessarily be first decreasing, then increasing, and finally decreasing again which is impossible according to Lemma \ref{lem:vert}. A similar argument shows that $x^\star \not\in C$. Hence $x^\star\in B$, or, equivalently in $B\cap \Theta\ne \emptyset$.  

We consider the image of the region $R$ by the map $S$, i.e., $S(R)$, and set $\tilde{R}=S(R)\cap R\ne\emptyset$. We know by the Lemmas \ref{lem:vert} and \ref{lemma:theta_gamma} that the image by $S$ of each horizontal line $y=y_0$ for $ y_0 \in(\alpha_0,\alpha_1)$  is folded over $\Gamma$; more precisely we have that $\tilde{R}=S(R)\cap R$ is bounded below by $\Gamma$ in the $y$-interval $(\alpha_0,\alpha_1)$ and bounded 
above by $\Gamma$ in the $y$-interval $(\alpha_1,\alpha_2)$, as qualitatively shown in Figure \ref{fig:grey}
(grey regions). The points in such grey region have, according to Lemma \ref{lem:inverse},  either one or two preimages in $R$. Differently,  the points belonging to the region $R \setminus \tilde{R}$ have no preimages in $R$ (white points in Figure \ref{fig:grey}).

Let $y_0$ be such that the line $y=y_0\in(\alpha_0,\alpha_1)$ intersects the set $U$. Its image in $R$ is a segment in $x=y_0$ given by $\varphi_{y_0}(x)$, which is folded on a unique point $\varphi_{y_0}(x^\star)\in \Gamma$ (see Lemma \ref{lem:vert}). Moreover, the image of $\{y=y_0\} \cap B$ must be outside the external boundary of $\mathcal A^{\star}\left(\alpha_1\right)$. So, according to the results in Section \ref{sec:interior} about the structure of the external boundary of $\mathcal A^{\star}\left(\alpha_1\right)$ (specifically Proposition \ref{prop:boundary_4cycle}) the image of $\{y=y_0\} \cap B$ must cross the external boundary of $\mathcal A^{\star}\left(\alpha_1\right)$ in a point of the arc $I$ connecting the focal points $Q_{0,1}$ and $Q_{1,0}$. Moreover as we have shown in Corollary \ref{coro:4cycle}, we have that $S^4(I)$ is folded on a subarc $I_1\subset I$ whose points have two inverses, at least one, say $w$, belonging to $\Lambda \subset \partial \mathcal A^{\star}\left(\alpha_1\right)$. The border point of $I_1$ is a point $r_{I}=\Gamma \cap I$. Thus, the arc $I$ splits in two subarcs: $I_1$ (with two preimages in $R$) and $I_2=I\setminus I_1$ (with no preimages in $R$), and $r_{I}=\overline{I}_1\cap \overline{I}_2$. We claim that the described configuration of images and preimages is a contradiction with Lemma \ref{lem:inverse}. To see the claim we observe that there must be two points $x_1$ and $x_2$ in $\{y=y_0\}\cap \partial B$ mapped to $\partial \mathcal A^{\star}\left(\alpha_1\right)$ at the same point $z\in I$. But $z$ is a point in $R$ with two preimages in $R$ while the configuration implies that $z$ has three preimages in $R$, $x_1,x_2$ and $w$, a contradiction. 

If $y_0\in(\alpha_1,\alpha_2)$ is such that the horizontal line $y=y_0$ intersects the set $U$ the arguments follows similarly. First we notice that, as before, $x^\star\left(y_0\right)$ must belong to the segment $B$ intersecting $U$. Moreover the image of $\{y=y_0\} \cap B$ must be outside the external boundary of $\mathcal A^{\star}\left(\alpha_1\right)$ and the image of $\{y=y_0\} \cap B$ must cross the external boundary of $\mathcal A^{\star}\left(\alpha_1\right)$ in a point of the arc $K$ connecting the focal points $Q_{1,2}$ and $Q_{2,1}$. The arc $K$ splits in two subarcs $K_1$ and $K_2$ having, respectively, two (at least one, say $w^{\prime}$, belonging to $J \subset \partial \mathcal A^{\star}\left(\alpha_1\right)$) and none preimages in $R$. The border point of $K_1$ is a point $r_{K}=\Gamma\cap K=\overline{K}_1\cap \overline{K}_2$. 
As before the described configuration of images and preimages is a contradiction with Lemma \ref{lem:inverse} since it creates a point $z^{\prime} \in K_1$ with three preimages.

 \begin{figure}[ht]
    \centering
     \includegraphics[width=0.25\textwidth]{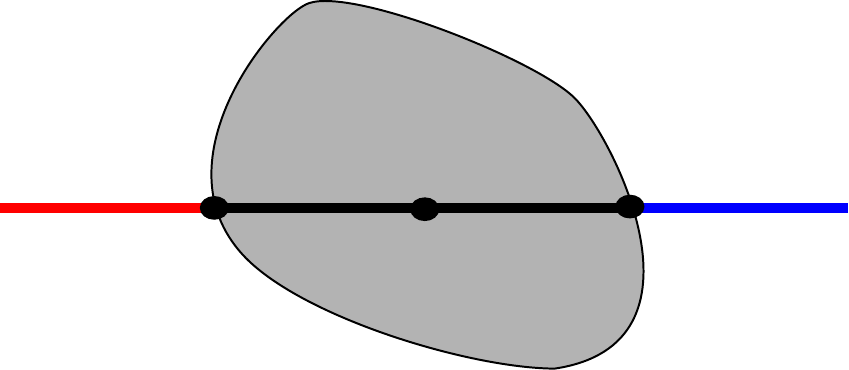}
          \put(-129,18){\small $H_{y_0}$}
            \put(-58,24){\small $x^\star $}
            \put(-93,25){\small $x_1 $}
            \put(-29,25){\small $x_2$}
             \put(-97,13){\tiny $A$}
             \put(-45,13){\tiny $B$}
             \put(-23,13){\tiny $C$}
           \put(-57,51){\small $U$}
                    \caption{\small{Sketch of the phase plane assuming that $\mathcal A^*(\alpha_1)$ is multiply connected.}}
    \label{fig:U}
    \end{figure}

\begin{figure}[ht]
    \centering
     \includegraphics[width=0.65\textwidth]{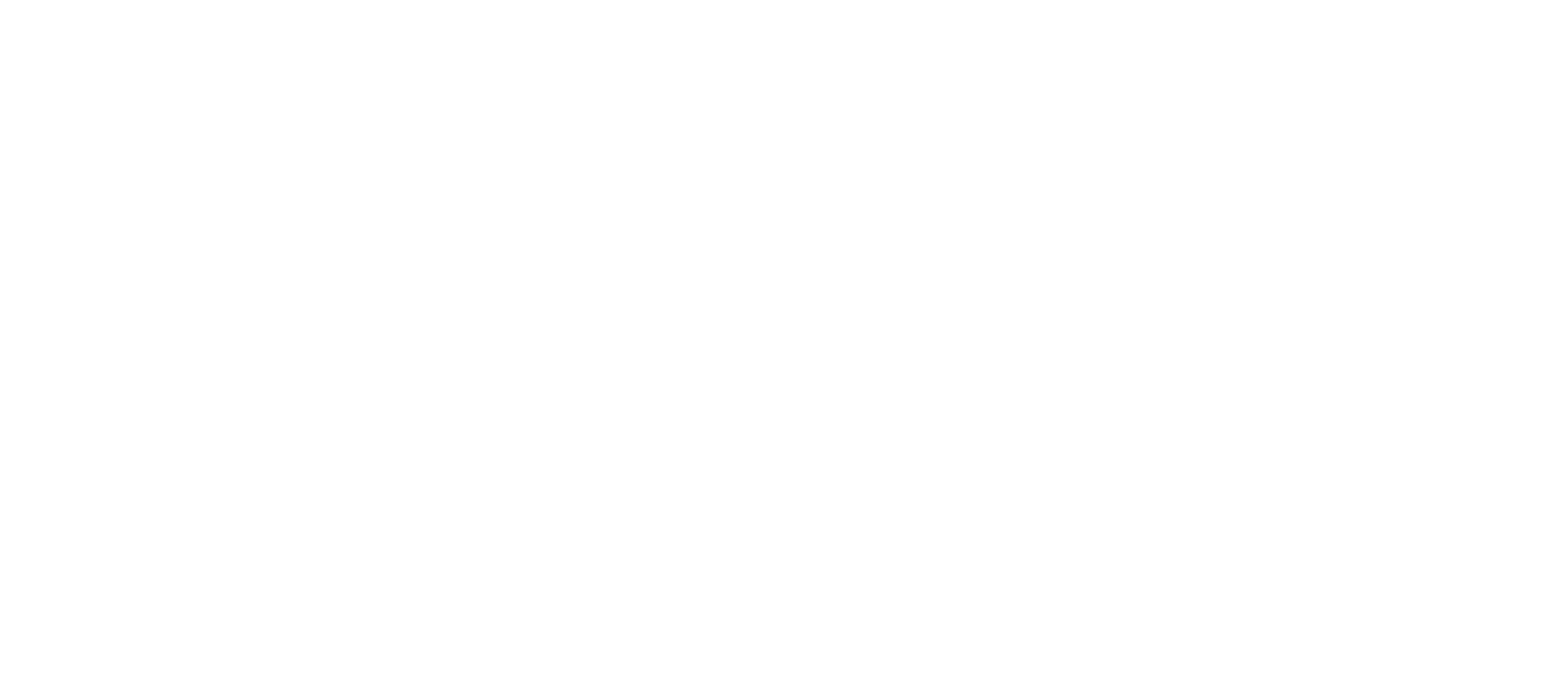}
          \put(-48,66){\small $\Gamma$}
           \put(-47,13){\small $\Theta$}
            \put(-135,111){\small $R$}
            \put(-117,61){\footnotesize $(\alpha_1,\xi)$}
            \put(-69,37){\footnotesize $(\xi,\alpha_1)$}
              \put(-230,-9){\small (a)}
            \put(-261,51){\small $\Gamma$}
             \put(-280,111){\small $R$}
             \put(-240,111){\small $\Theta$}
              \put(-264,78){\footnotesize $(\xi,\alpha_1)$}
              \put(-206,37){\footnotesize $(\alpha_1,\xi)$}
                            \put(-80,-9){\small (b)}
                    \caption{\small{Qualitative picture of the image of $S(R)$ and the set $\tilde{R}=S(R) \cap R$ in grey. 
                    Figure (a) corresponds to $\gamma_0<\alpha_1$  and Figure (b) corresponds to $\gamma_0>\alpha_1$.}}
    \label{fig:grey}
    \end{figure}
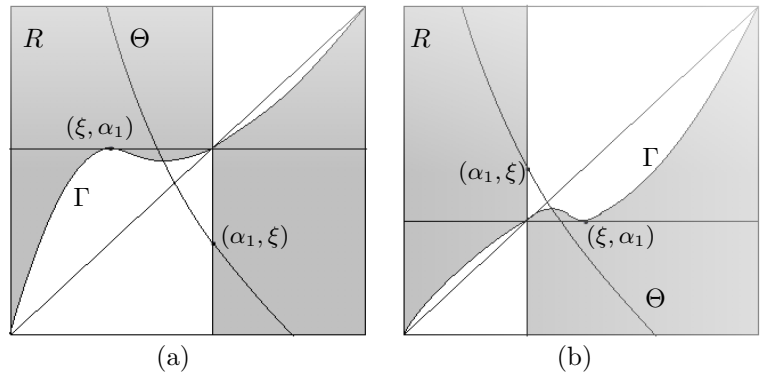
\end{proof}

As before, we exemplify the above arguments with an example. We  consider again the secant map applied to the polynomial $p(x)= \frac{x5}{5}-\frac{x^3}{3}-0.05x+0.15$ (see Figure \ref{fig:multiply_connected}). The polynomial $p$ exhibits three simple real roots  $ \alpha_0 \simeq  -1.43014$ , $ \alpha_1 \simeq    0.817633$  and  $ \alpha_2 \simeq  1.17823$  and the immediate basin of attraction of the internal root $\alpha_1$ is not simply connected. 
In Figure \ref{fig:not_simply} we show the phase space of the secant map (left) and the graph of the map $\varphi$ given by
\[
\varphi(x)= y_0 - \frac{p(y_0)} {q(x,y_0)} = y_0 - \frac{\frac{y_0^5}{5}-\frac{y_0^3}{3}-0.05y_0+0.15}{\frac{1}{5} (y_0^4 + y_0^3 x + y_0^2 x^2 + y_0 x^3 + x^4) - \frac{1}{3} (y_0^2+ x y_0 +x^2) -0.05  } 
\]

In this case when $x$ moves in $H_{y_0}$ from $\alpha_0$  to $\alpha_2$ the graph of $\varphi$ exhibits two local minimum and one local maximum, and thus the secant map could map outside $\mathcal A^*(\alpha_1)$.
  
 \begin{figure}[ht]
    \centering
    \subfigure[\scriptsize{ Phase space of the secant map of the polynomial $p(x)=x^5/5-x^3/3-0.05x+0.15$. Range of the picture $[-2,2]\times[-2,2]$.}]{
     \includegraphics[width=0.42\textwidth]{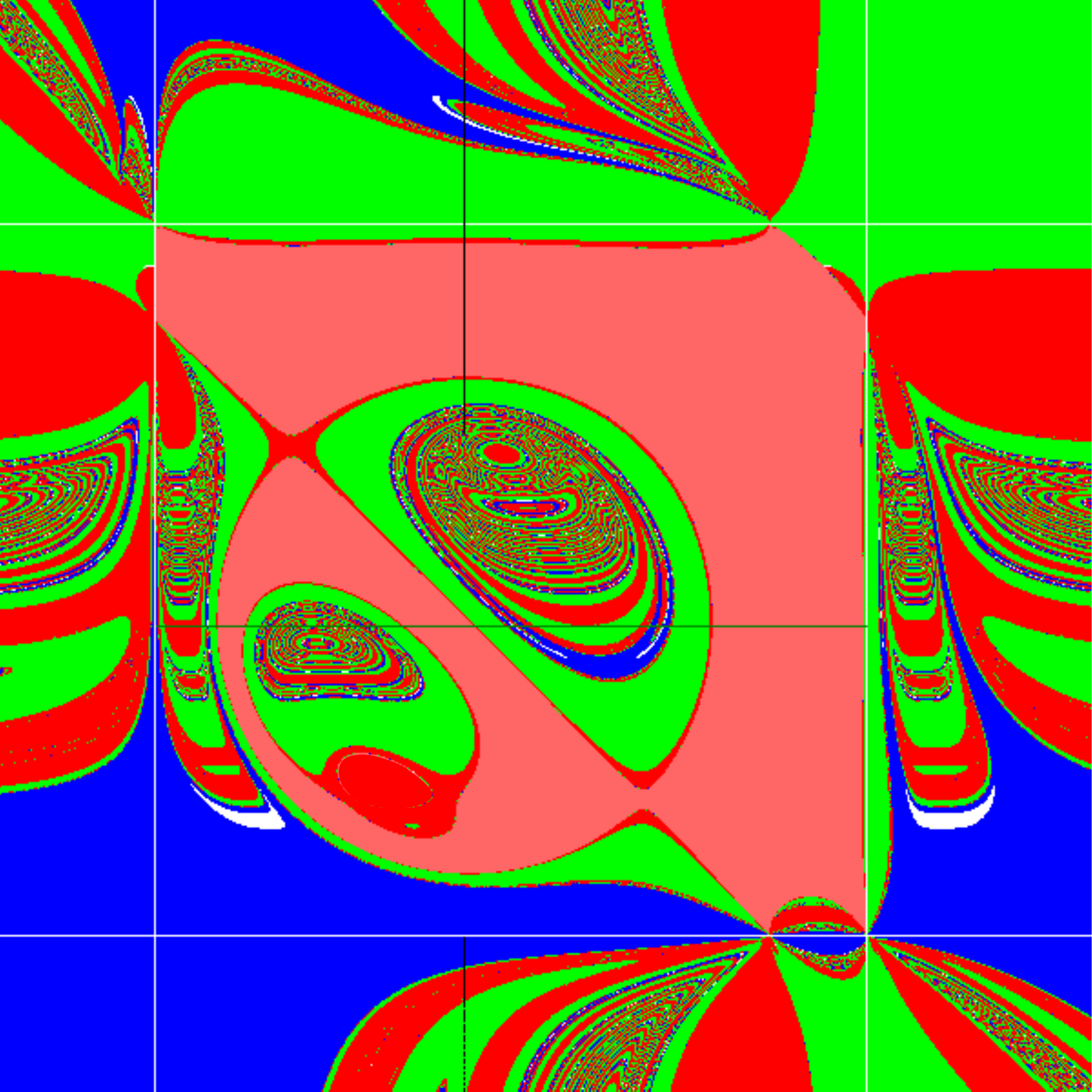}}
          \put(-62,83){\tiny $H_{y_0}$}
          \put(-105,130){\tiny $S(H_{y_0})$}
            \subfigure[\scriptsize{Graph of $\varphi(x) = y_0 - p(y_0)/ q(x,y_0)$.}]{
     \includegraphics[width=0.49\textwidth]{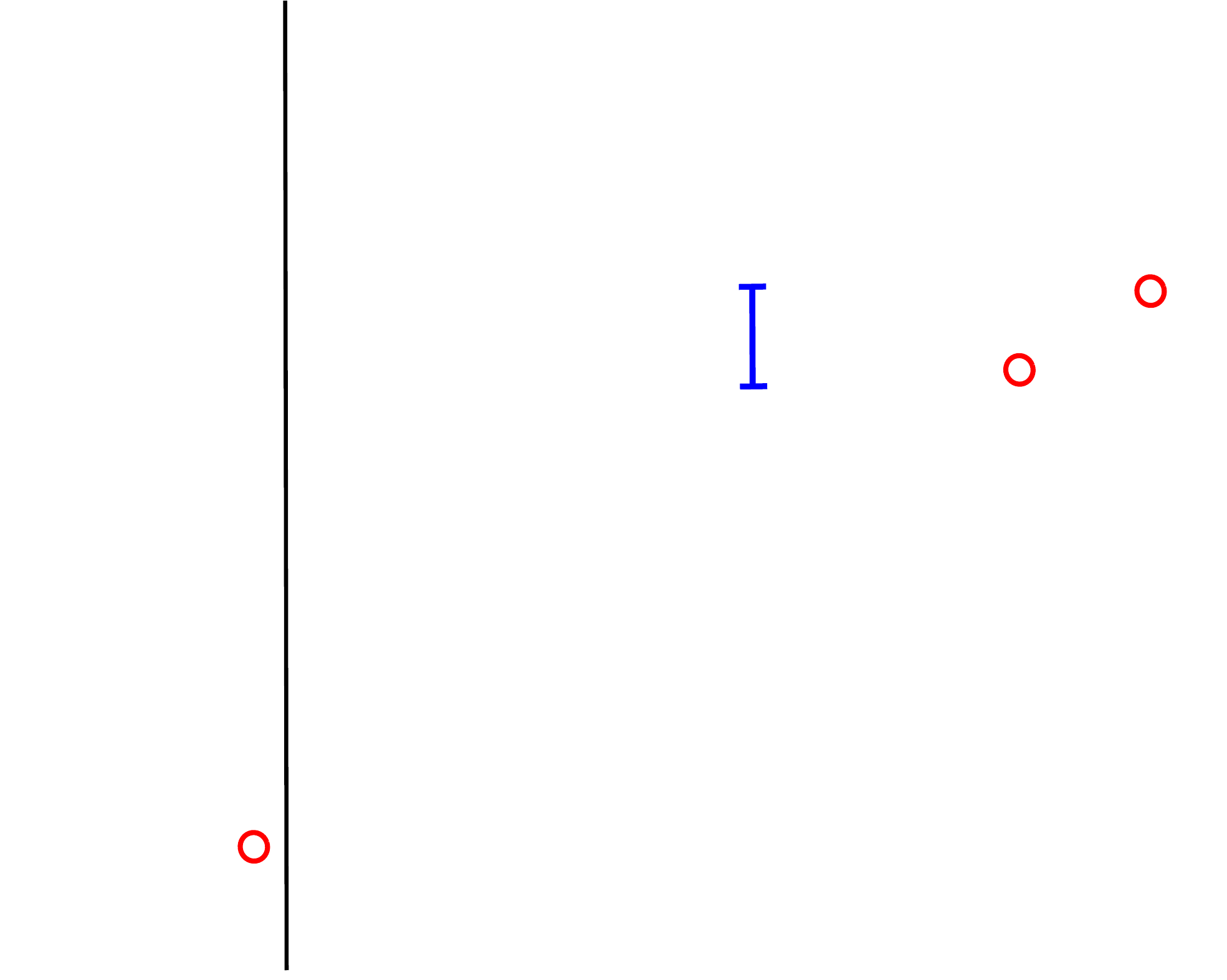}} 
                       \put( -82,112){\small {\b  $J_{y_0}$}} 
                        \caption{\small{The secant map applied to a degree five polynomial for which the (only) internal immediate basin is multiply connected. In (a) it is shown the line $H_{y_0}$ with $y=y_0$ in the interval $\left(\alpha_0,\alpha_1\right)$. In (b) we show the graph of the related function $\varphi_{y_0}(x)$ when $y_0\in \left(\alpha_0,\alpha_1\right)$. In particular we illustrate that $\varphi$ exhibits three critical points, due to the existence of more than one change of convexity of $p$ in the relevant interval.}}
    \label{fig:not_simply}
    \end{figure}

\begin{remark}
From the proof of the above proposition we conclude that a simply connected immediate basin of attraction of an internal root $\mathcal A^*(\alpha_1)$ is forward invariant, i.e. $S(\mathcal A^*(\alpha_1)) \subset \mathcal A^*(\alpha_1)$. This is due to the fact that no point of $\mathcal A^*(\alpha_1)$ can be mapped outside the set $\mathcal A^*(\alpha_1)$.
\end{remark}

In Corollary \ref{cor:simple_roots} we collect the main results of this paper. Assuming that $p$ is a polynomial of degree $k$ with $k$ simple roots and only one inflexion point between any three consecutive roots we conclude, by Theorems A  and B, that 
the immediate basin of attraction of an internal root is simply connected and the boundary is controlled by a 4-cycle of type I of the secant map.

We finally mention the case of the external roots, i.e., $\alpha_0$ and $\alpha_{n-1}$, of the polynomial $p$.  In that case a similar approach could be done as in the case of the internal ones. However, several difficulties appear. The first one is that the immediate basin of attraction of an external root is unbounded and points in the set of no definition of $\delta_S$ belong to this immediate basin. The second difficulty is that depending on the oddity of the degree of $p$ the boundary of $\mathcal A^*(\alpha_0)$  contains  a critical three cycle $(c,c)\to (c,\infty)\to (\infty,c) \to (c,c)$ where $p'(c)=0$ (degree of $p$ even)  or  a 4-cycle  (degree of $p$ odd).

    \bibliographystyle{alpha}
\bibliography{BoundariesSecant}
\end{document}